\documentclass[11pt,a4paper]{amsart} 		 
\usepackage[english]{babel}
\usepackage[utf8]{luainputenc}
\usepackage[T1]{fontenc}
\usepackage{lmodern}
\usepackage{amsmath}
\usepackage{amsfonts}
\usepackage{amsthm}
\usepackage{amsmath}
\usepackage{amssymb}
\usepackage{xr-hyper}
\usepackage{mathtools}
\usepackage{stmaryrd}
\usepackage{graphicx}
\usepackage[all]{xy}
\usepackage{a4wide}
\usepackage{enumitem}
\usepackage[mathcal]{euscript}
\usepackage{mathrsfs}
\usepackage{leftidx}
\usepackage{color} 
\usepackage[colorlinks=true,linkcolor=black,citecolor=black,urlcolor=black]{hyperref}
\usepackage[automake=true,acronyms,toc]{glossaries} 

	\newglossaryentry{mapctsxy}
{
    name=\ensuremath{\mathrm{Map_{cts}}(X,Y)},
    description={Set of continuous maps between the topological spaces $X$ and $Y$}, 
    sort=Mapcts
}

		\newglossaryentry{nn}
	{
	    name=\ensuremath{\mathbb{N}},
	    description={Natural numbers, starting with $1$}, 
	    sort=Na
	}
	
		\newglossaryentry{nnn}
	{
	    name=\ensuremath{\mathbb{N}_0},
	    description={Natural numbers, starting with $0$, i.e. $\NN_0=\NN \cup \{0\}$}, 
	    sort=Naa
	}

		\newglossaryentry{mku}
	{
	    name=\ensuremath{\mathrm{M}(K,U)},
	    description={Set of all $f \in \mapcts(X,Y)$ with $f(K) \subseteq U$, where $K \subseteq X$ compact
	    and $U \subseteq Y$ open, i.e. basis open sets for the compact open topology on $\mapcts(X,Y)$}, 
	    sort=Mapku
	}

		\newglossaryentry{homrmn}
	{
	    name=\ensuremath{\mathrm{Hom}_R(M,N)},
	    description={R-linear Homomorphisms between the $R$-modules $M$ and $N$}, 
	    sort=Homrmn,
	    symbol=\ensuremath{\Hom_{\ZZ[M]}(\ZZ,A)}
	}

		\newglossaryentry{DDM}
	{
	    name=\ensuremath{\mathcal{DIS}_{M}},
	    description={Category of discrete abelian groups with a continuous action from a topological monoid $M$
	    together with continuous group homomorphisms respecting the action from $M$}, 
	    sort=DDM
	}
	
		\newglossaryentry{DDG}
	{
	    name=\ensuremath{\mathcal{DIS}_{G}},
	    description={Category of discrete abelian groups with a continuous action from a profinite group $G$
	    together with continuous group homomorphisms respecting the action from $G$}, 
	    sort=DDG
	}
	
		\newglossaryentry{DDGM}
	{
	    name=\ensuremath{\mathcal{DIS}_{G,M}},
	    description={Category of discrete abelian groups with commuting continuous actions from a profinite group $G$
	    and a topological monoid $M$ 
	    together with continuous group homomorphisms respecting the actions from $G$ and $M$}, 
	    sort=DDGM
	}
	
		\newglossaryentry{AbsM}
	{
	    name=\ensuremath{\mathcal{ABS}_{M}},
	    description={Category of (abstract) abelian groups with an action from a topological monoid $M$
	    together with group homomorphisms respecting the action from $M$}, 
	    sort=ABSM
	}
	
		\newglossaryentry{AbsG}
	{
	    name=\ensuremath{\mathcal{ABS}_{G}},
	    description={Category of (abstract) abelian groups with an action from a profinite group $G$
	    together with group homomorphisms respecting the action from $G$}, 
	    sort=ABSG
	}
	
		\newglossaryentry{AbsGM}
	{
	    name=\ensuremath{\mathcal{ABS}_{GM}},
	    description={Category of (abstract) abelian groups with commuting actions from a profinite group $G$
	    and a topological monoid $M$ 
	    together with group homomorphisms respecting the actions from $G$ and $M$}, 
	    sort=ABSGM
	}

		\newglossaryentry{TopG}
	{
	    name=\ensuremath{\mathcal{TOP}_{G}},
	    description={Category of topological abelian Hausdorff groups with a continuous action from a 
	    profinite group $G$
	    together with continuous group homomorphisms respecting the action from $G$}, 
	    sort=TOPG
	}
	
		\newglossaryentry{TopGM}
	{
	    name=\ensuremath{\mathcal{TOP}_{G,M}},
	    description={Category of topological abelian Hausdorff groups with commuting continuous actions f
	    rom a profinite group $G$ and a topological monoid $M$ 
	    together with continuous group homomorphisms respecting the actions from $G$ and $M$}, 
	    sort=TOPGM
	}

		\newglossaryentry{adel}
	{
	    name=\ensuremath{A^\delta},
	    description={$=\bigcup_{U \leq G \mathrm{\ open}} A^{U},$
	    for an abelian group $A$ with an action from a profinite group $G$}, 
	    sort=A
	}

		\newglossaryentry{ker}
	{
	    name=\ensuremath{\mathrm{ker}(f)},
	    description={Kernel of the homomorphism $f$}, 
	    sort=Ker
	}

		\newglossaryentry{im}
	{
	    name=\ensuremath{\mathrm{im}(f)},
	    description={Image of the homomorphism $f$}, 
	    sort=Im
	}

		\newglossaryentry{coker}
	{
	    name=\ensuremath{\mathrm{coker}(f)},
	    description={Cokernel of the homomorphism $f$}, 
	    sort=Coker
	}

		\newglossaryentry{Abcat}
	{
	    name=\ensuremath{\mathbf{Ab}},
	    description={Category of abelian groups}, 
	    sort=Ab
	}

		\newglossaryentry{totabb}
	{
	    name=\ensuremath{\mathrm{Tot}(A^{\bullet,\bullet})},
	    description={Total complex of the double complex $A^{\bullet,\bullet}$}, 
	    sort=Tot
	}
	
		\newglossaryentry{totnabb}
	{
	    name=\ensuremath{\mathrm{Tot}^n(A^{\bullet,\bullet})},
	    description={$n$-th object of the total complex $\mathrm{Tot}(A^{\bullet,\bullet})$ of the 
	    double complex $A^{\bullet,\bullet}$}, 
	    sort=Tot
	}
	
		\newglossaryentry{difnabb}
	{
	    name=\ensuremath{\mathrm{d}_{\mathrm{Tot}(A^{\bullet,\bullet})}^n},,
	    description={$n$-th differential of the total complex $\mathrm{Tot}(A^{\bullet,\bullet})$ of the 
	    double complex $A^{\bullet,\bullet}$}, 
	    sort=Dif
	}
	
		\newglossaryentry{Ccb}
	{
	    name=\ensuremath{\mathcal{C}^{\bullet}(G,A)},
	    description={Standard complex for continuous cochain cohomology of the profinite group $G$ with
	    coefficients in $A$, cf. \cite[Chapter II \S 7, p.\,136]{nswo}}, 
	    sort=CCB
	}
	
		\newglossaryentry{CcbX}
	{
	    name=\ensuremath{\mathcal{C}^{\bullet}_X(G,A)},
	    description={Total complex of the double complex 
	    $\mathcal{C}^{\bullet}(G,A) \overset{X-1}{\longrightarrow} \mathcal{C}^{\bullet}(G,A)$}, 
	    sort=CCBX
	}
	
		\newglossaryentry{Ccbf}
	{
	    name=\ensuremath{\mathcal{C}^{\bullet}_f(G,A)},
	    description={$=\mathcal{C}^{\bullet}_X(G,A)$ where the $\mathbb{N}_0$-action on $A$ comes from
	    an endomorphism $f$ of $A$.}, 
	    sort=CCBXf
	}

		\newglossaryentry{HXGA}
	{
	    name=\ensuremath{\mathcal{H}_X^*(G,A)},
	    description={Cohomology of the complex $\mathcal{C}^{\bullet}_X(G,A)$}, 
	    sort=HXGA
	}
	
		\newglossaryentry{HXGAf}
	{
	    name=\ensuremath{\mathcal{H}_f^*(G,A)},
	    description={Cohomology of the complex $\mathcal{C}^{\bullet}_f(G,A)$}, 
	    sort=HXGAf
	}

		\newglossaryentry{Qp}
	{
	    name=\ensuremath{\mathbb{Q}_p},
	    description={Field of $p$-adic numbers}, 
	    sort=Qp
	}
	
		\newglossaryentry{Zp}
	{
	    name=\ensuremath{\mathbb{Z}_p},
	    description={Integral $p$-adic numbers}, 
	    sort=Zp
	}
	
		\newglossaryentry{Qpalg}
	{
	    name=\ensuremath{\overline{\mathbb{Q}_p}},
	    description={Fixed algebraic closure of $\mathbb{Q}_p$}, 
	    sort=Qpa
	}

		\newglossaryentry{Cp}
	{
	    name=\ensuremath{\mathbb{C}_p},
	    description={Completion of $\overline{\mathbb{Q}_p}$ with respect to $v_p$ with $v_p(p)=1$}, 
	    sort=Cp
	}

		\newglossaryentry{OK}
	{
	    name=\ensuremath{\mathcal{O_K}},
	    description={Ring of integers of the extension $\mathcal{K}|\mathbb{Q}_p$}, 
	    sort=OK,
	    symbol=\ensuremath{\mathcal{O}_L}
	}
	
		\newglossaryentry{piK}
	{
	    name=\ensuremath{\pi_\mathcal{K}},
	    description={Prime element of the finite extension $\mathcal{K}|\mathbb{Q}_p$}, 
	    sort=PiK,
	    symbol=\ensuremath{\pi_L}
	}

		\newglossaryentry{kK}
	{
	    name=\ensuremath{k_\mathcal{K}},
	    description={Residue class field of the finite extension $\mathcal{K}|\mathbb{Q}_p$}, 
	    sort=kK,
	    symbol=\ensuremath{k_L}
	}
	
		\newglossaryentry{Kur}
	{
	    name=\ensuremath{\mathcal{K}^{\mathrm{ur}}},
	    description={Maximal unramified extension of $\mathbb{Q}_p$ inside $\mathcal{K}$}, 
	    sort=Kur,
	    symbol=\ensuremath{L^{\mathrm{ur}}}
	}

		\newglossaryentry{GK}
	{
	    name=\ensuremath{G_{\mathcal{K}}},
	    description={Absolute Galois Group of $\mathcal{K}|\mathbb{Q}_p$}, 
	    sort=GK,
	    symbol=\ensuremath{G_{\mathbb{Q}_p}}
	}

		\newglossaryentry{WL}
	{
	    name=\ensuremath{W(\cdot)_L},
	    description={Ramified Witt vectors over $L$ with $L|\mathbb{Q}_p$ finite}, 
	    sort=WL
	}

		\newglossaryentry{Kb}
	{
	    name=\ensuremath{\mathcal{K}^{\flat}},
	    description={Tilt of the perfectoid field $\mathcal{K}\subseteq \mathbb{C}_p$}, 
	    sort=Kb
	}	
	
		\newglossaryentry{RX}
	{
	    name=\ensuremath{R \llbracket X_1,\dots X_n \rrbracket},
	    description={Power series ring with variables $X_1, \dots X_n$ and with coefficients in the ring $R$}, 
	    sort=RX,
	    symbol=\ensuremath{\mathcal{O} \llbracket X \rrbracket}
	}

		\newglossaryentry{GGphi}
	{
	    name=\ensuremath{\mathcal{G}_{\phi}},
	    description={Lubin-Tate formal group associated to the Frobenius power series $\phi$}, 
	    sort=Gphi
	}

		\newglossaryentry{endphi}
	{
	    name=\ensuremath{[a]_{\phi}},
	    description={Endomorphism of $\mathcal{G}_{\phi}$ associated to $a \in \mathcal{O}_L$}, 
	    sort=Gphiend
	}

		\newglossaryentry{GGnphi}
	{
	    name=\ensuremath{\mathcal{G}_{\phi,n}},
	    description={$=\mathrm{ker}([\pi_L^n]_\phi\colon \mathfrak{M} \to \mathfrak{M}) = 
	    \{ x \in \mathfrak{M} \mid [\pi_L^n]_\phi(x)=0\}$}, 
	    sort=Gphia
	}

		\newglossaryentry{TGphi}
	{
	    name=\ensuremath{\mathcal{TG}_{\phi}},
	    description={Tate module of $\mathcal{G}_{\phi}$}, 
	    sort=TGphi
	}

		\newglossaryentry{Ln}
	{
	    name=\ensuremath{L_n},
	    description={$=L(\mathcal{G}_{\phi}[\pi_L^n])$}, 
	    sort=Ln
	}

		\newglossaryentry{Lin}
	{
	    name=\ensuremath{L_{\infty}},
	    description={$=\cup_n L_n$}, 
	    sort=Lna
	}

		\newglossaryentry{Kn}
	{
	    name=\ensuremath{K_n},
	    description={$=K(\mathcal{G}_{\phi}[\pi_L^n])=KL_n$}, 
	    sort=Kn
	}

		\newglossaryentry{Kin}
	{
	    name=\ensuremath{K_{\infty}},
	    description={$=\cup_n K_n$}, 
	    sort=Kna
	}

		\newglossaryentry{GaL}
	{
	    name=\ensuremath{\Gamma_L},
	    description={$=\mathrm{Gal}(L_{\infty}|L)$}, 
	    sort=GaL
	}

		\newglossaryentry{GaK}
	{
	    name=\ensuremath{\Gamma_K},
	    description={$=\mathrm{Gal}(K_{\infty}|K)$}, 
	    sort=GaL
	}

		\newglossaryentry{HL}
	{
	    name=\ensuremath{H_L},
	    description={$=\mathrm{Gal}(\overline{\mathbb{Q}_p}|L_{\infty})$}, 
	    sort=GaL
	}

		\newglossaryentry{HK}
	{
	    name=\ensuremath{H_K},
	    description={$=\mathrm{Gal}(\overline{\mathbb{Q}_p}|K_{\infty})$}, 
	    sort=GaL
	}

		\newglossaryentry{RXl}
	{
	    name=\ensuremath{R((X_1,\dots,X_n))},
	    description={Ring of Laurent series in the variables $X_1,\dots,X_n$ with coefficients in $R$}, 
	    sort=RXl,
	    symbol=\ensuremath{\mathcal{O}_L((X))}
	}

		\newglossaryentry{AL}
	{
	    name=\ensuremath{\mathscr{A}_L},
	    description={$=\varprojlim_n \mathcal{O}_L/\pi^n_L \mathcal{O}_L ((X))$; in 
	    \hyperref[cwh]{Chapter \ref*{cwh}} we use the variable $Z$}, 
	    sort=AL
	}

		\newglossaryentry{EEL}
	{
	    name=\ensuremath{\mathbf{E}_L},
	    description={Image of the inclusion $k_L((X)) \hookrightarrow \mathbb{C}_p^{\flat}, \ X \mapsto \omega$}, 
	    sort=EEL
	}

		\newglossaryentry{EELp}
	{
	    name=\ensuremath{\mathbf{E}_L^+},
	    description={Ring of integers of $\mathbf{E}_L$}, 
	    sort=EELp
	}
	
		\newglossaryentry{EELsep}
	{
	    name=\ensuremath{\mathbf{E}_L^{\mathrm{sep}}},
	    description={Separable closure of $\mathbf{E}_L$ inside $\mathbb{C}_p^{\flat}$}, 
	    sort=EELsep
	}
	
		\newglossaryentry{EELsepp}
	{
	    name=\ensuremath{\mathbf{E}_L^{\mathrm{sep},+}},
	    description={Integral closure of $\mathbf{E}_L^+$ inside $\mathbf{E}_L^{\mathrm{sep}}$}, 
	    sort=EELsepp
	}
	
		\newglossaryentry{omega}
	{
	    name=\ensuremath{\omega},
	    description={Uniformizer of $\mathbf{E}_L$}, 
	    sort=omega
	}
	
		\newglossaryentry{omphi}
	{
	    name=\ensuremath{\omega_{\phi}},
	    description={Lift of $\omega \in \mathbf{E}_L$ to $W(\mathbf{E}_L)_L$}, 
	    sort=omegaphi
	}

		\newglossaryentry{Fr}
	{
	    name=\ensuremath{\mathrm{Fr}},
	    description={Frobenius on $W(\mathbb{C}_p^{\flat})_L$}, 
	    sort=Frobenius
	}

		\newglossaryentry{AAL}
	{
	    name=\ensuremath{\mathbf{A}_L},
	    description={Image of the inclusion $\mathscr{A}_L \hookrightarrow W(\mathbf{E}_L)_L, \ 
	    X \mapsto \omega_{\phi}$}, 
	    sort=ALbf
	}

		\newglossaryentry{AALp}
	{
	    name=\ensuremath{\mathbf{A}_L^+},
	    description={$=\mathcal{O}_L \llbracket \omega_{\phi} \rrbracket$}, 
	    sort=ALbfp
	}
	
		\newglossaryentry{AALnr}
	{
	    name=\ensuremath{\mathbf{A}_L^{\mathrm{nr}}},
	    description={maximal unramified extension of $\mathbf{A}_L$ inside $W(\mathbf{E}_L^{\mathrm{sep}})_L$}, 
	    sort=ALbfpnr
	}

		\newglossaryentry{AALnrp}
	{
	    name=\ensuremath{\mathbf{A}_L^{\mathrm{nr},+}},
	    description={$=\mathbf{A}_L^{\mathrm{nr}} \cap W(\mathbf{E}_L^{\mathrm{sep},+})_L $}, 
	    sort=ALbfpnrp
	}

		\newglossaryentry{AAA}
	{
	    name=\ensuremath{\mathbf{A}},
	    description={$=\varprojlim_n \mathbf{A}_L^{\mathrm{nr}}/\pi_L^n\mathbf{A}_L^{\mathrm{nr}}$}, 
	    sort=ALc
	}
	
		\newglossaryentry{AAAp}
	{
	    name=\ensuremath{\mathbf{A}^+},
	    description={$=\mathbf{A} \cap W(\mathbf{E}_L^{\mathrm{sep},+})_L $}, 
	    sort=ALcp
	}

		\newglossaryentry{AAK}
	{
	    name=\ensuremath{\mathbf{A}_{\mathcal{K}|L}},
	    description={$=(\mathbf{A})^{H_{\mathcal{K}|L}}$ for $\mathcal{K}|L$ finite}, 
	    sort=ALbfz,
	    symbol=\ensuremath{\mathbf{A}_K}
	}

		\newglossaryentry{AAKp}
	{
	    name=\ensuremath{\mathbf{A}_{\mathcal{K}|L}^+},
	    description={$=\mathbf{A}_{\mathcal{K}|L}\cap W(\mathbf{E}_L^{\mathrm{sep},+})_L$ for $\mathcal{K}|L$ finite}, 
	    sort=ALbfzp,
	    symbol=\ensuremath{\mathbf{A}_K^+}
	}

		\newglossaryentry{phKL}
	{
	    name=\ensuremath{\varphi_{K|L}},
	    description={Restriction from the Frobenius of $W(\mathbf{E}_L^{\mathrm{sep}})_L$ to
	    $\mathbf{A}_K$}, 
	    sort=phizKL
	}	

		\newglossaryentry{BBL}
	{
	    name=\ensuremath{\mathbf{B}_{L}},
	    description={Quotient field of $\mathbf{A}_{L}$}, 
	    sort=BLbf
	}
	
		\newglossaryentry{BBK}
	{
	    name=\ensuremath{\mathbf{B}_{\mathcal{K}|L}},
	    description={Quotient field of $\mathbf{A}_{\mathcal{K}|L}$ for $\mathcal{K}|L$ finite}, 
	    sort=BLKbf,
	    symbol=\ensuremath{\mathbf{B}_{K|L}}
	}

		\newglossaryentry{BBLnr}
	{
	    name=\ensuremath{\mathbf{B}_{L}^{\mathrm{nr}}},
	    description={Quotient field of $\mathbf{A}_L^{\mathrm{nr}}$}, 
	    sort=BLbfnr
	}

		\newglossaryentry{BBB}
	{
	    name=\ensuremath{\mathbf{B}},
	    description={Quotient field of $\mathbf{A}$}, 
	    sort=BLbfz
	}

		\newglossaryentry{EEK}
	{
	    name=\ensuremath{\mathbf{E}_{\mathcal{K}|L}},
	    description={$=(\mathbf{E}_L^{\mathrm{sep}})^{H_{\mathcal{K}}}$ for $\mathcal{K}|L$ finite}, 
	    sort=EELK,
	    symbol=\ensuremath{\mathbf{E}_{K|L}}
	}
	
		\newglossaryentry{EEKp}
	{
	    name=\ensuremath{\mathbf{E}_{\mathcal{K}|L}^+},
	    description={Integral closure of $\mathbf{E}_{L}^+$ inside $\mathbf{E}_{\mathcal{K}|L}$ for $\mathcal{K}|L$ finite}, 
	    sort=EELKp,
	    symbol=\ensuremath{\mathbf{E}_{K|L}^+}
	}

		\newglossaryentry{tenphaaK}
	{
	    name=\ensuremath{\ _{f}\otimes_{R}},
	    description={Tensor product over the ring $R$, where $f$ is an endomorphism of $R$ and the module on
	    the left is considered as right $R$-module via $f$ while the module on the right has its usual left-operation
	    from $R$}, 
	    sort=Ten,
	    symbol=\ensuremath{\ _{\varphi_{K|L}} \otimes_{\mathbf{A}_{K|L}}}
	}

		\newglossaryentry{phMlin}
	{
	    name=\ensuremath{f_M^{\mathrm{lin}}},
	    description={Linearization of the $f$-linear endomorphism $f_M$ of $M$, where $M$ is
	    an $R$-module and $f$ is an endomorphism of $R$, i.e. the homomorphism
	    \[ f_M^{\mathrm{lin}}\colon R \ _{f}\otimes_R M \to M, \ a \otimes m \mapsto a f_M(m)\]}, 
	    sort=flin,
	    symbol=\ensuremath{\varphi_M^{\mathrm{lin}}}
	}	

		\newglossaryentry{ModpGKe}
	{
	    name=\ensuremath{\mathbf{Mod}_{\varphi,\Gamma}^{\mathrm{\acute{e}t}}(\mathbf{A}_{\mathcal{K}|L})},
	    description={Category of \'{e}tale $(\varphi_{\mathcal{K}|L},\Gamma_{\mathcal{K}})$-modules
	    over $\mathbf{A}_{\mathcal{K}|L}$, where $\mathcal{K}|L$ is finite}, 
	    sort=modpgk,
	    symbol=\ensuremath{\mathbf{Mod}_{\varphi,\Gamma}^{\mathrm{\acute{e}t}}(\mathbf{A}_{{K|L}})}
	}

		\newglossaryentry{RepGKOLfg}
	{
	    name=\ensuremath{\mathbf{Rep}_{\mathcal{O}_L}^{\mathrm{(fg)}}(G)},
	    description={Category of finitely generated $G$-representations of $\mathcal{O}_L$, where $G$
	    is a group}, 
	    sort=RepGK,
	    symbol=\ensuremath{\mathbf{Rep}_{\mathcal{O}_L}^{\mathrm{(fg)}}(G_K)}
	}

		\newglossaryentry{VK}
	{
	    name=\ensuremath{\mathcal{V_{K|L}}},
	    description={Functor from $\mathbf{Mod}_{\varphi,\Gamma}^{\mathrm{\acute{e}t}}(\mathbf{A}_{\mathcal{K}|L})$
	    to $\mathbf{Rep}_{\mathcal{O}_L}^{\mathrm{(fg)}}(G_{\mathcal{K}})$ with
	   $\mathcal{V_K}(M)=\left( \mathbf{A} \otimes_{\mathbf{A}_{\mathcal{K}|L}} M)^{\mathrm{Fr} 
	   \otimes \varphi_M=1}\right)$; defining an equivalence with inverse $\mathcal{M_{K|L}}$ ($\mathcal{K}|L$ finite)}, 
	    sort=VK,
	    symbol=\ensuremath{\mathcal{V}_{\mathcal{K}|L}}
	}

		\newglossaryentry{MAK}
	{
	    name=\ensuremath{\mathcal{M}_{\mathcal{K}|L}},
	    description={Functor from $\mathbf{Rep}_{\mathcal{O}_L}^{\mathrm{(fg)}}(G_{\mathcal{K}})$ to 
	    $\mathbf{Mod}_{\varphi,\Gamma}^{\mathrm{\acute{e}t}}(\mathbf{A}_{\mathcal{K}|L})$ with
	    $\mathcal{M}_{\mathcal{K}|L}(V)=\left(\mathbf{A} \otimes_{\mathcal{O}_L} V \right)^{H_{\mathcal{K}}}$; defining an 
	    equivalence with inverse $\mathcal{V}_{\mathcal{K}|L}$ ($\mathcal{K}|L$ finite)}, 
	    sort=MK,
	    symbol=\ensuremath{\mathcal{M}_{K|L}}
	}

		\newglossaryentry{adV}
	{
	    name=\ensuremath{\mathrm{ad}_V},
	    description={The homomorphism
	    \[ \mathbf{A} \otimes_{\mathbf{A}_{\mathcal{K}}} \mathcal{M_K}(V) \to \mathbf{A} \otimes_{\mathcal{O}_L} V,
	    \ a \otimes m \mapsto am\]
	    for $V \in \mathbf{Rep}_{\mathcal{O}_L}^{\mathrm{(fg)}}(G_{\mathcal{K}})$
	    and $\mathcal{K}|L$ finite}, 
	    sort=adv
	}

		\newglossaryentry{adM}
	{
	    name=\ensuremath{\mathrm{ad}_M},
	    description={The homomorphism
	    \[ \mathbf{A} \otimes_{\mathcal{O}_L} \mathcal{V_K}(M) \to \mathbf{A} \otimes_{\mathbf{A}_{\mathcal{K}}} M,
	    \ a \otimes v \mapsto av\]
	    for $M \in \mathbf{Mod}_{\varphi,\Gamma}^{\mathrm{\acute{e}t}}(\mathbf{A}_{\mathcal{K}})$ 
	    and $\mathcal{K}|L$ finite}, 
	    sort=adm
	}

		\newglossaryentry{sigKL}
	{
	    name=\ensuremath{\sigma_{\mathcal{K}|L}},
	    description={Lift of the $q_L$-Frobenius of the residue class field extension to $\mathcal{K}$ for
	    $\mathcal{K}|L$ finite unramified}, 
	    sort=sigkl,
	    symbol=\ensuremath{\sigma_{K|L}}
	}

		\newglossaryentry{fthet}
	{
	    name=\ensuremath{f^{\vartheta}},
	    description={Applying $\vartheta$ to the coefficients of $f \in \mathbf{A}_{\mathcal{K}}$, where
	    $\mathcal{K}|L$ is finite unramified and $\vartheta$ is an $\mathcal{O}_L$-linear endomorphism
	    of $\mathcal{O_K}$, i.e.
	    \[ f^{\vartheta}=\sum \vartheta(a_i) \omega_{\phi}^{i}\]}, 
	    sort=ftheta
	}

		\newglossaryentry{chLT}
	{
	    name=\ensuremath{\chi_{\mathrm{LT}}},
	    description={Lubin-Tate character, giving the isomorphism 
	    $\chi_{\mathrm{LT}}\colon \Gamma_L \overset{\cong}{\longrightarrow} \mathcal{O}_L^{\times}$}, 
	    sort=chilt
	}

		\newglossaryentry{gLT}
	{
	    name=\ensuremath{g_{\mathrm{LT}}},
	    description={The inverse of $\frac{\partial(\mathcal{G}_{\phi}(X,Y))}{\partial Y}\vert_{(X,Y)=(0,Z)}$ in
	    $\mathcal{O}_L \llbracket Z \rrbracket$}, 
	    sort=glt
	}

		\newglossaryentry{logLT}
	{
	    name=\ensuremath{\mathrm{log_{LT}}},
	    description={Power series in $ZL\llbracket Z \rrbracket$ whose formal derivative is $g_{\mathrm{LT}}$}, 
	    sort=loglt
	}

		\newglossaryentry{dinv}
	{
	    name=\ensuremath{\partial_{\mathrm{inv}}},
	    description={Invariant derivation corresponding to $\mathrm{d} \mathrm{log_{LT}}$}, 
	    sort=dinv
	}

		\newglossaryentry{pscolt}
	{
	    name=\ensuremath{\widetilde{\psi_{\mathrm{Col}}}},
	    description={Unique endomorphism of $\mathcal{O}_K \llbracket Z \rrbracket$ with
	    \[ \left( \varpi_{K|L} \circ \widetilde{\psi_{\mathrm{Col}}} \right)(f(Z)) =
	  \sum_{a \in \mathcal{G}_{\phi}[\pi_L]} f(a +_{\mathcal{G}_\phi} Z) \]
	  for all $f \in \mathcal{O} \llbracket Z \rrbracket$  }, 
	    sort=pscolt
	}

		\newglossaryentry{Nort}
	{
	    name=\ensuremath{\widetilde{\mathcal{N}}},
	    description={Unique multiplicative map from $\mathcal{O}_K \llbracket Z \rrbracket$ to itself with
	        \[ \left( \varpi_{K|L} \circ \widetilde{\mathcal{Nor}} \right)(f(Z)) =
	  \prod_{a \in \mathcal{G}_{\phi}[\pi_L]} f(a +_{\mathcal{G}_\phi} Z) \]
	  for all $f \in \mathcal{O} \llbracket Z \rrbracket$  }, 
	    sort=nort
	}
	
		\newglossaryentry{pscol}
	{
	    name=\ensuremath{\psi_{\mathrm{Col}}},
	    description={$=\sigma_{K|L}^{-1} \circ \widetilde{\psi_{\mathrm{Col}}}=\widetilde{\psi_{\mathrm{Col}}}
	    \circ \sigma_{K|L}^{-1}$ }, 
	    sort=pscolz
	}
	
		\newglossaryentry{Nor}
	{
	    name=\ensuremath{\mathcal{N}},
	    description={$= \sigma_{K|L}^{-1} \circ \widetilde{\mathcal{N}} = \widetilde{\mathcal{N}} \circ 
	    \sigma_{K|L}^{-1} $}, 
	    sort=norz
	}	

		\newglossaryentry{gut}
	{
	    name=\ensuremath{g_{u,t_0}},
	    description={Unique Laurent series in $(\mathcal{O}_K((Z))^{\times})^{\mathcal{N}=\mathrm{id}}$ to
	    $u=(u_n)_n \in \varprojlim_n K_n^{\times}$ with $\sigma_{K|L}^{-n}(g_{u,t_0}(t_{0,n})) = u_n$ were
	    $t_0 = (t_{0,n})_n$ is an $\mathcal{O}_L$-generator of $\mathcal{TG}_{\phi}$ }, 
	    sort=gut
	}

		\newglossaryentry{DLT}
	{
	    name=\ensuremath{\Delta_{\mathrm{LT}}},
	    description={Logarithmic homomorphism from $\mathcal{O}_K \llbracket Z \rrbracket^{\times}$ to
	    $\mathcal{O}_K \llbracket Z \rrbracket$ with $\Delta_{\mathrm{LT}}(f) = \frac{\partial_{\mathrm{inv}}(f)}{f}$}, 
	    sort=Dellt
	}

		\newglossaryentry{AK}
	{
	    name=\ensuremath{\mathscr{A}_{\mathcal{K}|L}},
	    description={$\pi_L$-adic completion of $\mathcal{O_K} ((Z))$ for $\mathcal{K}|L$ finite, unramified}, 
	    sort=ALz
	}

		\newglossaryentry{BL}
	{
	    name=\ensuremath{\mathscr{B}_L},
	    description={Fraction field of $\mathscr{A}_L$}, 
	    sort=BL
	}

		\newglossaryentry{BK}
	{
	    name=\ensuremath{\mathscr{B}_{\mathcal{K}|L}},
	    description={Fraction field of $\mathscr{A}_{\mathcal{K}|L}$, for $\mathcal{K}|L$ finite, unramified}, 
	    sort=BLz
	}

		\newglossaryentry{psKL}
	{
	    name=\ensuremath{\psi_{K|L}},
	    description={$=\frac{1}{\pi_L} \phi_{K|L}^{-1} \circ \mathrm{Tr}$, where $\mathrm{Tr}$ is the trace map
	    of $\mathscr{B}_{K|L}|\phi_{K|L}(\mathscr{B}_{K|L})$}, 
	    sort=psikl
	}

		\newglossaryentry{NorKL}
	{
	    name=\ensuremath{\mathrm{Nor}_{K|L}},
	    description={$=\phi_{K|L}^{-1} \circ \mathrm{Nor}$, where $\mathrm{Nor}$ is the norm map
	    of $\mathscr{B}_{K|L}|\phi_{K|L}(\mathscr{B}_{K|L})$}, 
	    sort=norkl
	}

		\newglossaryentry{OmeAK}
	{
	    name=\ensuremath{\Omega^1_{\mathscr{A}_{\mathcal{K}|L}}},
	    description={$=\mathscr{A}_{\mathcal{K}|L} \mathrm{d} Z$; free rank one differential forms over 
	    $\mathscr{A}_{\mathcal{K}|L}$, for $\mathcal{K}|L$ finite, unramified}, 
	    sort=omeak
	}

		\newglossaryentry{Res}
	{
	    name=\ensuremath{\mathrm{Res}},
	    description={Residue homomorphism from $\Omega^1_{\mathscr{A}_{\mathcal{K}|L}}$ to
	    $\mathcal{O_K}$ with $\mathrm{Res}(\sum a_i Z^{i} \mathrm{d} Z) = a_{-1}$}, 
	    sort=res
	}

		\newglossaryentry{Homcts}
	{
	    name=\ensuremath{\mathrm{Hom^{cts}}},
	    description={Continuous homomorphisms}, 
	    sort=homcts
	}

		\newglossaryentry{psM}
	{
	    name=\ensuremath{\psi_M},
	    description={The homomorphism
	    \[ \begin{xy} \xymatrix{
		 M \ar[r]^-{(\varphi_M^{\mathrm{lin}})^{-1}} &\mathscr{A} K \ _{\varphi{K|L}}\otimes_{\mathscr{A}_K} M 
		 		\ar[r] &  M \\
				& f \otimes m \ar@{|->}[r]& \psi_{K|L} (f)m.
	} \end{xy}\]
	for $M \in \mathbf{Mod}_{\varphi,\Gamma}^{\acute{e}t}(\mathscr{A}_K)$}, 
	    sort=
	}

		\newglossaryentry{Mvee}
	{
	    name=\ensuremath{M^{\vee}},
	    description={$= \mathrm{Hom^{cts}}_{\mathcal{O}_L}(M,L/\mathcal{O}_L)$. Pontrjagin dual of the 
	    $\mathcal{O}_L$-module $M$}, 
	    sort=Mvee
	}
	
		\newglossaryentry{HIW}
	{
	    name=\ensuremath{H_{\mathrm{Iw}}^*(K_{\infty}|K,-)},
	    description={Generalized Iwasawa cohomology, i.e. for 
	    $V \in \mathbf{Rep}_{\mathcal{O}_L}^{\mathrm{(fg)}}(G_K)$ we have
	   \[ H_{\mathrm{Iw}}^{i}(K_{\infty}|K,V)= \varprojlim_{\overset{K \subseteq E \subseteq K_{\infty}}{\mathrm{finite}}}
	   	 H^{i}(G_E,V)\]
	where $i \in \mathbb{N}_0$ }, 
	    sort=HIW,
	    symbol=\ensuremath{H_{\mathrm{Iw}}^{i}}
	}	

		\newglossaryentry{rec}
	{
	    name=\ensuremath{\mathrm{rec}},
	    description={The map from $\varprojlim K_n^{\times}$ into the maximal abelian pro-$p$ quptient
	   $ H_K^{\mathrm{ab}}(p)$ of $H_K$ induced from the reciprocity map}, 
	    sort=rec
	}

		\newglossaryentry{recEK}
	{
	    name=\ensuremath{\mathrm{rec}_{\mathbf{E}_K}},
	    description={The reciprocity homomorphism $\mathbf{E}_K \to H_K^{\mathrm{ab}}(p)$ in characteristic $p$}, 
	    sort=recek
	}

	\newglossaryentry{ModpGKef}
	{
	    name=\ensuremath{\mathbf{Mod}_{\varphi,\Gamma}^{\mathrm{\acute{e}t},f}(\mathbf{A}_{\mathcal{K}|L})},
	    description={Category of free \'{e}tale $(\varphi_{\mathcal{K}|L},\Gamma_{\mathcal{K}})$-modules
	    over $\mathbf{A}_{\mathcal{K}|L}$, where $\mathcal{K}|L$ is finite}, 
	    sort=modpgkf,
	    symbol=\ensuremath{\mathbf{Mod}_{\varphi,\Gamma}^{\mathrm{\acute{e}t,f}}(\mathbf{A}_{{K}|L})}
	}

		\newglossaryentry{RepGKOLfgf}
	{
	    name=\ensuremath{\mathbf{Rep}_{\mathcal{O}_L}^{\mathrm{(fg,f)}}(G)},
	    description={Category of finite free $G$-representations of $\mathcal{O}_L$, where $G$
	    is a group}, 
	    sort=RepGKf,
	    symbol=\ensuremath{\mathbf{Rep}_{\mathcal{O}_L}^{\mathrm{(fg,f)}}(G_K)}
	}
	
		\newglossaryentry{Fn}
	{
	    name=\ensuremath{F_n},
	    description={Unramified extension of degree $p^n$ over $F$, where $F|L$ is unramified}, 
	    sort=Fn
	}

		\newglossaryentry{Fin}
	{
	    name=\ensuremath{F_{\infty}},
	    description={$=\cup_n F_n$}, 
	    sort=Fnz
	}
	
		\newglossaryentry{Upsilon}
	{
	    name=\ensuremath{\Upsilon},
	    description={Galois group if $F_{\infty}|F$, where $F|L$ unramified}, 
	    sort=Upsi
	}

		\newglossaryentry{Upsilonn}
	{
	    name=\ensuremath{\Upsilon_{F_n|F}},
	    description={Galois group of $F_{n}|F$, where $F|L$ unramified and $F_n|F$ is the unique unramified 
	    extension of degree $p^n$}, 
	    sort=Upsin
	}

		\newglossaryentry{RG}
	{
	    name=\ensuremath{R[G]},
	    description={Group ring of $G$ with coefficients in the ring $R$}, 
	    sort=RG
	}

		\newglossaryentry{LambdaRG}
	{
	    name=\ensuremath{\Lambda_R(G)},
	    description={$=\varprojlim_{H \triangleleft G \ \mathrm{open}} R(G/H)$ the Iwasawa module of the profinite
	    group $G$ with coefficients in the ring $R$}, 
	    sort=LambdaRG
	}

		\newglossaryentry{ThetOCpb}
	{
	    name=\ensuremath{\Theta_{\mathcal{O}_{\mathbb{C}_p^{\flat}}}},
	    description={Surjective homomorphism from $W(\mathcal{O}_{\mathbb{C}_p^{\flat}})$ to 
	    $\mathcal{O}_{\mathbb{C}_p^{\flat}}$ with kernel generated by $\xi=\tau(\widetilde{\pi_L}-\pi_L)$}, 
	    sort=thetaocpb
	}

		\newglossaryentry{xi}
	{
	    name=\ensuremath{\xi},
	    description={$=\tau(\widetilde{\pi_L})-\pi_L$, where $\tau$ denotes the Teichm\"uller Lift}, 
	    sort=xi
	}

		\newglossaryentry{tildpiL}
	{
	    name=\ensuremath{\widetilde{\pi_L}},
	    description={$=(\pi_n \bmod \pi_L \mathcal{O}_{\mathbb{C}_p})_n \in \mathcal{O}_{\mathbb{C}_p^{\flat}}$, where
	    $\pi_0=\piL$ and $\pi_{n+1}^{q_L}=\pi_n$}, 
	    sort=piltild
	}

		\newglossaryentry{Bdrp}
	{
	    name=\ensuremath{\mathbf{B}_{\mathrm{dR}}^+},
	    description={$=\varprojlim_n W(\mathcal{O}_{\mathbb{C}_p^{\flat}})[1/p]/(\xi)^n$}, 
	    sort=bdrp
	}

		\newglossaryentry{Bdr}
	{
	    name=\ensuremath{\mathbf{B}_{\mathrm{dR}}},
	    description={$=\mathbf{B}_{\mathrm{dR}}^+[1/\xi]]$ -  de Rham period ring}, 
	    sort=bdr
	}

		\newglossaryentry{tdr}
	{
	    name=\ensuremath{t_{\mathrm{LT}}},
	    description={$=\mathrm{log_{LT}}(\omega_{\phi}) $}, 
	    sort=tlt
	}

		\newglossaryentry{Acrisn}
	{
	    name=\ensuremath{\mathbf{A}_{\mathrm{cris}}^0},
	    description={$=\left\{\left. \sum_{n=0}^N a_n \frac{\xi^n}{n!}\right| N \in \mathbb{N}_0, a_n \in 
	    W(\mathcal{O}_{\mathbb{C}_p^{\flat}}) \right\}$}, 
	    sort=Acris
	}

		\newglossaryentry{Acris}
	{
	    name=\ensuremath{\mathbf{A}_{\mathrm{cris}}},
	    description={$=\varprojlim_{n} \mathbf{A}_{\mathrm{cris}}^0/p^n\mathbf{A}_{\mathrm{cris}}^0$}, 
	    sort=Acrisa
	}

		\newglossaryentry{Bcrisp}
	{
	    name=\ensuremath{\mathbf{B}_{\mathrm{cris}}^+},
	    description={$=\mathbf{A}_{\mathrm{cris}}[1/\pi_L]$}, 
	    sort=Bcris
	}

		\newglossaryentry{Bcris}
	{
	    name=\ensuremath{\mathbf{B}_{\mathrm{cris}}},
	    description={$=\mathbf{B}_{\mathrm{cris}}[1/t_{\mathrm{LT}}]=\mathbf{A}_{\mathrm{cris}}[1/t_{\mathrm{LT}}]$
	    - crystalline period ring}, 
	    sort=Bcrisa
	}

		\newglossaryentry{ALcris}
	{
	    name=\ensuremath{\mathbf{A}_{\mathrm{cris},L}},
	    description={$=\mathbf{A}_{\mathrm{cris}} \otimes_{\mathcal{=}_{L_0}} \mathcal{O}_L$}, 
	    sort=Acrisl
	}

		\newglossaryentry{BLcrisp}
	{
	    name=\ensuremath{\mathbf{B}_{\mathrm{cris},L}^+},
	    description={$=\mathbf{B}_{\mathrm{cris}}^+ \otimes_{L_0} L$}, 
	    sort=Bcrisl
	}

		\newglossaryentry{BLcris}
	{
	    name=\ensuremath{\mathbf{B}_{\mathrm{cris},L}},
	    description={$=\mathbf{B}_{\mathrm{cris}} \otimes_{L_0} L$}, 
	    sort=Bcrisla
	}

		\newglossaryentry{Mdr}
	{
	    name=\ensuremath{\mathcal{D}_{\mathrm{dR}}(-)},
	    description={$=\left(\mathbf{B}_{\mathrm{dR}} \otimes_{\mathbb{Q}_p} - \right)^{G_L}$}, 
	    sort=Mdr,
	    symbol=\ensuremath{\mathcal{D}_{\mathrm{dR}}(V)}
	}

		\newglossaryentry{Mcris}
	{
	   name=\ensuremath{\mathcal{D}_{\mathrm{cris}}(-)},
	    description={$=\left(\mathbf{B}_{\mathrm{cris}} \otimes_{\mathbb{Q}_p} - \right)^{G_L}$}, 
	    sort=Dcris,
	    symbol=\ensuremath{\mathcal{D}_{\mathrm{cris}}(V)}	}

		\newglossaryentry{MLcris}
	{
	    name=\ensuremath{\mathcal{D}_{\mathrm{cris},L}(-)},
	    description={$=\left(\mathbf{B}_{\mathrm{cris},L} \otimes_{L} - \right)^{G_L}=
	    \left(\mathbf{B}_{\mathrm{cris}} \otimes_{L_0} - \right)^{G_L}$}, 
	    sort=Dcrisl,
	    symbol=\ensuremath{\mathcal{D}_{\mathrm{cris},L}(V)}	}

		\newglossaryentry{FiliMdr}
	{
	    name=\ensuremath{\mathrm{Fil}^{i}\mathcal{D}_{\mathrm{dR}}(V)},
	    description={$i$-th filtration step of $\mathcal{D}_{\mathrm{dR}}(V)$, where
	    $V \in \mathbf{Rep}_L^{\mathrm{(fg)}}(G_L)$}, 
	    sort=Filmdr
	}

		\newglossaryentry{griMdr}
	{
	    name=\ensuremath{\mathrm{Fil}^{i}\mathcal{D}_{\mathrm{dR}}(V)},
	    description={$=\mathrm{Fil}^{i}\mathcal{D}_{\mathrm{dR}}(V)/\mathrm{Fil}^{i+1}\mathcal{D}_{\mathrm{dR}}(V)$,
	     where  $V \in \mathbf{Rep}_L^{\mathrm{(fg)}}(G_L)$}, 
	    sort=grmdr
	}

		\newglossaryentry{RepGLOLcrisan}
	{
	    name=\ensuremath{\mathbf{Rep}_{\mathcal{O}_L}^{\mathrm{cris,an}}(G_L)},
	    description={Full subcategory of $\mathbf{Rep}_{\mathcal{O}_L}^{\mathrm{(fg,f)}}(G_L)$ consisting 
	    of the crystalline and analytic reprensentations}, 
	    sort=RepGLOLcrisan
	}
	
		\newglossaryentry{RepGLLcrisan}
	{
	    name=\ensuremath{\mathbf{Rep}_{L}^{\mathrm{cris,an}}(G_L)},
	    description={Full subcategory of $\mathbf{Rep}_{L}^{\mathrm{(fg)}}(G_L)$ consisting 
	    of the crystalline and analytic reprensentations}, 
	    sort=RepGLOLcrisan
	}

		\newglossaryentry{Sn}
	{
	    name=\ensuremath{S_n},
	    description={$=(\mathcal{O}_{F_n}[\Upsilon_n])^{\Delta_1=\Delta_2}$}, 
	    sort=Sn
	}

		\newglossaryentry{Delta1}
	{
	    name=\ensuremath{\Delta_1},
	    description={Action from $\Upsilon_n$ on $\mathcal{O}_{F_n}$ with
	    \[ \Delta_1(h,\sum x_g \cdot g) = \sum h(x_g) \cdot g\]}, 
	    sort=Deltaz1
	}

		\newglossaryentry{Delta2}
	{
	    name=\ensuremath{\Delta_2},
	    description={Action from $\Upsilon_n$ on $\mathcal{O}_{F_n}$ with
	    \[ \Delta_1(h,\sum x_g \cdot g) = \sum x_{h^{-1}g} \cdot g\] }, 
	    sort=Deltaz2
	}

		\newglossaryentry{Xin}
	{
	    name=\ensuremath{\Xi_n},
	    description={Galois group of $F_n|F_{n-1}$}, 
	    sort=Xin
	}

		\newglossaryentry{Trn}
	{
	    name=\ensuremath{\mathrm{Tr}_n},
	    description={Trace map of $F_n|F_{n-1}$}, 
	    sort=Trn
	}

		\newglossaryentry{Sin}
	{
	    name=\ensuremath{S_{\infty}},
	    description={$=\varprojlim_n S_n$}, 
	    sort=Sna
	}

		\newglossaryentry{Thetan}
	{
	    name=\ensuremath{\Theta_n},
	    description={Galois group of $F_{\infty}|F_n$}, 
	    sort=thetan
	}

		\newglossaryentry{NE}
	{
	    name=\ensuremath{\mathcal{N}_{\mathcal{K}|L}(T)},
	    description={Wach module of $T \in \mathbf{Rep}_{\mathcal{O}_L}^{\mathrm{cris,an}}(G_{\mathcal{K}})$ for
	    $\mathcal{K}|L$ finite}, 
	    sort=N,
	    symbol=\ensuremath{\mathcal{N}_{E|L}(T)}
	}

		\newglossaryentry{NOFin}
	{
	    name=\ensuremath{\mathcal{N}_{{F_\infty}|L}(T)},
	    description={$=\varprojlim_n \mathcal{N}_{{F_n}|L}(T)$}, 
	    sort=NA
	}

		\newglossaryentry{tenc}
	{
	    name=\ensuremath{\widehat{\otimes}},
	    description={Completed tensor product}, 
	    sort=tenc
	}

		\newglossaryentry{Qph}
	{
	    name=\ensuremath{Q_{\phi}},
	    description={$=\frac{[\piL]_{\phi} \omega_{\phi}}{\omega_{\phi}}$}, 
	    sort=Qphi
	}

		\newglossaryentry{phNEV}
	{
	    name=\ensuremath{\varphi^*\mathcal{N}_{\mathcal{K}|L}(V)},
	    description={The $\mathbf{A}_{\mathcal{K}|L}^+$-submodule of $\mathcal{N}_{\mathcal{K}|L}(V)[1/Q_{\phi}]$
	    generated by $\mathrm{im}(\varphi_{\mathcal{N_K}(V)})$, where 
	    $V \in \mathbf{Rep}_{L}^{\mathrm{cris,an}}$ and $\mathcal{K}|L$ finite}, 
	    sort=phNEV,
	    symbol=\ensuremath{\varphi^*\mathcal{N}_{E|L}(V)}
	}
	
		\newglossaryentry{psNEV}
	{
	    name=\ensuremath{\psi_{\mathcal{N}_{\mathcal{K}|L}(V)}},
	    description={From $\psi_{\mathcal{M}_{\mathcal{K}|L}(V)}$ induced homomorphism from 
	    $\varphi^*\mathcal{N}_{\mathcal{K}|L}(V)$ to
	    $\mathcal{N}_{\mathcal{K}|L}(V)$, where 
	    $V \in \mathbf{Rep}_{L}^{\mathrm{cris,an}}$ and $\mathcal{K}|L$ finite}, 
	    sort=psNEV,
	    name=\ensuremath{\psi_{\mathcal{N}_{E|L}(V)}}
	}	
	
		\newglossaryentry{Robp}
	{
	    name=\ensuremath{\mathcal{R}_{\mathcal{K}}^+},
	    description={Subring of the power series ring with coefficients in $\mathcal{K}$, consisting of those
	    power series converging for all $z \in \mathbb{C}_p$ with absolute value less than $1$, 
	    $\mathcal{K}|L$ finite}, 
	    sort=Robp,
	    symbol={\ensuremath{\mathcal{R}_K^+}}
	}
	
		\newglossaryentry{Robi}
	{
	    name=\ensuremath{\mathcal{R}_{\mathcal{K}}^{I}},
	    description={Ring inside $\mathcal{K}\llbracket Z \rrbracket$, consisting of those
	    elements converging for $z \in \mathbb{C}_p$ with absolute value in $I$, where $I \subseteq [0,1]$
	    is an interval and $\mathcal{K}$ is a complete extension of $L$}, 
	    sort=Robpa
	}
	
		\newglossaryentry{Robr}
	{				
	    name=\ensuremath{\mathcal{R}_{\mathcal{K}}^{[r,1)}},			
	    description={$= \varprojlim_{r<s<1} \mathcal{R}_{\mathcal{K}}^{[r,s]}$ with
	    $0<r<1$ and where $\mathcal{K}$ is a complete extension of $L$}, 
	    sort=Robpb
	}

		\newglossaryentry{Rob}
	{
	    name=\ensuremath{\mathcal{R}_{\mathcal{K}}},		
	    description={$\cup_{0<r<1} \mathcal{R}_{\mathcal{K}}^{[r,1)}$, where 	
	    $\mathcal{K}$ is a complete extension of $L$}, 			
	    sort=Robpc
	}
	
		\newglossaryentry{RobO}
	{
	    name=\ensuremath{\mathcal{R}_{\mathcal{K}}(\mathcal{O}_L)},		
	    description={ring extension of $D_L(\mathcal{O}_L,\mathcal{K})$ corresponding to the
	    extension $\mathcal{R}_{\mathcal{K}}^+ \subseteq \mathcal{R}_{\mathcal{K}}$, where
	    $\mathcal{K}$ is a complete extension of $L$. 
	    Note: $\mathcal{R}_{\mathcal{K}}^+ \cong D_L(\mathcal{O}_L,\mathcal{K})$}, 			
	    sort=Robpd
	}
	
		\newglossaryentry{RobGa}
	{
	    name=\ensuremath{\mathcal{R}_{\mathcal{K}}(\Gamma_L)},		
	    description={ring extension of $D_L(\Gamma_L,\mathcal{K})$ corresponding to the
	    extension $(\mathcal{R}_{\mathcal{K}}^+)^{\psi_L=0} \subseteq (\mathcal{R}_{\mathcal{K}})^{\psi_L=0}$, where
	    $\mathcal{K}$ is a complete extension of $L$. 
	    Note: $(\mathcal{R}_{\mathcal{K}}^+)^{\psi_L=0} \cong D_L(\Gamma_L,\mathcal{K})$}, 	
	    sort=Robpd
	}

		\newglossaryentry{etabz}
	{
	    name=\ensuremath{\eta(b,Z)},
	    description={$=\mathrm{exp}(b \Omega \mathrm{log_{LT}}(Z))$, where 
	    $b \in \mathcal{O}_L/\pi_L\mathcal{O}_L$ and $\Omega$ the period of a fixed generator 
	    $t'$ of the dual of the Tate module $\mathcal{TG}_{\phi}$ of the chosen Lubin-Tate group}, 
	    sort=etabZ
	}

		\newglossaryentry{LGUV}
	{
	    name=\ensuremath{\mathcal{L}_{V}^{\Gamma_L, \Upsilon}},
	    description={Regulator map from $H_{\mathrm{Iw}}^1(K_{\infty}F_{\infty}|K,T)$ to
	    $D_{L}(\Gamma_L,\mathbb{C}_p) \widehat{\otimes}_{\mathcal{O}_F}(\Upsilon, \mathbb{C}_p) 
	    \otimes_L \mathcal{M}_{\mathrm{cris},L}(V)$,
	    where $T \in \mathbf{Rep}_{\mathcal{O}_L}^{\mathrm{cris,an}}(G_L)$ and $V= T[1/\pi_L]$}, 
	    sort=LGUV
	}

		\newglossaryentry{mupin}
	{
	    name=\ensuremath{\mu_{\pi_L^n}},
	    description={Multiplication with $\pi_L^n$ on an $\mathcal{O}_L$-module}, 
	    sort=mupi
	}

		\newglossaryentry{An}
	{
	    name=\ensuremath{A_n},
	    description={$=\mathrm{ker}(\mu_{\pi_L^n}\colon A \to A$, where $A$ is a cofinitely
	    generated $\mathcal{O}_L$-module}, 
	    sort=An
	}

		\newglossaryentry{UM}
	{
	    name=\ensuremath{{}_U M},
	    description={$= \mathrm{Hom}_{\mathcal{O}_L}(\mathcal{O}_L[G/U],M)$, where $M$ is an
	    ind-admissible $\mathcal{O}_L[G]$-module, $G$ a profinite group and $U \subseteq G$ an
	    open subgroup}, 
	    sort=MU
	}

		\newglossaryentry{UK}
	{
	    name=\ensuremath{\mathcal{U}(G;H)},
	    description={Open subgroups of a profinite group $G$ containing $H$, which is a closed, normal
	    subgroup }, 
	    sort=UGH,
	    symbol=\ensuremath{\mathcal{U}(G;H)}
	}
	
		\newglossaryentry{UKK}
	{
	    name=\ensuremath{\mathcal{U}_K},
	    description={$=\mathcal{U}(G_K;H_K)$}, 
	    sort=UGHK,
	    symbol=\ensuremath{\mathcal{U}_K=\mathcal{U}(G_K;H_K)}
	}

	\newglossaryentry{UKE}
	{
	    name=\ensuremath{\mathcal{U}(G)},
	    description={$=\mathcal{U}(G;\{1\}) $ for a profinite group $G$}, 
	    sort=UGHE
	}

		\newglossaryentry{FGH}
	{
	    name=\ensuremath{F_{G/H}(M)},
	    description={$= \varinjlim_{U \in \mathcal{U}(G;H)} {}_U M$, where $M$ is a discrete 
	    $\mathcal{O}_L[G]$-module, $G$ is a profinite group and $H\triangleleft G$ is a closed, normal
	    subgroup}, 
	    sort=FGAH,
	    symbol=\ensuremath{F_{G/H}(M)}
	}
	
			\newglossaryentry{FG}
	{
	    name=\ensuremath{F_{G}(M)},
	    description={$F_{G/\{1\}}(M)$, where $M$ is a discrete 
	    $\mathcal{O}_L[G]$-module and $G$ is a profinite group}, 
	    sort=FGAHG
	}
	
		\newglossaryentry{FGK}
	{
	    name=\ensuremath{F_{\Gamma_K}(M)},
	    description={$F_{G_K/H_K}(M)$, where $M$ is a discrete 
	    $\mathcal{O}_L[G_K]$-module}, 
	    sort=FGAK
	}

		\newglossaryentry{LamK}
	{
	    name=\ensuremath{\Lambda_{\mathcal{K}}},
	    description={$= \mathcal{O}_L \llbracket \Gamma_{\mathcal{K}} \rrbracket$, Iwasawa algebra of 
	    $\Gamma_{\mathcal{K}}$, with $\mathcal{K}|L$ finite },
	    sort=LamK,
	    symbol=\ensuremath{\Lambda_{K}=\mathcal{O}_L \llbracket \Gamma_K \rrbracket}
	}

		\newglossaryentry{ODK}
	{
	    name=\ensuremath{\overline{\mathrm{D}_{\mathcal{K}}}(-)},
	    description={$=\mathrm{Hom}_{\Lambda_{\mathcal{K}}}(-,\Lambda_{\mathcal{K}}^{\vee})$, 
	    so called Matlis dual, where $\mathcal{K}|L$ finite}, 
	    sort=DK,
	    symbol=\ensuremath{\overline{\mathrm{D}_{K}}}(M)
	}	

		\newglossaryentry{RGa}
	{
	    name=\ensuremath{\mathbf{R}\Gamma(\mathfrak{C}^{\bullet})},
	    description={The image of the complex $\mathfrak{C}^{\bullet}$ in the derived corresponding category}, 
	    sort=RG
	}

		\newglossaryentry{RGacts}
	{
	    name=\ensuremath{\mathbf{R}\Gamma_{\mathrm{cts}}^{\bullet}(G,M)},
	    description={$=\mathbf{R}\Gamma(\mathcal{C}^{\bullet}(G,M))$ for a profinite group $G$ and a
	    topological $G$-module $M$}, 
	    sort=RGc
	}

		\newglossaryentry{FCK}
	{
	    name=\ensuremath{\mathcal{F}_{\Gamma_{\mathcal{K}}}(T)},
	    description={$\varprojlim_{U \in \mathcal{U_K}} M \otimes_{\mathcal{O}_L} \mathcal{O}_L[G_\mathcal{K}/U]$,
	    for a topological $G_{\mathcal{K}}$-module $T$, where $\mathcal{K}|L$ finite}, 
	    sort=FK,
	    symbol=\ensuremath{\mathcal{F}_{\Gamma_{K}}(T)}
	}

		\newglossaryentry{RGaiw}
	{
	    name=\ensuremath{\mathbf{R}\Gamma_{\mathrm{Iw}}^{\bullet}(\mathcal{K}_{\infty}|\mathcal{K},T)},
	    description={$=\mathbf{R}\Gamma_{\mathrm{cts}}(G_{\mathcal{K}},\mathcal{F}_{\Gamma_{\mathcal{K}}} (T))$,
	    where $T$ is an $\OL$-representation of $G_{\mathcal{K}}$ and $\mathcal{K}|L$ finite}, 
	    sort=RGiw,
	    symbol=\ensuremath{\mathbf{R}\Gamma_{\mathrm{Iw}}^{\bullet}(K_{\infty}|K,T)}
	}

		\newglossaryentry{tender}
	{
	    name=\ensuremath{\protect\overset{\mathbf{L}}{\otimes}_R},
	    description={Tensor product in the derived category over the ring $R$}, 
	    sort=tenzd
	}

		\newglossaryentry{Xctsb}
	{
	    name=\ensuremath{X_{\mathrm{cts}}^{\bullet}(G,A)},
	    description={the complex with objects $X_{\mathrm{cts}}^n(G,A)$ and differentials $\partial_{\mathrm{cts}}$}, 
	    sort=xctsb
	}

		\newglossaryentry{Xctsn}
	{
	    name=\ensuremath{X_{\mathrm{cts}}^n(G,A)},
	    description={$=\mathrm{Map_{cts}}(G^{n+1},A)$ for an topological abelian Hausdorff group $A$ with
	    a continuous actions from the profinite group $G$}, 
	    sort=Xctsn
	}

		\newglossaryentry{difcts}
	{
	    name=\ensuremath{\partial_{\mathrm{cts}}},
	    description={the differential $X_{\mathrm{cts}}^{n-1}(G,A) \to X_{\mathrm{cts}}^n(G,A)$}, 
	    sort=difcts,
	    symbol=\ensuremath{\partial_{\mathrm{cts}}^n}
	}
	
		\newglossaryentry{Cctsn}
	{
	    name=\ensuremath{C_{\mathrm{cts}}^n(G,A)},
	    description={$=X_{\mathrm{cts}}(G,A)^G$}, 
	    sort=cctsn
	}	
	
		\newglossaryentry{Cctsb}
	{
	    name=\ensuremath{C_{\mathrm{cts}}^{\bullet}(G,A)},
	    description={the complex with objects $C_{\mathrm{cts}}^n(G,A)$ and differentials $\partial_{\mathrm{cts}}$}, 
	    sort=cctsb
	}

		\newglossaryentry{varprojlim}
	{
	    name=\ensuremath{\protect\varprojlim^r},
	    description={$r$-th right derived functor of $\varprojlim$}, 
	    sort=lim
	}

		\newglossaryentry{sigfn}
	{
	    name=\ensuremath{\sigma_{F_n}},
	    description={generator of the Galois group $\Upsilon_n=\mathrm{Gal}{F_n|F}$}, 
	    sort=sigfn
	}

		\newglossaryentry{phsin}
	{
	    name=\ensuremath{\varphi_{S_{\infty}}},
	    description={$=\varprojlim_n \Delta_1(\sigma_{F_n})$, Frobenius of $S_{\infty}$}, 
	    sort=phsin
	}

		\newglossaryentry{pssin}
	{
	    name=\ensuremath{\psi_{S_{\infty}}},
	    description={inverse of the Frobenius $\varphi_{S_{\infty}}$ of $S_{\infty}$}, 
	    sort=pssin
	}

		\newglossaryentry{IndUG}
	{
	    name=\ensuremath{\mathrm{Ind}_U^{G}(M)},
	    description={$=\{ f \colon G \to X\mid f \text{ locally constant and } U\text{-linear}\}$, for a profinite group $G$
	    an open subgroup $U$ and a discrete $U$-module $M$.}, 
	    sort=IndUG
	}

		\newglossaryentry{Adt}
	{
	    name=\ensuremath{\widetilde{\mathrm{Ad}}},
	    description={the action of $G$ on $\mathrm{Ind}_U^G(M)$ given by 
	    $\widetilde{\mathrm{Ad}}(g)(f)(\sigma)=g(f(g^{-1}\sigma g))$ of $G/U$, where $G$ is a profinite group, 
	    $U \triangleleft G$ an open, normal
	    subgroup, $M$ a discrete $\mathcal{O}_L[G]$-module, $f \in \mathrm{Ind}_U^G(M)$ and $g \in G$. This action
	    is trivial on $U$ and therefore induces an action of $G/U$.\\
	    It denotes also the action $\widetilde{\mathrm{Ad}}(gU)(f)(\sigma U) = f(\sigma g U)$ from $G/U$ on
	    ${}_U M$. If $H \triangleleft G$ is a closed normal subgroup, such that $G/H$ is abelian, then
	    it also denotes the induced actions from both, $G/H$ and $\mathcal{O}_L
	    \llbracket G/H \rrbracket$, on $F_{G/H}(M)$.\\
	    Furthermore, it denotes the action $\widetilde{\mathrm{Ad}}{gU}(a \otimes xU) = a \otimes xg^{-1}U$
	    on $T \otimes \mathcal{O}_L[G_K/U]$
	    }, 
	    sort=Adguf,
	    symbol=\ensuremath{\widetilde{\mathrm{Ad}}(g)(f)}
	}

		\newglossaryentry{Ad}
	{
	    name=\ensuremath{\mathrm{Ad}},
	    description={the action of $G$ on $C_{\mathrm{cts}}^{\bullet}(H,M)$ given by
	    \[ \mathrm{Ad}(g)(c)(h_0,\dots,h_n) = g(c(g^{-1}h_0g,\dots,g^{-1}h_ng)),\]
	    where $c \in C_{\mathrm{cts}}^{n}(H,M)$, $G$ is a profinite group,
	    $H \triangleleft G$ a closed, normal subgroup and $M$ a discrete $G$-module\\
	    It denotes also the induced action from $G/H$ on $\mathbf{R}\Gamma_{\mathrm{cts}}(H,M)$
	    }, 
	    sort=Ad,
	    symbol=\ensuremath{\mathrm{Ad}(g)(c)(h_0,\dots,h_n)}
	}

		\newglossaryentry{CD}
	{
	    name=\ensuremath{\mathcal{D}(-)},
	    description={$=\Hom_{\mathbf{A}_{\mathcal{K}|L}}(-,\Omega^1_{\mathbf{A}_{\mathcal{K}|L}} 
	    \otimes_{\mathbf{A}_{\mathcal{K}|L}} \mathbf{B}_{\mathcal{K}|L}/\mathbf{A}_{\mathcal{K}|L})$, where
	    $\mathcal{K}|L$ is finite},
	    sort=DJ,
	    symbol=\mathcal{D}(M)
	}

		\newglossaryentry{DA}
	{
	    name=\ensuremath{\mathbf{D}(\mathbf{A})},
	    description={derived category of the abelian category $\mathbf{A}$}, 
	    sort=DA,
	    symbol=\ensuremath{\mathbf{D}(\mathbf{C})}
	}

		\newglossaryentry{DPA}
	{
	    name=\ensuremath{\mathbf{D}^+(\mathbf{A})},
	    description={full subcategory of $\mathbf{D}(\mathbf{A})$ whose objects are the complexes 
	    with no nonnegative entries}, 
	    sort=DAA,
	    symbol=\ensuremath{\mathbf{D}^+(\mathbf{C})}
	}
	
		\newglossaryentry{DBA}
	{
	    name=\ensuremath{\mathbf{D}^{\mathrm{b}}(\mathbf{A})},
	    description={full subcategory of $\mathbf{D}(\mathbf{A})$ whose objects are the bounded below complexes 
		}, 
	    sort=DAB,
	    symbol=\ensuremath{\mathbf{D}^{\mathrm{b}}(\mathbf{C})}
	}

		\newglossaryentry{iota}
	{
	    name=\ensuremath{\iota},
	    description={the involution $x \mapsto x^{-1}$ of a group}, 
	    sort=iota
	}

		\newglossaryentry{Miota}
	{
	    name=\ensuremath{M^{\iota}},
	    description={the $\Lambda_K$-module $M$ where $\Gamma_K$ acts via the involution, i.e. we have
	    $\gamma \cdot m= \gamma^{-1} m$ for all $\gamma \in \Gamma_K$ and $m \in M$}, 
	    sort=Miota,
	    symbol=\ensuremath{\Lambda_K^{\iota}}
	   }
	
		\newglossaryentry{Can}
	{
	    name=\ensuremath{C^{X\text{-an}}(B,E)},
	    description={locally $X$-analytic functions from $B$ to $E$, where $X|\mathbb{Q}_p$ is finite and
	    $E|X$ is complete, $W$ is a finite dimensional $X$-vector space, and $B\subseteq W$ is
	    a closed polydisk}, 
	    sort=Can
	}
	
		\newglossaryentry{DisX}
	{
	    name=\ensuremath{D_X(G,E)},
	    description={$E$-valued locally $X$-analytic distributions on $B$, where $X|\Qp$ is finite and $E|X$ is
	    complete, and $G$ is a Lie group over $X$. Equivalently, this is the continuous dual of $C^{X\text{-an}}(G,E)$}, 
	    sort=DisX
	}	
	
		\newglossaryentry{ModpGKan}
	{
	    name=\ensuremath{\mathbf{Mod}_{\varphi,\Gamma}^{\mathrm{an}}(\mathbf{A}_{\mathcal{K}|L}^+)},
	    description={Category of analytic $(\varphi_{\mathcal{K}|L},\Gamma_{\mathcal{K}})$-modules
	    over $\mathbf{A}_{\mathcal{K}|L}^+$, where $\mathcal{K}|L$ is finite}, 
	    sort=modpgkan
	}

		\newglossaryentry{CCbs}
	{
	    name=\ensuremath{{C}^{\bullet}[n]},
	    description={shift of the complex by $n \in \ZZ$, i.e. we have ${C}^{i}[n]={C}^{i+n}$}, 
	    sort=CCbs
	}

	\makeglossaries

\usepackage[ngerman,shadow,color=blue!30,textsize=footnotesize,colorinlistoftodos]{todonotes}
	
\DeclareMathOperator{\im}{im}												%
\DeclareMathOperator{\Tot}{Tot}
\DeclareMathOperator{\coker}{coker}										%

\DeclareMathOperator{\DDG}{\mathfrak{Dis}_G}								
\DeclareMathOperator{\DDGNMOD}{\mathfrak{Dis}_{G/N}}						%
\DeclareMathOperator{\DDGN}{\mathfrak{Dis}_{G,\mathbb{N}_0}}					%
\DeclareMathOperator{\DDM}{\mathfrak{Dis}_{M}}								%
\DeclareMathOperator{\DDGM}{\mathfrak{Dis}_{G,M}}							%
\DeclareMathOperator{\DDGNM}{\mathfrak{Dis}_{G/N,M}}							%
\DeclareMathOperator{\DDGXM}{\mathfrak{Dis}_{G\times M}}						%
\DeclareMathOperator{\AbsG}{\mathfrak{Abs}_G}								%
\DeclareMathOperator{\AbsN}{\mathfrak{Abs}_{\mathbb{N}_0}}						%
\DeclareMathOperator{\AbsM}{\mathfrak{Abs}_{M}}								%
\DeclareMathOperator{\AbsGM}{\mathfrak{Abs}_{G,M}}							%
\DeclareMathOperator{\TopG}{\mathfrak{Top}_G}								%
\DeclareMathOperator{\TopGM}{\mathfrak{Top}_{G,M}}							%
\DeclareMathOperator{\TopGN}{\mathfrak{Top}_{G,\NN_0}}						%

	 \renewcommand{\theenumi}{\emph{\arabic{enumi}.}}
	
	\theoremstyle{definition}
	\newtheorem{mydef}{Definition}[subsection]
	\newtheorem{rem1}[mydef]{Remark}
	
	\theoremstyle{plain}
	\newtheorem{prop}[mydef]{Proposition}
	\newtheorem{lem}[mydef]{Lemma}
	\newtheorem{cor}[mydef]{Corollary}
	\newtheorem{thm}[mydef]{Theorem}
	\newtheorem{thma}{Theorem}

	\newtheorem{rem}[mydef]{Remark}
    \newtheorem{qes}[mydef]{Question}

	\newcommand{\RR}{\mathbb{R}}
	
	\newcommand{\NN}{\mathbb{N}}		%
	\newcommand{\ZZ}{\mathbb{Z}}
	
	\newcommand{\Zp}{\mathbb{Z}_p}

	\newcommand{\OO}{\mathcal{O}}
	\newcommand{\OK}{\mathcal{O}_{K}}

	\newcommand{\kK}{k_K}
	
	\newcommand{\kapX}{\kappa \llbracket X \rrbracket}
	
	\newcommand{\kapY}{\kappa \llbracket Y \rrbracket}
	\newcommand{\kX}{k \llbracket X \rrbracket}
	\newcommand{\kpY}{k' \llbracket Y \rrbracket}
	\newcommand{\Kur}{K_{0}}
	\newcommand{\OKur}{\mathcal{O}_{\Kur}}
	
	\newcommand{\Kin}{K_{\infty}}

	\newcommand{\OL}{\mathcal{O}_{L}}

	\newcommand{\OLX}{\mathcal{O}_{L} \llbracket X \rrbracket}
	\newcommand{\OLXY}{\mathcal{O}_{L} \llbracket X, Y \rrbracket}
	
	\newcommand{\kL}{k_{L}}
	\newcommand{\Lur}{L_0}  
	\newcommand{\piL}{\pi_{L}}
	\newcommand{\piK}{\pi_{K}}
	\newcommand{\qL}{q_{L}}
	\newcommand{\qK}{q_{K}}
	\newcommand{\Lin}{L_{\infty}}

	\newcommand{\OLur}{\mathcal{O}_{\Lur}}

	\newcommand{\End}{\mathrm{End}}
	\newcommand{\EndctsG}{\mathrm{End}_{\mathrm{cts},G}}
	\newcommand{\Aut}{\mathrm{Aut}}
	
	\newcommand{\GG}{\mathcal{G}}
	
	\newcommand{\GGphn}{\mathcal{G}_{\phi,n}}
	\newcommand{\MM}{\mathfrak{M}}
	\newcommand{\TG}{\mathcal{TG}}
	\newcommand{\Gal}{\mathrm{Gal}}
	\newcommand{\Qp}{\mathbb{Q}_p}
	\newcommand{\oQp}{\overline{\mathbb{Q}_p}}
	
	\newcommand{\Cp}{\mathbb{C}_{p}}
	\newcommand{\Cpb}{\mathbb{C}_{p}^{\flat}}
	\newcommand{\OCp}{\mathcal{O}_{\Cp}}
	\newcommand{\OCpb}{\mathcal{O}_{\Cp^{\flat}}}

	\newcommand{\Fr}{\mathrm{Fr}}
	
	\newcommand{\phL}{\varphi_L}
	\newcommand{\phKL}{\varphi_{K|L}}
	\newcommand{\phEL}{\varphi_{E|L}}
	\newcommand{\phMlin}{\varphi_M^{\mathrm{lin}}}

	\newcommand{\GaL}{\Gamma_L}
	\newcommand{\GaLn}{\Gamma_{L_n|L}}
	\newcommand{\GaK}{\Gamma_K}

	\newcommand{\chLT}{\chi_{\mathrm{LT}}}
	\newcommand{\chL}{\chi_{L}}
	\newcommand{\chcyc}{\chi_{\mathrm{cyc}}}

	\newcommand{\phM}{\varphi_M}	
	\newcommand{\phN}{\varphi_N}
	\newcommand{\ModpGKe}{\mathbf{Mod}^{\mathrm{\acute{e}t}}_{\varphi,\Gamma}(\AAK)}

	\newcommand{\RepGKOLfg}{\mathbf{Rep}_{\OL}^{\mathrm{(fg)}}(G_K)}

	\newcommand{\MOEK}{\mathcal{D}_{K|L}}
	\newcommand{\MAK}{\mathcal{D}_{K|L}}
	\newcommand{\VK}{\mathcal{V}_{K|L}}

	\newcommand{\pr}{\mathrm{pr}}
	\newcommand{\EEL}{\mathbf{E}_L}
	\newcommand{\EELp}{\mathbf{E}_L^{+}}
	\newcommand{\EELsep}{\mathbf{E}_L^{\mathrm{sep}}}
	\newcommand{\EELsepp}{\mathbf{E}_L^{\mathrm{sep},+}}

	\newcommand{\EEK}{\mathbf{E}_{K|L}}
	\newcommand{\EEKp}{\mathbf{E}_{K|L}^{+}}
	
	\newcommand{\EE}{\mathbf{E}}
	\newcommand{\EEp}{\mathbf{E}^{+}}
	\newcommand{\FF}{\mathbb{F}}
	\newcommand{\AAL}{\mathbf{A}_L}
	\newcommand{\AALp}{\mathbf{A}_L^+}
	\newcommand{\AALnr}{\mathbf{A}_L^{\mathrm{nr}}}
	\newcommand{\AALnrp}{\mathbf{A}_L^{\mathrm{nr},+}}
	\newcommand{\BBL}{\mathbf{B}_L}
	\newcommand{\AAK}{\mathbf{A}_{K|L}}
	\newcommand{\AAKp}{\mathbf{A}_{K|L}^+}

	\newcommand{\AK}{\mathscr{A}_{K|L}}
	\newcommand{\AL}{\mathscr{A}_L}

	\newcommand{\tenphaaK}{\ _{\phKL}{\otimes}_{\AAK}}

	\newcommand{\omphi}{\omega_{\phi}}
	\newcommand{\nuphi}{\nu_{\phi}}

	\newcommand{\AAA}{\mathbf{A}}
	\newcommand{\AAAp}{\mathbf{A}^{+}}
	\newcommand{\EEE}{\mathscr{E}}
	
	\newcommand{\BBK}{\mathbf{B}_{K|L}}
	\newcommand{\BBLnr}{\mathbf{B}_L^{\mathrm{nr}}}
	\newcommand{\Ap}{A^+}
	\newcommand{\Epp}{E^+}
	\newcommand{\Unm}{U_{n,m}}
	\newcommand{\Vnm}{V_{n,m}}
	\newcommand{\Vkm}{V_{k,m}}
	
	\newcommand{\Vqliem}{V_{q_L^{l_{i-1}+1},m}}
	\newcommand{\Uqlmm}{U_{q_L^{l_m},m}}
	\newcommand{\Vqlmem}{V_{q_L^{l_{m-1}+1},m}}
	\newcommand{\Uqlmem}{U_{q_L^{l_{m-1}+1},m}}
	\newcommand{\Uqlmmem}{U_{q_L^{l_{m-1}},m}}

	\newcommand{\mapcts}{\mathrm{Map}_{\mathrm{cts}}} 			%
	\newcommand{\Xcts}{X_{\mathrm{cts}}}
	\newcommand{\Xctsb}{X_{\mathrm{cts}}^{\bullet}}
	\newcommand{\Cctsb}{C_{\mathrm{cts}}^{\bullet}}
	\newcommand{\Ccts}{C_{\mathrm{cts}}}
	\newcommand{\difcts}{\partial_{\mathrm{cts}}}

	\newcommand{\Abb}{A^{\bullet,\bullet}}
	
	\renewcommand{\Bbb}{B^{\bullet,\bullet}}
	
	\newcommand{\id}{\mathrm{id}}
	\newcommand{\Hctss}{H_{\mathrm{cts}}^*}
	\newcommand{\Hcts}{H_{\mathrm{cts}}}
	\newcommand{\Cb}{C^{\bullet}}
	\newcommand{\Cc}{\mathcal{C}}					%
	\newcommand{\Ccb}{\mathcal{C}^{\bullet}}			%
	\newcommand{\Cf}{\mathcal{C}}					%
	\newcommand{\Cfb}{\mathcal{C}^{\bullet}}	
	\newcommand{\Cfrb}{\mathfrak{C}^{\bullet}}				%

	\newcommand{\HfphKLs}{\mathcal{H}_{\phKL}^*}
	\newcommand{\HfphKL}{\mathcal{H}_{\phKL}}
	\newcommand{\Hfph}{\mathcal{H}_{\varphi}}
	
	\newcommand{\Hf}{\mathcal{H}}
	\newcommand{\HFr}{\mathcal{H}_{\mathrm{Fr}}}
	\newcommand{\HFrs}{\mathcal{H}_{\mathrm{Fr}}^*}						%

	\newcommand{\CphKLb}{\mathcal{C}^{\bullet}_{\phKL}}
	\newcommand{\CfphKLb}{\mathfrak{C}^{\bullet}_{\phKL}}
	
	\newcommand{\CFrb}{\mathcal{C}^{\bullet}_{\Fr}}
	\newcommand{\CFr}{\mathcal{C}_{\Fr}}

	\newcommand{\CfphKL}{\mathcal{C}_{\phKL}}
	
	\newcommand{\Mmn}{M_{m,n}}
	\newcommand{\Men}{M_{1,n}}
	\newcommand{\alphmn}{\alpha_{m,n}}
	\newcommand{\Ein}{E_{\infty}}

	\newcommand{\HomctsOL}{\mathrm{Hom}^{\mathrm{cts}}_{\OL}}
	
	\newcommand{\Xb}{X^{\bullet}}
	\newcommand{\Yb}{Y^{\bullet}}
	\newcommand{\Zb}{Z^{\bullet}}
	\newcommand{\gb}{g^{\bullet}}
	\newcommand{\dy}{\mathrm{d}_{Y}}

	\newcommand{\fb}{f^{\bullet}}
	\newcommand{\fbb}{f^{\bullet,\bullet}}
	
	\newcommand{\Hom}{\mathrm{Hom}}
	
	\newcommand{\etale}{{\'{e}tale}}
	
	\newcommand{\CKK}{\mathcal{K}}
	\newcommand{\OCKK}{\overline{\mathcal{K}}}
	\newcommand{\CLL}{\mathcal{L}}
	\newcommand{\CCC}{\mathcal{C}}
	
	\newcommand{\GCLL}{G_{\CLL}}
	\newcommand{\GCLLH}{G_{\CLL^H}}
	\newcommand{\GCLK}{\mathrm{Gal}(\CLL|\CKK)}
	\newcommand{\HCLL}{\widehat{\CLL}}
	
	\newcommand{\sigKL}{\sigma_{K|L}}
	\newcommand{\fsig}{f^{\sigKL}}
	\newcommand{\gLT}{g_{\mathrm{LT}}}

	\newcommand{\mrmd}{\mathrm{d}}
	
	\newcommand{\fthet}{f^{\vartheta}}

	\newcommand{\OmeAK}{\Omega^1_{\AK}}
	\newcommand{\OmeAAK}{\Omega^1_{\AAK}}

	\newcommand{\HomOL}{\mathrm{Hom}_{\OL}}
	
	\newcommand{\HomcOK}{\mathrm{Hom}^{\mathrm{cts}}_{\OK}}
	\newcommand{\HomcOL}{\mathrm{Hom}^{\mathrm{cts}}_{\OL}}

	\newcommand{\HIW}{H_{\mathrm{Iw}}}

	\newcommand{\LamK}{\Lambda_K}
	\newcommand{\LamKp}{\Lambda_{K'}}

	\newcommand{\UK}{\mathcal{U}(G;H)}
	\newcommand{\UKK}{\mathcal{U}_K}
	\newcommand{\UU}{\mathcal{U}}

	\newcommand{\FGK}{F_{\Gamma_K}}
	\newcommand{\FGH}{F_{G/H}}
	
	\newcommand{\ODK}{\overline{\mathrm{D}_K}}
	\newcommand{\CD}{\mathrm{D}}    
	\newcommand{\RGa}{\mathbf{R}\Gamma}
	\newcommand{\RGacts}{\mathbf{R}\Gamma_{\mathrm{cts}}}
	\newcommand{\RGactsb}{\mathbf{R}\Gamma_{\mathrm{cts}}^{\bullet}}
	\newcommand{\RGaiw}{\mathbf{R}\Gamma_{\mathrm{Iw}}}
	\newcommand{\RGaiwb}{\mathbf{R}\Gamma_{\mathrm{Iw}}^{\bullet}}
	\newcommand{\FCK}{\mathcal{F}_{\GaK}}
	\newcommand{\tender}{\overset{\mathbf{L}}{\otimes}}
	\newcommand{\GaKp}{\Gamma_{K'}}

	\newcommand{\Nekovar}{Nekov\'a\v{r}\ }
	\newcommand{\Nekovars}{Nekov\'a\v{r}'s\ }
	
	\newcommand{\IndUG}{\mathrm{Ind}_U^{G}}
	
	\newcommand{\OLGH}{\OL \llbracket G/H \rrbracket}
	
	\newcommand{\Ad}{\mathrm{Ad}}
	\newcommand{\Adt}{\widetilde{\mathrm{Ad}}}
	
	\newcommand{\Amn}{\mathcal{A}_{mn}}
	
	\newcommand{\OLMod}{\OL\text{-}{\mathbf{Mod}}}
	\newcommand{\LamKMod}{\LamK\text{-}{\mathbf{Mod}}}
	
	\newcommand{\DDp}{\mathbf{D}^+}
	\newcommand{\DDb}{\mathbf{D}^{\mathrm{b}}}
	
\address{Mathematisches Institut\\Universit\"{a}t Heidelberg\\Im
  Neuenheimer Feld 205\\D-69120 Heidelberg}

\author{Benjamin Kupferer}
\email[B.\,Kupferer]{bkupferer@mathi.uni-heidelberg.de}
\urladdr[B.\,Kupferer]{https://www.mathi.uni-heidelberg.de/$\,\tilde{}\,$bkupferer/}

\author{Otmar Venjakob}
\email[O.\,Venjakob]{otmar@mathi.uni-heidelberg.de}
\urladdr[O.\,Venjakob]{https://www.mathi.uni-heidelberg.de/$\,\tilde{}\,$otmar/}

\title[Herr-complexes in the Lubin-Tate setting]{Herr-complexes in the Lubin-Tate setting}



\begin{document}

\begin{abstract}
In this article we extend work of Herr from the cyclotomic $(\varphi,\Gamma)$-module to the general case of Lubin-Tate $(\varphi,\Gamma)$. In particular, we define generalized $\varphi$- and $\psi$-Herr complexes, which calculate Galois cohomology, when applied to the \'{e}tale $(\varphi,\Gamma)$-modules attached to the coefficients.
\end{abstract}
	\maketitle
	
	
	
	

\section{Introduction} 
\pagenumbering{arabic}

Fontaine's theory \cite{fo1} of (cyclotomic) $(\varphi,\Gamma)$-modules  plays a central role both in the $p$-adic local Langlands programme, more specifically in Colmez' celebrated work \cite{col4}, as well as in (local) Iwasawa, where, for example, it contributes to proofs of reciprocity formulas or the construction of regulator maps and big exponential maps \cite{ber03} \`{a} la Perrin Riou. One reason for this impact stems from the possibility to explicitly calculate Iwasawa and Galois cohomology of a Galois representation $V$ in terms of the associated $(\varphi,\Gamma)$-module $\mathcal{D}(V)$. While the description of Iwasawa cohomology was given by Fontaine himself, it was his disciple Herr \cite{Herr, Herr1} who described
 Galois cohomology as the cohomology of the following complex, now named after him. To this end we fix an odd prime $p$ and consider a  $\Zp$-representation $V$    of the absolute Galois group $G_{\Qp}$ of the $p$-adic numbers $\Qp.$  Then the complex $\mathcal{C}^{\bullet}_{\varphi}(\Gamma,\mathcal{D}(V)) $
	\[  \xymatrix@C=1.3pc{
		0 \ar[r] & \mathcal{D}(V) \ar[rr]^-{(\varphi-1,\gamma-1)} & & \mathcal{D}(V) \oplus \mathcal{D}(V)
		\ar[rrrr]^-{(\gamma-1)\pr_1 - (\varphi-1)\pr_2} & & & & \mathcal{D}(V) \ar[r] & 0
	}  \]
	computes the group cohomology of $G_{\Qp}$ with values in $V$, where   $\varphi$ denotes the Frobenius endomorphism while
	 $\gamma$ is a topological generator of $\Gamma=G(\Qp(\mu(p))/\Qp).$
	(cf.\ e.g. \cite[Theorem 5.2.2., p.\,93--94]{col1} and \cite[Theorem 5.3.15, p.\,103--104]{col1}).
Upon replacing $\varphi$ by  its left-inverse $\psi$ we obtain the complex $\mathcal{C}^{\bullet}_{\psi}(\Gamma,\mathcal{D}(V)) $, which like a miracle  turns out to be quasi-isomorphic to $\mathcal{C}^{\bullet}_{\varphi}(\Gamma,\mathcal{D}(V)) $ in the  cyclotomic  theory (cf.\ \cite[Proposition 5.3.14, p.\,103]{col1}). Moreover, Herr established that taking an appropriate dual of the Herr complex $\mathcal{C}^{\bullet}_{\varphi}(\Gamma,\mathcal{D}(V^*(1))) $ results - up to shifting by $2$ - another complex quasi-isomorphic to $\mathcal{C}^{\bullet}_{\psi}(\Gamma,\mathcal{D}(V)) $ giving rise to Tate's local duality.

Recently there has been quite some activity to develop a theory of Lubin-Tate $(\varphi,\Gamma)$-modules \cite{ber16,berFo,ber13,berschxie,SV15,fx}, where the ground field now is  a finite extension $L$  of $\Qp$ with ring of integers $\mathcal{O}_L$ and prime element $\pi_L.$ Fixing  a Lubin-Tate formal group $\mathcal{G}$ associated to $\pi_L$, we obtain the  Galois extension $L_\infty$ by adjoining the $\pi_L^n$-division points of $\mathcal{G}$ to $L$ with $L$-analytic Galois group $\Gamma_L=G(L_\infty/L);$ we also set $H_L:=G_{L_\infty} $. Following Fontaine's original ideas Kisin-Ren   established an equivalence of categories \cite[Theorem (1.6), p.\,446]{kisren}), which for the convenience of the reader we recall in section 3 adding some details concerning topologies etc.\ based on the very detailed account \cite[Theorem 3.3.10, p.\,134]{schn1}). In particular, to any finitely generated $\mathcal{O}_L$-module $V$ with linear and continuous action by $G_K$, where $K$ is any finite extension of $L$, we may attach an \'{e}tale $(\varphi,\Gamma)$-module $\MOEK(V)$.
In this context the Iwasawa cohomology again has been successfully described in terms of $\MOEK(V)$ in \cite{SV15}. The purpose of this article is to add an explicit description of the Galois cohomology groups in terms of a $\varphi$- and $\psi$-Herr complex, respectively. To this end let    $V$ any such $\OL$-representation  of $G_K$ as above. Then there is a complex of the corresponding $(\varphi,\Gamma)$-module, of
	which the cohomology is exactly the continuous group cohomology of $G_K$ with coefficients in $V$ as follows.
	By $\Cctsb(G,A)$ we denote the continuous cochain complex of a profinite group $G$ with values
	in the abelian group $A$. Furthermore, for any $(\varphi,\Gamma)$-module $M $ we introduce a generalized Herr complex $\CphKLb(\GaK,M)$ as the
	total complex of the double complex
	\[ \begin{xy} \xymatrix{
		\Cctsb(\GaK,M) \ar[rrr]^-{\Cctsb(\GaK,\varphi_{M})-\id} & & & \Cctsb(\GaK,M)
	} \end{xy} \]
	and we denote by $\HfphKLs(\GaK,M)$ its cohomology. If $\Ccb$ is a bounded below complex of abelian groups (or of $R$-modules
	for a suitable ring $R$), then we denote by $\RGa(\Ccb)$ the same complex
	viewed as object in the derived category $\mathbf{D}^{\mathrm{b}/+}(\mathbf{Ab})$ (respectively in
	$\mathbf{D}^{\mathrm{b}/+}(R\text{-}\mathbf{Mod})$).
	\begin{thma}[{cf.\ \hyperref[galc1.3]{Theorem \ref*{galc1.3}}}]
	Let $V \in \RepGKOLfg$ and set $M = \MOEK(V)$. Then there are isomorphisms
	\[ \begin{xy} \xymatrix@R-1pc{
		&\Hctss(G_K,V) \ar[r]^{\cong} &\HfphKLs(\GaK,M),\\
		&\Hctss(H_K,V) \ar[r]^{\cong} &\HfphKLs(M).
	} \end{xy} \]
	These isomorphisms are functorial in $V$ and compatible with restriction and corestriction. They stem from isomorphisms in $\DDp(\OL\text{-}\mathbf{Mod})$
\[ \xymatrix@R-1pc{
		&\RGa (\Cctsb(G_K,V)) \ar[r]^{\cong} &\RGa (\CphKLb(\GaK,M)),\\
		&\RGa (\Cctsb(H_K,V)) \ar[r]^{\cong} &\RGa (\CphKLb(M)).
	}                                                                  \]
	\end{thma}
In our proof, we follow closely the approach of Scholl in \cite[Theorem 2.2.1, p.\,702--705]{scholl}. Due to the lack of Hochschild-Serre spectral sequences for general continuous cohomology (see \cite{thomas} for a discussion of this issue), the main technical difficulty consists of using Mittag-Leffler type arguments to reduce to cases in which the coefficients become discrete (and hence admit such a spectral sequence).
We would like to mention that in the course of writing up our results we learned that  independently 	Aribam and Kwatra have achieved a (partial) result of this kind, too, concerning torsion coefficients (cf.\ \cite[Theorem 3.16, p.\,10--11]{arikwa}).

The situation for a generalized $\psi$-Herr complex is more difficult. First of all there is no reason why one should obtain a quasi-isomorphic complex upon replacing $\varphi$ by $\psi$. One reason is, that the integral operator $\psi$ considered in \cite{SV15} is   no longer a left inverse to $\varphi$ if $L\neq\Qp.$ Furthermore, the structure of the kernel  $M^{\psi=0}$ of the $\psi$-operator of an \'{e}tale -module  $M$ is difficult to analyse (but see \cite{SV2,berFo} for some aspects in this regard). Therefore, we decided to follow the path of dualizing as this was already successful in \cite{SV15} for the purpose of Iwasawa cohomology:
		 Since $\varphi$ and $\psi$ are related to each other
	under Pontrjagin duality (cf.\ \cite[Remark 5.6, p.\,27]{SV15}), it seems to be the correct way, to dualize
	the complex of $\varphi$.
	One attempt would have been to imitate the methods of Herr (cf.\
	\cite[Lemme 5.6, p.\,333]{Herr1}) to establish a quasi isomorphism between the complexes
	of $(\varphi,\Gamma)$-modules related to $\varphi$ and $\psi$
	using Tate duality. This approach requires to show that all the differentials of the $\varphi$-Herr complex
	have closed image, which implies that they are strict which then implies that the cohomology
	groups of the dualized complex coincide with the dual of the cohomology groups of the complex we
	started with. In his original work, Herr checked that the differentials have closed
	image for each differential separately (cf.\ \cite[p.\,334]{Herr1}). Unfortunately, in the general case
	we have to deal with direct products of Herr's differentials and modules and it is no longer clear,
	that the differentials have closed image.\\
	Instead we imitate results of \Nekovar (cf.\
	\cite[Sections (8.2) and (8.3), p.\,157--160]{neksc}) to replace the complex $\Cctsb(H_K,A)$ with a complex
	$\Cctsb(G_K,\FGK(A))$ of $\LamK=\OL \llbracket \GaK \rrbracket$-modules, where $A = V^{\vee}$ is the dual
	of some $G_K$-representation. Here "replace" means, that the two complexes are quasi isomorphic
	(cf.\ \hyperref[shaind]{Proposition \ref*{shaind}}). This then has the advantage that we can apply the
	Matlis dual $\ODK=\Hom_{\LamK}(-,\LamK^{\vee})$ to this complex. \Nekovar proved that this dualized
	complex is quasi isomorphic to a complex computing the Iwasawa cohomology
	(cf.\ \hyperref[iwdcts]{Lemma \ref*{iwdcts}}). We then finally check, that the complex related to $\psi$
	is quasi isomorphic to this dualized complex.
Using again a result of \Nekovar, we then get the following
	statement.
	\begin{thma}[{cf.\ \hyperref[thmpsicts]{Theorem \ref*{thmpsicts}}}]
	Let $T \in \RepGKOLfg$ and let
	$K \subseteq K' \subseteq \Kin$ an intermediate field, finite over $K$, such that
	$\GaKp\coloneqq \Gal(\Kin|K')$ is isomorphic to some $\Zp^r$. Then we have an isomorphism
	in $\DDp(\OL\text{-}\mathbf{Mod})$
	\[ \RGa(\Ccb_{\psi}(\MAK(T(\tau^{-1})) \tender_{\Lambda_{K'}} \OL
		\cong \RGactsb(G_{K'},T).\]
	\end{thma}		
The left hand side of the isomorphism in the theorem should be considered as generalised $\psi$-Herr complex. For instance, by choosing a Koszul type complex associated with topological generators of $ \GaKp$ one obtains a quite explicit complex, which specialises to the $\psi$-Herr complex in the cyclotomic situation.
Moreover, this result is crucial for descent calculations, see e.g.\ \cite{SV2} in the context of pairings and regulator maps.

\section{Preliminaries} 

	By \gls{nn} we denote the natural numbers starting with $1$ and we let
	$\gls{nnn} = \NN \cup \{0 \}$. For a homomorphism $f \colon A \to B$ we denote by
	\gls{ker} its kernel, by \gls{im} its image and by \gls{coker} its cokernel.
			
	\subsection{On Continuous Group Cohomology}
Continuous group cohomology has been introduced by Tate. For the convenience of the reader we recall the basic notions, but refer the reader to \cite{tate76}, \cite[\S 2.7]{nsw} and \cite[\S 2.1]{kupferer} for further details.

	For topological spaces $X,Y$ we endow the set of continuous maps
	\gls{mapctsxy} always with the compact open topology \index{Topology!compact open}
	(cf.\ \cite[Definition 1, Chapter X \S3.4, p.\,301]{bour3}).
	Note, that in this topology $\mapcts(X,Y)$ is a Hausdorff space if $Y$ is (cf.\ \cite[Remarks (1),
	Chapter X \S3.4, p.\,301--302]{bour3}). For $K \subseteq X$ compact and $U \subseteq Y$ open
	denote by \gls{mku} the set of all $f \in \mapcts(X,Y)$ with $f(K) \subseteq U$.

	\begin{rem} \label{defctscoh}
	We recall that  for  a profinite group $G$   and  an abelian topological group $A$ on which $G$ acts continuously we have the canonical acyclic complex
\begin{equation*}
  \label{pre6.7}  \begin{xy} \xymatrix@R-2pc{
		0 \ar[r] &A \ar[r] &\mapcts(G,A) \ar[r] &\mapcts(G^2,A) \ar[r] & \mapcts(G^3,A) \ar[r] & \cdots
	} \end{xy}
\end{equation*}
Let for $n \in \NN_0$
	\[ \gls{Xctsn} \coloneqq \mapcts(G^{n+1},A) \]
	and $\glssymbol{difcts}\colon \Xcts^{n-1} \to \Xcts^{n}$ be the differential \index{differential}, which is given by
	\[ \difcts^n(x)(\sigma_0,\dots,\sigma_n) = \sum_{i=0}^n (-1)^{i} x(\sigma_0,\dots,\hat{\sigma_i},\dots,\sigma_n), \]
	where "\ $\widehat{\ }$\ " means that the corresponding element is omitted. Furthermore, we denote by
	$\gls{Xctsb}$ the corresponding complex, i.e.
	\[ \begin{xy} \xymatrix{
		\Xctsb(G,A) = \cdots \ar[r]^{\difcts^{n-1}} & \Xcts^{n-1}(G,A) \ar[r]^{\difcts^n}
		& \Xcts^{n}(G,A) \ar[r]^-{\difcts^{n+1}} & \cdots.
	} \end{xy} \]
	As usual, we then set
	\[ \gls{Cctsn} \coloneqq \Xcts^n(G,A)^G.\]
	One checks that $\difcts^n$ restricts to a homomorphism $\Ccts^{n-1}(G,A)\to \Ccts^n(G,A)$. We then let
	\gls{Cctsb} be the complex
	\[ \begin{xy} \xymatrix{
		\Cctsb(G,A) = \cdots \ar[r]^{\difcts^{n-1}} & \Ccts^{n-1}(G,A) \ar[r]^{\difcts^n}
		& \Ccts^{n}(G,A) \ar[r]^-{\difcts^{n+1}} & \cdots.
	} \end{xy} \]
	This complex is called the \emph{\bfseries continuous standard resolution} of $G$ with coefficients in $A$.
	We denote
	its $n$-th cohomology group by $\Hcts^n(G,A)$ and call it the $n$-th \emph{\bfseries
	continuous cohomology group} \index{group cohomology ! continuous}of $G$ with coefficients in $A$.
	
	\end{rem}

	\begin{lem} \label{pre6.10}
	Let $G$ be a profinite group and let
	\[ \begin{xy} \xymatrix{
		&0 \ar[r] & A \ar[r]^-{\alpha} & B \ar[r]^-{\beta} & C \ar[r] &0
	} \end{xy} \]
	be an exact sequence of topological $G$-modules such that the topology of $A$ is induced by that of $B$
	and that $B \to C$ has a continuous set theoretical section $s \colon C \to B$. Then for all
	$n > 0$ the diagrams
	\[ \begin{xy} \xymatrix{
		0 \ar[r] & \mapcts(G^{n-1},A) \ar[r] \ar[d] & \mapcts(G^{n-1},B) \ar[r] \ar[d]& \mapcts(G^{n-1},C) \ar[r] \ar[d] &0 \\
		0 \ar[r] & \mapcts(G^n,A) \ar[r] & \mapcts(G^n,B) \ar[r] & \mapcts(G^n,C) \ar[r] &0
	} \end{xy} \]
	and
	\[ \begin{xy} \xymatrix{
		0 \ar[r] & \mapcts(G^n,A)^G \ar[r] \ar[d]& \mapcts(G^n,B)^G \ar[r] \ar[d]& \mapcts(G^n,C)^G \ar[r] \ar[d]&0\\
		0 \ar[r] & \mapcts(G^{n+1},A)^G \ar[r] & \mapcts(G^{n+1},B)^G \ar[r] & \mapcts(G^{n+1},C)^G \ar[r] &0
	} \end{xy} \]
	are commutative with exact rows and the latter diagram induces a long exact sequence of continuous
	cohomology
	\[ \begin{xy} \xymatrix{
		0 \ar[r] & A^G \ar[r] & B^G \ar[r] & C^G \ar[r] & \Hcts^1(G,A) \ar[r] & \cdots \\
		\cdots \ar[r] & \Hcts^n(G,A) \ar[r] & \Hcts^n(G,B) \ar[r] & \Hcts^n(G,C) \ar[r] & \Hcts^{n+1}(G,A) \ar[r] & \cdots
	} \end{xy} \]
	Furthermore, the topology of $\mapcts(G^n,A)$ is induced by the topology of
	$\mapcts(G^n,B)$ and the section $s\colon C \to B$ induces a continuous, set theoretical
	section  $s_*\colon \mapcts(G^n,C) \to \mapcts(G^n,B)$.
	\end{lem}
	
	\subsection{Monoid Cohomology}  \label{pre7}
	
	As described in the introduction, the aim of \hyperref[galco]{Section 5} is to compute Galois cohomology using
	the theory of Lubin-Tate $(\varphi,\Gamma)$-modules. For this, we also compute
	the cohomology of complexes like $A \overset{f-1}{\longrightarrow} A$, where $A$ is a
	topological abelian group and $f$ is a continuous endomorphism of $A$.\\
	This can be embedded in the theory of monoid cohomology, which then allows us, in the case
	of discrete coefficients, to write this cohomological functor as derived functor. We then combine this
	with a usual group action, which commutes with the endomorphism and obtain spectral sequences
	on cohomology.
		
	 Let $A$ be a topological abelian group and $f \in \End(A)$ continuous. Then
\begin{align}\label{pre7.2}
 \begin{xy} \xymatrix{
		&\cdot\colon \NN_0 \times A \ar[r] &A, \ (n,a) \ar@{|->}[r] & f^n(a),
	} \end{xy}
\end{align}
	defines a continuous $\NN_0$-action on $A$. As in the group case the following holds:
		
		Let $M$ be a topological monoid and $A$ be a discrete abelian group with a continuous action of $M$. Then we have
	\begin{align} \label{pre7.3} A^M \cong \glssymbol{homrmn},
	\end{align}
as $\ZZ[M]$-modules, where $\ZZ$ is considered as trivial $\ZZ[M]$-module.

	We are mostly interested in the case of a discrete $G$-module $A$, where $G$ is a profinite group,
	together with an $\NN_0$-action (which then automatically is continuous since both, $\NN_0$ and $A$
	are discrete), which comes from a $G$-homomorphism of $A$. To shorten notation, we make the following
	definitions.
\begin{mydef}
	Let $G$ be a profinite group and $M$ a topological monoid.\\
	By $\DDM$ we denote the category whose objects are discrete abelian groups with a
	continuous action of $M$ and whose morphisms are the continuous group homomorphisms which
	respect the operation of $M$. \\
	Similarly we denote by $\DDG$ the category whose objects are discrete abelian groups with a
	continuous action of $G$ and whose morphisms are the continuous group homomorphisms which respect the
	operation of $G$. \\
	And finally we denote by $\DDGM$ the category whose objects are discrete
	abelian groups, together with
	commuting continuous actions of $G$ and $M$ and whose morphisms are the continuous
	group homomorphisms which respect the operations from $G$ and $M$.\\
	The corresponding categories, whose objects are abstract abelian groups, are denoted
	by $\AbsM$, $\AbsG$ and $\AbsGM$.\\
	Furthermore, by $\TopG$ we denote the category of topological abelian
	Hausdorff groups with a continuous action from $G$. The morphisms of this category are the
	continuous group homomorphisms which respect the action from $G$.\\
	Analogously we denote by $\TopGM$ the category of topological abelian Hausdorff groups
	with continuous actions from both, $G$ and $M$, such that these actions commute. The morphisms
	of this category are the continuous group homomorphisms which respect the actions from $G$ and $M$.
	\end{mydef}
	
	\begin{rem} \label{pre7.4a}
	Let $G$ be a profinite group and $M$ a topological monoid. Then the categories
	$\DDGM$ and $\DDGXM$ coincide, where $G \times M$ is considered as a topological monoid.
	\end{rem}
	
	Our aim now is to see that the category $\DDGM$ has enough injective objects. For this,
	we follow the idea of  \cite[(2.6.5) Lemma, Chapter I \S 6, p.\,131]{nswo} and outline some details.
	
	\begin{rem} \label{pre7.5}
	Let $G$ be a group and $M$ a monoid.
	Then the category $\AbsGM$ coincides with the category of $\ZZ[G][M]$-modules. In particular, the category $\AbsGM$ has enough injectives.
	\end{rem}

Let $G$ be a profinite group and $M$ a discrete monoid. The usual arguments show that the category $\DDGM$ has 	enough injective objects (cf.\ \cite[Prop.\ 2.2.12]{kupferer} for details).
	
	\begin{lem} \label{pre7.13}
	Let $G$ be a profinite group and $M$ a discrete monoid. Then the functor
	\[ (-)^{G,M}\colon \DDGM \to \gls{Abcat}\]
	is left exact and additive ($\mathbf{Ab}$ denotes the category of abelian groups).
	\begin{proof}
	Since $\DDGM$ and $\DDGXM$ coincide (cf.\ \hyperref[pre7.4a]{Remark \ref*{pre7.4a}}) we can view
	the functor $(-)^{G,M}$ as $(-)^{G\times M}$. Then \hyperref[pre7.3]{  \eqref{pre7.3}}
	says
	\[ (-)^{G \times M} = \Hom_{\ZZ[G \times M]}(\ZZ,-)\]
	which immediately gives the claim, since $\Hom(\ZZ,-)$ is left exact and additive.
	\end{proof}
	\end{lem}
	
	By the above the right derivations for $(-)^{G,M}$, where $G$ is a profinite group and $M$ a discrete monoid,
	exist (cf.\ \cite[\href{https://stacks.math.columbia.edu/tag/05TI}{Tag 0156,
	Lemma 10.3.2 (2)}]{stacks-project}). This then leads us to the following definition.
	
	\begin{mydef}
	Let $G$ be a profinite group and $M$ a discrete monoid.
	Then \linebreak $H^n(G,M; - ) \coloneqq R^n (-)^{G,M}$ denotes the $n$-th right derived functor of $(-)^{G,M}$ and
	is called the $n${\bfseries -th cohomology group}. \index{Cohomology Group}
	\end{mydef}

	\begin{prop} \label{pre7.18}
	Let $G$ be a profinite group, $N \triangleleft G$ a closed, normal subgroup and $M$ a
	discrete monoid. Then for every $A \in \DDGM$ there are two
	cohomological spectral sequences converging to $H^n(G,M;A)$:
	\[ \begin{xy} \xymatrix@R-2pc{
		& H^{a}(G/N, H^b(N,M;A)) \ar@{=>}[r] & H^{a+b}(G,M;A) \\
		& H^{a}(G/N,M; H^b(N,A)) \ar@{=>}[r] & H^{a+b}(G,M;A). \\
	} \end{xy} \]
	\begin{proof}
	  Recall that the categories $\DDGM$, $\DDGNM$ and
	$\DDGNMOD$ have enough injectives. The
	functors $(-)^{N,M}\colon \DDGM \to \DDGNMOD $ respectively   $(-)^N\colon \DDGM \to \DDGNM$
	send injectives to injectives as is straightforward to check, see \cite[Lem.\ 2.2.16]{kupferer}. Furthermore, since the actions of $G$ and $M$ on objects of
	$\DDGM$ commute, the compositions
	\[ \begin{xy} \xymatrix{
		& \DDGM \ar[rr]^-{(-)^{N,M}} & & \DDGNMOD \ar[rr]^-{(-)^{G/N}} & & \mathbf{Ab}
	} \end{xy} \]
	and
	\[ \begin{xy} \xymatrix{
		& \DDGM \ar[rr]^-{(-)^{N}} & & \DDGNM \ar[rr]^-{(-)^{G/N,M}} & & \mathbf{Ab}
	} \end{xy} \]
	both coincide with $(-)^{G,M}$. This then leads to the claimed Grothendieck
	spectral sequences.
	\end{proof}
	\end{prop}
	
	As we now have accomplished the abstract theory for our goals, we want to discuss how to compute
	these cohomology groups when the monoid action arises from an endomorphism. First of all, we want to
	compare $\NN_0$-actions with $\ZZ[X]$-modules.
	
	\begin{rem}
	The category $\AbsN$ coincides with the category of $\ZZ[X]$-modules.
	\end{rem}
	
	 In the following, we will switch between these
	two concepts without further mentioning it.
	
	\begin{rem}
	Let $G$ be a profinite group, $A \in \DDGN$.
	For every $n \in \NN_0$ we can define an $\NN_0$-action on $\Ccts^n(G,A)$ by operating on the coefficients:
	\[ (X \cdot f)(\sigma) \coloneqq X \cdot (f(\sigma)).\]
	\end{rem}
	
	\begin{rem} \label{pre7.19}
	Let $\Abb$ be a (commutative) double complex of abelian groups. We write $\gls{totabb}$ for its total complex,
	\index{Total Complex} by which we mean the complex with objects
	\[ \gls{totnabb} \coloneqq \bigoplus_{i+j=n} A^{i,j}\]
	and differentials
	\[ \gls{difnabb}\coloneqq \bigoplus_{i+j=n} \mathrm{d}^{i,j}_{\mathrm{hor}}\circ \mathrm{pr}_{i-1,j} \oplus (-1)^{i}
	 \mathrm{d}^{i,j}_{\mathrm{vert}} \circ \mathrm{pr}_{i,j-1}.\]
	 If $\fbb\colon \Abb \to \Bbb$ is a morphism of (commutative) double complexes, then
	 \[ \begin{xy} \xymatrix@R-2pc{
	 	&\Tot^n(\fbb)\colon \Tot^n(\Abb) \ar[rr] & & \Tot^n(\Bbb) \\
		& \hspace*{2cm} (a_{ij})_{i+j=n} \ar@{|->}[rr] & & (f_{ij}(a_{ij}))_{i+j=n}
	} \end{xy} \]
	defines a morphism of the corresponding total complexes.\\
	If $\Xb$ and $\Yb$ are complexes of abelian groups and $\gb\colon \Xb \to \Yb$ is a morphism
	of complexes, then it also is a double complex concentrated in degrees $0$ and $1$ and we again write
	$\Tot(\gb\colon \Xb \to \Yb)$
	for its total complex.
	\end{rem}
	
	\begin{rem}
	Let $G$ be a profinite group and $A \in \DDG$. As in \emph{\cite[p.\,12--13]{nswo}} we omit the subscript
	"$\mathrm{cts}$" for the notations introduced in \emph{\hyperref[defctscoh]{Remark \ref*{defctscoh}}}, i.e., we write
	\[ X^n(G,A)\coloneqq \mapcts(G^{n+1},A),\]
	$\partial^n$ for the differential $X^{n-1}(G,A) \to X^n(G,A)$ and
	\[ C^n(G,A) \coloneqq X^n(G,A)^G.\]
	\end{rem}
	
	\begin{mydef} \label{deftotcoh}
	Let $G$ be a profinite group and $A \in \DDGN$. Then define
	\[ \begin{xy} \xymatrix@R-2pc{
	&\gls{CcbX}\coloneqq \Tot( \Cb(G,A) \ar[rr]^-{X-1} & & \Cb(G,A) ), \\
	&\gls{HXGA}\coloneqq H^*(\Ccb_X(G,A)). & &
	} \end{xy} \]
	If the $\NN_0$-action on $A$ comes from an endomorphism $f \in \End_G(A)$
	(cf.\ \hyperref[pre7.2]{ \eqref{pre7.2}}), then we also write
	\[ \begin{xy} \xymatrix@R-2pc{
	&\gls{Ccbf}\coloneqq \Tot( \Cb(G,A) \ar[rr]^-{\Ccb(G,f)-\id} & & \Cb(G,A) ),\\
	&\gls{HXGAf}\coloneqq H^*(\Ccb_f(G,A)).
	} \end{xy} \]
	If $A \in \AbsN$ then we also write $\Hf_X^*(A)$ for the cohomology of the
	complex $A \overset{X-1}{\longrightarrow} A$ concentrated in the degrees $0$ and $1$.
	\end{mydef}
	
	The aim now is to see that the cohomology of the complex $\Ccb_X(G,A)$ coincides with the
	right derived functors of $(-)^{G,\NN_0}$. Before proving this, we want to make a smaller step and explain
	first how to compute the right derived functors of $(-)^{\NN_0}$ and that these coincide with
	the cohomology of the complex $A \overset{X-1}{\longrightarrow} A$ concentrated in degrees $0$ and $1$.
	\begin{prop} \label{pre7.20}
	Let $A \in \AbsN$. Then we have
	\begin{align*}
		 H^0(\NN_0;A) & = A^{\NN_0},\\
		 H^1(\NN_0;A) &= A_{\NN_0},\\
		 H^{i}(\NN_0;A) &= 0  \text{ for all } i \in \ZZ\setminus\{0,1\}.
	\end{align*}
	In particular, the right derived functors of $(-)^{\NN_0}$ coincide with the cohomology
	of the complex $A \overset{X-1}{\longrightarrow} A$ concentrated in degrees $0$ and $1$. Using the
	notation from above, this means that for all $i \in \ZZ$ there are natural isomorphisms
	\[H^{i}(\NN_0;A) = \Hf_X^{i}(A).\]
	\begin{proof} Follows immediately from %
the projective resolution of $\ZZ$ by:
	\[ \begin{xy} \xymatrix@R-2pc{
		& 0 \ar[r] & \ZZ[X] \ar[r] & \ZZ[X] \ar[r] & \ZZ \ar[r] & 0 \\
		& & P(X) \ar@{|->}[r] & (X-1)P(X) & & \\
		& & & P(X) \ar@{|->}[r] & P(1). &
	} \end{xy} \]
	\end{proof}
	\end{prop}

	\begin{prop} \label{pre7.21}
	Let $G$ be a profinite group and $A \in \DDGN$. Then the double complex
	\[ \begin{xy} \xymatrix{
		&K^{\bullet,\bullet}: \Cb(G,A) \ar[rr]^-{X-1} & & \Cb(G,A)
	} \end{xy} \]
	gives rise to two spectral sequences converging
	to the cohomology $\Hf_X^*(G,A)$:
	\[ \begin{xy} \xymatrix@R-2pc{
		&\Hf_X^{a}(H^b(G,A)) \ar@{=>}[r] & \Hf_X^{a+b}(G,A) \\
		&H^{a}(G,\Hf_X^b(A)) \ar@{=>}[r] & \Hf_X^{a+b}(G,A).
	} \end{xy} \]
	\begin{proof}
	Since for every $n \in \ZZ$ the double complex	$K^{\bullet,\bullet}$
	has at most two nonzero entries $K^{p,q}$ with $p+q=n$, this is shown in
	\cite[\href{https://stacks.math.columbia.edu/tag/012X}{Tag 012X, Lemma 12.22.6}]{stacks-project}.
	\end{proof}
	\end{prop}
	
	\begin{lem} \label{pre7.22}
	Let $G$ be a profinite group and $f \colon A \to B$ be a morphism in $\DDGN$. Then
	the diagram
	\[ \begin{xy} \xymatrix{
		&\Cc^n_X(G,A) \ar[d]^-{\partial_A} \ar[rr]^-{C^n_X(G,f)} & &C^n_X(G,B) \ar[d]^{\partial_B} \\
		&\Cc^{n+1}_X(G,A) \ar[rr]^-{C^{n+1}_X(G,f)} & &C^{n+1}_X(G,B)
	} \end{xy} \]
	is commutative for all $n \in \NN_0$.
	\begin{proof} Left to the reader.
	\end{proof}
	\end{lem}
	
	\begin{lem} \label{pre7.23}
	Let $G$ be a profinite group and
	\[ \begin{xy} \xymatrix{
		&0 \ar[r] & A \ar[r]^{\alpha} &B \ar[r]^{\beta} &C \ar[r] &0
	} \end{xy} \]
	be an exact sequence in $\DDGN$. Then, 
the sequence
	\[ \begin{xy} \xymatrix{
	&0 \ar[r] & \Ccb_X(G,A) \ar[rr]^-{\Ccb_X(G,\alpha)} & & \Ccb_X(G,B)
		\ar[rr]^-{\Ccb_X(G,\beta)} & &  \Ccb_X(G,C) \ar[r] &0
	} \end{xy} \]
	is exact.
	\begin{proof}
	Since $A, B$ and $C$ are discrete groups, we deduce from \hyperref[pre6.10]{Corollary \ref*{pre6.10}}
	that for all $n \in \NN_0$ the sequence
	\[ \begin{xy} \xymatrix{
		&0 \ar[r] & \Cc^n(G,A)  \ar[rr]^-{\Cc^n(G,\alpha)} & & \Cc^n(G,B)
		 \ar[rr]^-{\Cc^n(G,\beta)} & & \Cc^n(G,C) \ar[r] &0
	} \end{xy} \]
	is exact. But since $\Cc^n_X(G,Z)= \Cc^n(G,Z) \oplus \Cc^{n-1}(G,Z)$ (where
	$\Cc^{-1}(G,Z)=0$) and $\Cc^n_X(G,\eta) = \Cc^n(G,\eta) \oplus \Cc^{n+1}(G,\eta)$ for all
	$Z \in \DDGN$ and any morphism $\eta$ in $\DDGN$, we immediately deduce that the
	sequence
	\[ \begin{xy} \xymatrix{
		&0 \ar[r] & \Cc^n_X(G,A)  \ar[rr]^-{\Cc^n_X(G,\alpha)} & & \Cc^n_X(G,B)
		 \ar[rr]^-{\Cc^n_X(G,\beta)} & & \Cc^n_X(G,C) \ar[r] &0
	} \end{xy} \]
	is also exact.
	\end{proof}
	\end{lem}

	\begin{lem} \label{pre7.27}
	Let $G$ be a profinite group. The functors $(\Hf^n_X(G,-))_n$ then form a cohomological $\delta$-functor, i.e.
	if
	\[ \begin{xy} \xymatrix{
		&0 \ar[r] & A \ar[r]^{\alpha} &B \ar[r]^{\beta} &C \ar[r] &0
	} \end{xy} \]
	is an exact sequence in $\DDGN$ then, for every $n \in \NN_0$, there is a group homomorphism
	\[ \begin{xy} \xymatrix{
		&\delta^n \colon \Hf_X^n(G,C) \ar[rr] & & \Hf_X^{n+1}(G,A)
	} \end{xy} \]
	such that the sequence
	\[ \begin{xy} \xymatrix{
		 \cdots \ar[r] & \Hf_X^n(G,B) \ar[r] & \Hf_X^n(G,C) \ar[r]^-{\delta^n} & \Hf_X^{n+1}(G,A)
		\ar[r] & \Hf_X^{n+1}(G,B) \ar[r] & \cdots
	} \end{xy} \]
	is exact.
	\begin{proof}
	The proof is the standard application for the snake lemma (cf.\ for example at \cite[(1.3.2) Theorem,
	Chapter I \S 3, p.\,27]{nswo}). 
	\end{proof}
	\end{lem}
	
	
	\begin{lem} \label{pre7.29}
	Let $G$ be a group. Then there holds
	\[ \ZZ[G][X] \cong \ZZ[G] \otimes_\ZZ \ZZ[X].\]
	\end{lem}
	
	\begin{lem} \label{pre7.30}
	Let $G$ be an abelian profinite group. Then, for every $n \in \NN$ the functor $\Hf_X^n(G,-)$ is effaceable, i.e.
	for every $A \in \DDGN$ there exists a $B \in \DDGN$ and a monomorphism
	$u\colon A \to B$ in $\DDGN$ such that $\Hf_X^n(G,u)=0$.
	\begin{proof} Left to the reader, or \cite[Lem.\ 2.2.31]{kupferer}.
	\end{proof}
	\end{lem}
	
	\begin{cor} \label{pre7.31}
	Let $G$ be a profinite group. Then the family of functors $(\Hf_X^n(-))_n$ from
	$\DDGN$ to $\mathbf{Ab}$ forms a universal delta functor.
	\end{cor}
	
	\begin{thm} \label{pre7.33}
	Let $G$ be a profinite group. Then we have
	\[ \Hf^n_X(G,A) = H^n(G,\NN_0;A)\]
	for all $n \in \NN_0$ and $A \in \DDGN$.
	\begin{proof}
	Since $(H^n(G,\NN_0;-))_n$ are the right derived functors of $(-)^{\GG,\NN_0}$ this is a universal
	delta functor and since $(\Hf_X^n(G,-))_n$ is also an universal delta functor
	(cf.\ \hyperref[pre7.31]{Corollary \ref*{pre7.31}}), it remains to check that they coincide
	in degree $0$. For this, let $A \in \DDGN$. We have
	\begin{subequations}
	\begin{equation*} \begin{split}
		H^0(G,\NN_0;A)  &= A^{G,\NN_0} \hspace*{2,3cm}
	\end{split} \end{equation*}
	and
	\begin{equation*} \begin{split}
		\hspace*{3.2cm}	 \Hf_X^0(G,A)  &= H^0(\Ccb_X(G,A)) \\
				&= \ker(A \overset{d^0}{\longrightarrow} \Cc^1(G,A)) \cap
					\ker(A \overset{X-1}{\longrightarrow} A) \\
				&= A^G \cap A^{X=1}
	\end{split} \end{equation*}
	\end{subequations}
	Since, by definition, $A^{X=1} = A^{\NN_0}$ it follows immediately that
	$A^{G,\NN_0} = A^G \cap A^{\NN_0}$.
	\end{proof}
	\end{thm}
	
	Next we want to reformulate \hyperref[pre7.18]{Proposition \ref*{pre7.18}} with the above theorem,
	just to avoid confusions for latter applications.
	
	\begin{prop} \label{pre7.34}
	Let $G$ be a profinite group, $N \triangleleft G$ a closed, normal subgroup and $A \in \DDGN$.
	Then there are two
	cohomological spectral sequences converging to $\Hf_X^n(G,-)$:
	\[ \begin{xy} \xymatrix@R-2pc{
		& H^{a}(G/N, \Hf^b_X(N,A)) \ar@{=>}[r] & \Hf_X^{a+b}(G,M;A) \\
		& \Hf_X^{a}(G/N, H^b(N,A)) \ar@{=>}[r] & \Hf_X^{a+b}(G,M;A). \\
	} \end{xy} \]
	\begin{proof}
	This is \hyperref[pre7.18]{Proposition \ref*{pre7.18}} using $H^n(G,\NN_0;-)=\Hf_X^n(G,-)$
	from \hyperref[pre7.33]{Theorem \ref*{pre7.33}}.
	\end{proof}
	\end{prop}
	
	As for the standard continuous cohomology  (cf.\ \cite[(2.7.2) Lemma, Chapter II \S 7, p.\,137]{nswo}),
	we will also need a long exact sequence for $\Hf_X^*(G,-)$
	in a slightly different setting as in \hyperref[pre7.27]{Lemma \ref*{pre7.27}}.
		
	\begin{prop} \label{pre7.35}
	Let $G$ be a profinite group and let
	\[ \begin{xy} \xymatrix{
		&0 \ar[r] & A \ar[r]^{\alpha} &B \ar[r]^{\beta} &C \ar[r] &0
	} \end{xy} \]
	be a short exact sequence in $\TopGN$ such that the topology of
	$A$ is induced by that of $B$ and such that $\beta$ has a continuous, set theoretical section. Then there are
	continuous homomorphisms
	\[ \begin{xy} \xymatrix{
		&\delta^n \colon \Hf_X^n(G,C) \ar[rr] & & \Hf_X^{n+1}(G,A)
	} \end{xy} \]
	such that the sequence
	\[ \begin{xy} \xymatrix{
		 \cdots \ar[r] & \Hf_X^n(G,B) \ar[r] & \Hf_X^n(G,C) \ar[r]^-{\delta^n} & \Hf_X^{n+1}(G,A)
		\ar[r] & \Hf_X^{n+1}(G,B) \ar[r] & \cdots
	} \end{xy} \]
	is exact.\\
	\begin{proof}
	Algebraically this is exactly the same proof as \hyperref[pre7.27]{Lemma \ref*{pre7.27}}. It then remains
	to check, that the occurring homomorphisms are continuous which is only for the
	$\delta^n$ a real question. But this can be answered using a topological version
	of the snake lemma, like \cite[Proposition 4, p.\,133]{schoch}.
	\end{proof}
	\end{prop}	
	
	
	\subsection{Some Homological Algebra}
	
	In this subsection we want to collect and prove some facts we will need later on.
	
	\begin{mydef}
	Let ${C}^{\bullet}$ be a complex of abelian groups and $n \in \ZZ$. Then we denote by
	\gls{CCbs} the shift of this complex by $n$. This means, that for all $i \in \ZZ$ we have
	${C}^{i}[n]= C^{i+n}$.
	\end{mydef}
	
	\begin{lem} \label{galc1.3l1} \label{galc1.3c1}
	Let $\Yb$ and $\Zb$ be complexes of abelian groups and let $\gb\colon \Yb \to \Zb$ be a morphism
	of complexes, such that every $g^{i}$ is surjective. Then there is a canonical, surjective
	homomorphism
	\[ \ker (\dy^{i}) \cap \ker g^{i} \twoheadrightarrow H^{i}(\Tot(\gb\colon \Yb \to \Zb)).\]
	In particular, if all the $g^{i}$ are bijective, we have
	\[ H^{i}(\Tot(\gb\colon \Yb \to \Zb)) = 0. \]
	\begin{proof} Easy to check, see \cite[Lem.\ 2.3.2]{kupferer}
	\end{proof}
	\end{lem}
	

	\begin{lem} \label{galc1.3l2} \label{galc1.3l3}
	Let
	\[ \begin{xy} \xymatrix{
	0 \ar[r] & \Xb \ar[r]^{\fb} & \Yb \ar[r]^{\gb} & \Zb \ar[r] &0
	}\end{xy}\]
	be a short exact sequence of complexes of abelian groups. Then the sequence
	\[ \begin{xy} \xymatrix{
	0 \ar[r] & \Xb \ar[r] & \Tot (\Yb \to \Zb) \ar[r] & \Tot(\Yb/\fb(\Xb) \to \Zb) \ar[r] &0
	} \end{xy} \]
	is also an exact sequence of complexes and for the cohomology we have
	\[ H^{i}(\Xb) \cong H^{i}(\Tot(\gb\colon \Yb \to \Zb)).\]
	\begin{proof} Standard, see \cite[Lem.\ 2.3.3]{kupferer}.
	\end{proof}
	\end{lem}
	

	\begin{cor} \label{cohomendo}
	Let $G$ be a profinite group, $A, B \in \DDG$ and $f$ a continuous endomorphism of
	$B$ which respects the action of $G$ such that the sequence
	\[ \begin{xy} \xymatrix{
		0 \ar[r] &A \ar[r] &B \ar[r]^{f-1} &B \ar[r] &0
	} \end{xy} \]
	is exact. Then we have
	\[ H^{i}(G,A) = \Hf_f^{i}(G,B)\]
	for all $i \geq 0$.
	\begin{proof}
	This is just the above \hyperref[galc1.3l3]{Lemma \ref*{galc1.3l3}} with \hyperref[pre6.10]{Corollary \ref*{pre6.10}}
	and the notation from \hyperref[deftotcoh]{Definition \ref*{deftotcoh}}.
	\end{proof}
	\end{cor}

	\begin{cor}
	Let $G$ be a profinite group and let
	\[ \begin{xy} \xymatrix{
		0 \ar[r] & A \ar[r]^{\alpha} & B \ar[r]^{\beta} & C \ar[r] &0
	} \end{xy} \]
	be an exact sequence in $\TopG$, such that the topology of $A$ is induced by that of $B$
	and such that $\beta$ has a continuous, set theoretical section. Then the exact sequence of complexes
	\[ \begin{xy} \xymatrix{
		0 \ar[r] & \Cctsb(G,A) \ar[rr]^{\Cctsb(G,\alpha)} & &  \Cctsb(G,B) \ar[rr]^{\Cctsb(G,\beta)} & &
		\Cctsb(G,C) \ar[r] &0
	} \end{xy} \]
	(cf.\ \emph{\hyperref[pre6.10]{Corollary \ref*{pre6.10}}}) induces
	\[ A^G = \Hcts^0(G,A) \cong H^0(\Tot(\Cctsb(G,\beta)\colon\Cctsb(G,B)\to\Cctsb(G,C)))\]
	and
	\[ C^G \to \Hcts^1(G,A) \cong H^1(\Tot(\Cctsb(G,\beta)\colon\Cctsb(G,B)\to\Cctsb(G,C))).\]
	\begin{proof}
	This is an immediate consequence of the combination of the above
	\hyperref[galc1.3l3]{Lemma \ref*{galc1.3l3}} with \hyperref[pre6.10]{Corollary \ref*{pre6.10}}.
	\end{proof}
	\end{cor}

	Now let's turn to some facts about projective limits.

	\begin{rem}
	Note that $\Cctsb(G,-)$ commutes
	with projective limits, since the functors $\mathrm{Map_{cts}}(G^n, - )$ and $(-)^G$ commute
	with projective limits, i.e. if $A = \varprojlim_n A_n$, then
	\[ \Cctsb(G,A) \cong \varprojlim_n \Cctsb(G,A_n).\]
	\end{rem}
	
	\begin{lem} \label{galc1.8}
	Let $G$ be a profinite group, $A \in \TopG$ and let $(A_n)_n$ be an inverse system in $\TopG$ such
	that $A= \varprojlim_n A_n$ in $\TopG$. Let furthermore $f \in \EndctsG(A)$, such that
	$f = \varprojlim f_n$ with $f_n \in \EndctsG(A_n)$. Then it  holds
	\[ \Cfb_f(G,A) \cong \varprojlim_n \Cfb_f(G,A_n).\]
	\begin{proof}
	First we want to note, that for groups $X=\varprojlim X_n$ and $Y =\varprojlim Y_n$ always holds
	$X \times Y = \varprojlim_n (X_n \times Y_n)$.\\
	This means that the objects of the two complexes $\Cfb_f(G,A)$ and $\varprojlim \Cfb_f(G,A_n)$
	coincide, so it remains to check that the differentials do as well. If we denote the
	$i$-th object of $\Cctsb(G,A)$ by $C^{i}$ and the differential by $d^{i}$ then it suffices to check
	that the following cube is commutative
	\[ \begin{xy} \xymatrix{
		& C^{i} \ar[rrr]^-{\Cf^{i}(G,f)} \ar[dd]_{d^{i}} \ar[rd] & & & C^{i} \ar[dd]^{d^{i}} \ar[rd]& \\
		& & C^{i} \ar[rrr]^-{\varprojlim \Cf^{i}(G,f_n) \hspace*{1cm}} \ar[dd]_{d^{i}}
		& & & C^{i} \ar[dd]^{d^{i}} \\
		& C^{i+1} \ar[rrr]^-{\Cf^{i}(G,f)} \ar[rd] & & & C^{i+1} \ar[rd] & \\
		& & C^{i+1} \ar[rrr]^-{\varprojlim \Cf^{i}(G,f_n)} & & & C^{i+1}.
	} \end{xy} \]
	This is a direct consequence from the assumption $f=\varprojlim f_n$ and that $\Cctsb(G,-)$ commutes
	with inverse limits.
	\end{proof}
	\end{lem}
	
	\begin{lem} \label{galc1.8l2}
	Let $G$ be a profinite group and $(A_n)_n$ be an inverse system in $\TopG$
	such that the inverse system of complexes $(\Cctsb(G,A_n))_n$
	has surjective transition maps and let $A\coloneqq\varprojlim_n A_n$. If $f \in \EndctsG(A)$
	then also the system $(\Cfb_f(G,A_n))_n$ has surjective transition maps.
	\begin{proof}
	By assumption, for every $k \in \NN_0$, the transition map
	$\Ccts^k(G,A_n) \to \Ccts^k(G,A_{n-1})$ is surjective. But then also the transition map
	\[ \begin{xy} \xymatrix@R-1pc{
		 \Ccts^k(G,A_n) \oplus \Ccts^{k-1}(G,A_{n}) \ar[r] \ar@{=}[d] &
		\Ccts^k(G,A_{n-1}) \oplus \Ccts^{k-1}(G,A_{n-1})  \ar@{=}[d] \\
		\Cc_f^k(G,A_n) & \Cc_f^k(G,A_{n-1})
	} \end{xy} \]
	is surjective, since it's the direct sum of two surjective maps.
	\end{proof}
	\end{lem}
	
	\begin{mydef} \label{galc1.13}
	An inverse system (of abelian groups) $(X_n)_{n \in \NN}$
	is called {\bfseries Mittag-Leffler} \index{Mittag-Leffler} ({\bfseries ML}) if for any $n \in \NN$, there
	is an $m \geq n$ such
	that the image of the transition maps $X_k \to X_n$ coincide for all $k \geq m$ (cf.\ \cite[p.\,138]{nswo}).
	An inverse system (of abelian groups) $(X_n)_{n \in \NN}$
	is called {\bfseries Mittag-Leffler zero} \index{Mittag-Leffler!zero}
	({\bfseries ML-zero}) if for any $n \in \NN$ there is an $m \geq n$ such that the transition map
	$X_{k} \to X_n$ is zero for all $k \geq m$ (cf.\ \cite[p.\,139]{nswo}).
	A morphism $(X_n)_n \to (Y_n)_n$ of inverse systems is called {\bfseries Mittag-Leffler isomorphism}
	\index{Mittag-Leffler!isomorphism} ({\bfseries ML-isomorphism}) if the corresponding systems
	of kernels and cokernels are ML-zero.\\
	By \gls{varprojlim} we denote the $r$-th right derived functor of $\varprojlim$.
	\end{mydef}
	
	\begin{prop} \label{galc1.15}
	Let $(X_n)$ and $(Y_n)$ be inverse systems of abelian groups.
	\begin{enumerate}
	\item If $(X_n)$ has surjective transition maps, then it is $ML$.
	\item If $(X_n)$ is ML then $\varprojlim^r_n X_n =0$ for all $r \geq 0$.
	\item If $f_n\colon X_n \to Y_n$ is a ML-isomorphism then for all $i \geq 0$
	the homomorphism
	\[ \begin{xy} \xymatrix@R-1pc{
	&\varprojlim_n^{i} f_n\colon \varprojlim_n^{i} X_n \ar[r] &\varprojlim_n^{i} Y_n
	} \end{xy} \]
	is an isomorphism.
	\end{enumerate}
	\begin{proof} $ $ \renewcommand{\theenumi}{{\arabic{enumi}.}}
	\begin{enumerate}
	\item Let $\alpha_{nm}\colon X_m \to X_n$ denote the transition map for $m \geq n$. Then it is
	$\im(\alpha_{nm})=X_n$ for all $m \geq n$, i.e. the system $X_n$ is ML.
	\item \cite[Chapter II \S 7, (2.7.4) Proposition, p.\,140]{nswo}
	\item \cite[Tag 0918, Lem.\  15.79.2.]{stacks-project}
	\end{enumerate}
	\end{proof}
	\end{prop}
	
	\begin{prop} \label{galc1.16}
	Let $(\Xb_n)$ and $(\Yb_n)$ be inverse systems of complexes of abelian groups such that the
	transition maps $X^{i}_{n+1} \to X^{i}_n$ and $Y^{i}_{n+1} \to Y^{i}_n$ are surjective for all
	$i \in \ZZ$ and $n \geq 0$.
	\begin{enumerate}
	\item For all $i \in \ZZ$ we get a short exact sequence
	\[ \begin{xy} \xymatrix{
		&0 \ar[r] & \varprojlim^1_n H^{i-1}(\Xb_n) \ar[r] & H^{i} (\varprojlim_n \Xb_n)
		\ar[r] & \varprojlim_n H^{i} (\Xb_n) \ar[r] & 0.
	} \end{xy} \]
	\item Let $(\fb_n)\colon (\Xb_n) \to (\Yb_n)$ be a morphism of inverse systems of complexes.
	If the induced map on cohomology $H^{i}(\fb_n)\colon H^{i}(\Xb_n) \to H^{i}(\Yb_n)$
	is a ML-isomorphism for all $i \in \ZZ$, then
	$\varprojlim_n (\fb_n)\colon \varprojlim_n \Xb_n \to \varprojlim_n \Yb_n$ is a quasi
	isomorphism.
	\end{enumerate}
	\begin{proof} $ $ \renewcommand{\theenumi}{{\arabic{enumi}.}}
	\begin{enumerate}
	\item \cite[Chapter 3, Proposition 1, p.\,531; Corollary1.1, p.\,535--536]{nohms}
	\item From the first part of the proposition we obtain for every $i \in \ZZ$
	a commutative diagram with exact rows
	\[ \begin{xy} \xymatrix{
		&0 \ar[r] & \varprojlim^1_n H^{i-1}(\Xb_n) \ar[r] \ar[d]^{\varprojlim^1_nH^{i-1}(\fb_n)}
		& H^{i} (\varprojlim_n \Xb_n) \ar[r] \ar[d]^{H^{i}(\varprojlim_n \fb_n)}
		& \varprojlim_n H^{i} (\Xb_n) \ar[d]^{\varprojlim_n H^{i}(\fb_n)} \ar[r]& 0 \\
		&0 \ar[r] & \varprojlim^1_n H^{i-1}(\Yb_n) \ar[r] & H^{i} (\varprojlim_n \Yb_n)
		\ar[r] & \varprojlim_n H^{i} (\Yb_n) \ar[r] & 0.
	} \end{xy} \]
	The assumption that $H^{i}(\fb_n)$ is a ML-isomorphism for all $i \in \ZZ$ then says that
	the left and the right horizontal maps in the above diagram are isomorphisms (cf.\
	\hyperref[galc1.15]{Proposition \ref*{galc1.15}}). The $5$-Lemma then implies that
	also $H^{i}(\varprojlim_n \fb_n)$ is an isomorphism for all
	$i \in \ZZ$, i.e. $\varprojlim_n \fb_n$ is a quasi isomorphism.
	\end{enumerate}
	\end{proof}
	\end{prop}

	\begin{rem} \label{galc1.16r}
	Since isomorphisms of inverse systems are always ML-isomorphisms, the above Proposition
	also states, that if $(\fb_n)\colon (\Xb_n) \to (\Yb_n)$ is a
	quasi isomorphism of inverse systems of complexes, for which the transition maps
	$X^{i}_{n+1} \to X^{i}_n$ and $Y^{i}_{n+1} \to Y^{i}_n$ are surjective for all
	$i \in \ZZ$ and $n \geq 0$, then also
	$\varprojlim_n (\fb_n)\colon \varprojlim_n \Xb_n \to \varprojlim_n \Yb_n$ is a quasi isomorphism.
	\end{rem}

	\begin{rem1}
	In the above \hyperref[galc1.16]{Proposition \ref*{galc1.16}} and \hyperref[galc1.16r]{Remark \ref*{galc1.16r}}
	one cannot easily drop the assumption that the transition maps are surjective. There exist examples (see \cite[Rem.\ 2.3.13.]{kupferer}) of two inverse systems of complexes which are quasi isomorphic, but their
	projective limits are not. Hence in the proof
	of \cite[Theorem 2.2.1, p.\,702--705]{scholl} right before \cite[Proposition 2.2.7, p.\,703--705]{scholl},   an explanation is missing why it really is enough to prove this proposition.
	\end{rem1}

\section{Lubin-Tate $(\varphi,\Gamma)$-modules} \label{chpgm}

	The goal in this section is to state in the Lubin-Tate case the equivalence of categories from \cite[Thm.\ 1.6]{kisren}, which follows closely   the original result \cite[3.4.3]{fo1} for
	$(\phi,\Gamma)$-modules in the cyclotomic case, in the style and with similar notation as in \cite{cher-col1999} or \cite[Theorem 4.22, p.\,82]{fouy}. Namely, if
	$K|L|\Qp$ are finite extensions, we want to describe an
	equivalence of categories between the category of continuous $\OL$-representations of the
	absolute Galois group $G_K$ and a yet to be defined category of \'etale $(\phL,\GaK)$-modules. While there is only a sketch of proof in \cite[Thm.\ 1.6]{kisren}, for the case $K=L$ a very detailed proof can be found in  the book \cite{schn1}. In his thesis \cite{kupferer} the first named author checked and comments how to adjust Schneider's proof in the general case.  One more
	useful source will be \cite{schol}.

	
	\subsection{Preparations and Notations} 

	Let $p$ be a prime number and let \gls{Qpalg} be a fixed algebraic closure of the $p$-adic
	numbers \index{pudic numbers@$p$-adic numbers}\gls{Qp} and let as usual \gls{Zp} be
	the integral $p$-adic numbers. \index{pudic numbers@$p$-adic numbers! integral} Each finite extension
	of $\Qp$ is considered to be a subfield of $\oQp$. Let $\Cp$ be the completion
	of $\oQp$ with respect to the
	valuation $v_p$ with $v_p(p)=1$ and let $\OCp$ be the ring of integers of $\Cp$.\\
	Let furthermore $L|\Qp$ be a finite extension, $d_L$ its degree over $\Qp$,
	\glssymbol{OK} the ring of integers, $\glssymbol{piK} \in \OL$ a prime element,
	\glssymbol{kK} the residue class field, $q_L=p^r$ its cardinality, $\Lur$ the maximal unramified extension
	of $\Qp$ in $L$ with ring of integers $\OLur$.\\
	Let furthermore $K|L$ be a finite extension, $d_K$ its degree over $\Qp$,
	$\OK$ its ring of integers, $\piK \in \OK$ a prime element, $\kK$ the residue class field, $\qK$ its
	cardinality, $\Kur$ the
	maximal unramified extension of $\Qp$ in $K$ with ring of integers $\OKur$. \\	
	We will denote the absolute Galois groups \index{Galois group! absolute}
	of $\Qp$, $L$ and $K$ by \glssymbol{GK}, $G_L$ and $G_K,$ respectively.\\
	For any $k_L$-algebra $B$  we will denote by $W(B)_L$ the ring of {\bfseries ramified Witt vectors} \index{Witt vectors! ramified} with values in $B.$
	(cf.\ \cite[Section 1.1, p.\,6--21]{schn1}).
	If $B$ happens to be perfect, these are standard Witt vectors tensored with $\OL$
	(cf.\ \cite[Proposition 1.1.26, p.\,23--24]{schn1}).\\
	A {\bfseries perfectoid} field \index{perfectoid field} $\mathcal{K} \subseteq \Cp$ is a complete field, such that
	its value group $|\mathcal{K}^{\times}|$ is dense in $\RR_{+}^{\times}$ and which satisfies
	$(\OO_{\mathcal{K}}/p\OO_{\mathcal{K}})^p = \OO_{\mathcal{K}}/p\OO_{\mathcal{K}}$
	(cf.\ \cite[p.\,42]{schn1}). \\
	Let $\mathcal{K}$ be a perfectoid field.
	The {\bfseries tilt} \index{tilt} \gls{Kb} of $\mathcal{K}$ is the fraction field of the ring
	\[ \OO_{\mathcal{K}^{\flat}}\coloneqq \varprojlim_{x \mapsto x^{\qL}}
	\OO_{\mathcal{K}}/ \varpi \OO_{\mathcal{K}},\]
	where $\varpi$ is an element in $\OO_{\mathcal{K}}$ such that $|\varpi| \geq |\piL|$. In fact, this
	definition is independent from the choice of the element $\varpi$ (cf.\ \cite[Lemma 1.4.5, p.\,43--44]{schn1}).
	The field $\mathcal{K}^{\flat}$ is perfect and complete and has characteristic $p$
	(cf.\ \cite[Proposition 1.4.7, p.\,45]{schn1}). Moreover, the field $\Cp^{\flat}$ is algebraically
	closed (cf.\ \cite[Proposition 1.4.10, p.\,46--47]{schn1}). The theory of perfectoid fields was originally
	established by Peter Scholze (cf.\ \cite{schol}) but Schneider's book covers all of the theory we do need
	here.	\\
	Let from now on, as in \cite[Definition 1.3.2, p.\,29]{schn1}, $\phi \in$ \gls{RX} be a fixed Frobenius power
	series \index{Frobenius power series} associated to $\piL$, i.e. we have
	\begin{align*}
			\phi(X) &\equiv \piL X \bmod \mathrm{deg} \ 2 \\
			\phi(X) &= X^{q_L} \bmod \piL \OLX.
	\end{align*}
	Let furthermore
	$\GG_\phi \in \OLXY$ be the Lubin-Tate formal group \index{Lubin-Tate formal group}
	which belongs to $\phi$ (cf.\ \cite[Proposition 1.3.4, p.\,31]{schn1}). For $a \in \OL$
	denote by $\gls{endphi} \in \OLX$ the corresponding endomorphism of $\GG_\phi$
	(cf.\ \cite[Proposition 1.3.6, p.\,32]{schn1}). Note that we then have\linebreak
	$[a]_\phi (X) \equiv aX \bmod \mathrm{deg} \ 2$ and
	$[\piL]_\phi=\phi$ (loc. cit.). We then set $\MM\coloneqq \{ x \in \oQp \mid |x|<1\}$ and obtain that
	the operation
	\[ \begin{xy} \xymatrix@R-2pc{
		& \OL \times \MM \ar[rr] & & \MM \\
		&(a, x) \ar@{|->}[rr] & & [a]_\phi (z)
	} \end{xy} \]
	makes $\MM$ into an $\OL$-module (cf.\ \cite[p.\,33]{schn1}). Then, for every $a \in \OL$, we can
	view $[a]_\phi$ as endomorphism of $\MM$ and therefore are able to define
	\[ \gls{GGnphi}\coloneqq \ker([\piL^n]_\phi\colon \MM \to \MM) = \{ x \in \MM \mid [\piL^n]_\phi(x)=0\}.\]
	Note that $(\GGphn)_n$ is via $[\piL]_\phi$ an inverse system and we let
	\[ \gls{TGphi}\coloneqq \varprojlim_n \GGphn\]
	be the projective limit of this system. (cf.\ \cite[p.\,50]{schn1}).
	$\TG_\phi$ is also called the
	{\bfseries Tate module} \index{Tate module} of the group $\GG_\phi$. From
	\cite[Proposition 1.3.10, p.\,34]{schn1} we can
	deduce that $\TG_\phi$ is a free $\OL$-module of rank one.\\
	Following \cite[(1.3.9), p.\,33]{schn1} we let $\gls{Ln}=L(\GG_\phi[\piL^n])$ and
	$\gls{Lin} = \cup_n L_n$. Denote as there the Galois group
	$\Gal(\Lin|L)$ by \gls{GaL}, set $\GaLn=\Gal(L_n|L)$ and $\gls{HL}=\Gal(\oQp|\Lin)$. Define furthermore
	$\gls{Kn}\coloneqq K(\GG_\phi [\piL^n])=KL_n$ and $\gls{Kin} \coloneqq\cup_n K_n=K\Lin$ as well  as
	$\gls{GaK}=\Gal(\Kin|K)$ and $\gls{HK}=\Gal(\oQp|\Kin)$.  These definitions can be summarized in
	the following diagram:
	
	\[ \xymatrix{
		& \oQp \ar@{-}[ddr]_{H_K} \ar@{-}[dddl]^{H_L} \ar@/_1.5cm/@{-}[dddddl]_{G_L}
		\ar@/^1.5cm/@{-}[ddddr]^{G_K}& \\
		& & \\
		&  & \Kin \ar@{-}[dll] \ar@{-}[dd]_{\Gamma_K}\\
		\Lin \ar@{-}[dd]^{\Gamma_L}& & \\
		& & K \ar@{-}[dll]\\
		L  & &
	}
	 \]
	
	\begin{rem} \label{pgm1.1}
	The group $\GaL$ is isomorphic to $\OL^\times$ via the Lubin-Tate character \index{Lubin-Tate character}
	\gls{chLT}. 	Furthermore, $\GaL$ acts continuously on $\TG_\phi$ via $\chL$, i.e. for all
	$\gamma \in \GaL$ and $t \in \TG_\phi$ we have
	\[ \gamma \cdot t = \chL(\gamma) \cdot t = [\chL(\gamma)]_\phi(t).\]
	\begin{proof}
	For the first assertion see \cite[(1.3.12),p.\,36]{schn1}, the second follows immediately from
	\cite[(1.3.11),p.\,34--35]{schn1} and is also stated at \cite[(1.4.17),p.\,51]{schn1}.
	\end{proof}
	\end{rem}
	
	\begin{rem} \label{pgm1.2}
	One can view $\GaK$ as an open subgroup of $\GaL$.
	If, in addition, $K|L$ is unramified, then we have $\GaK\cong\GaL$.
	\end{rem}

	\subsection{The coefficient ring}
	
	We first want to recall the definition of the coefficient ring used in \cite{schn1} and then
	deduce the coefficient ring in the general case.\\
	First we  recall the ring
	\[ \gls{AL}\coloneqq \varprojlim_n \OL/\piL^n \glssymbol{RXl},\]
	from \cite[p.\,75]{schn1}. This ring will be prototypical for our coefficients once we bring the
	variable $X$ to life.  $\AL$ carries an action from $\GaL$ by
	\[ \begin{xy} \xymatrix@R-2pc{
			&\GaL \times \AL \ar[r] &\AL \\
		&(\gamma,f)  \ar@{|->}[r] &f([\chL(\gamma)]_{\phi}(X)).
	} \end{xy} \]
	and possesses an injective $\OL$-algebra endomorphism \label{defphL}
	\[ \begin{xy} \xymatrix@R-2pc{
		& \phL\colon \AL \ar[r] &\AL\\
			&f \ar@{|->}[r] &f([\piL]_{\phi}(X))
	} \end{xy} \]
	(cf.\ \cite[p.\,78]{schn1}). At \cite[p.\,79]{schn1} Schneider defines a {\bfseries weak topology}
	\index{weak topology!on AL@on $\AL$} on $\AL$,
	for which the $\OLX$-submodules
	\[ U_m\coloneqq X^m\OLX + \piL^m\AL\]
	form a fundamental system of open neighbourhoods of $0 \in \AL$.
	 As $\phL(\AL)$-module $\AL$ is free with basis $1, X, \dots, X^{\qL-1}$ (\cite[Proposition 1.7.3, p.\,78]{schn1}),
	with respect to the weak topology $\AL$ is a complete Hausdorff topological $\OL$-algebra (\cite[Lemma 1.7.6, p.\,79--80]{schn1}) and both
the endomorphism $\phL$ and the $\GaL$-action are continuous for
	the weak topology (\cite[Proposition 1.7.8, p.\,80--82]{schn1}).

	Following  Colmez \cite[\S 9.2]{col6} one can
	find an element $\gls{omega} \in \OCpb$, such that
	$X \mapsto \omega$ defines an inclusion $\kL((X)) \hookrightarrow \Cpb$ (respecting important properties). As in
	\cite[p.\,50]{schn1} we denote the image of this inclusion by \gls{EEL} and we want to recall from loc. cit.
	that $\EEL$ is a complete nonarchimedean discretely valued field, with uniformizer $\omega$ and
	residue class field $\kL$. Let in addition \gls{EELp} denote the ring of integers inside $\EEL$.
	Furthermore, $\EEL$ carries a continuous operation by $\GaL$, for which
	we have $\gamma \cdot \omega = [\chL(\gamma)]_\phi (\omega) \bmod \piL$
	(cf.\ \cite[Lemma 1.4.15, p.\,51]{schn1}). By raising elements to its $\qL$-th power, it is clear that
	$\EEL$ also carries a Frobenius homomorphism, which is continuous and the reduction modulo $p$
	of $\phL$. Let furthermore
	\gls{EELsep} denote the separable closure of $\EEL$ inside $\Cpb$ and let \gls{EELsepp} denote
	the integral closure of $\EELp$ inside $\EELsep$.  	The Galois group
	$\Gal(\EELsep|\EEL)$ is isomorphic to $H_L$ by \cite[Section 1.6, p.\,68--75]{schn1} and \cite[Theorem 1.6.7, p.\,73--74]{schn1}.
	Then Schneider lifts $\omega$ to $W(\EEL)_L\subseteq W(\OCpb)_L$ and calls this lift \gls{omphi}
	(cf.\ \cite[Section 2.1, p.\,84--98; in particular p.\,93]{schn1}). Here one cannot just take the
	Teichm\"uller lift, because one wants that the lift fulfills the following relations
	\begin{align*}
		\gls{Fr}(\omphi)& = [\piL]_\phi(\omphi) \\
		\gamma \cdot \omphi &= [\chL(\gamma)]_\phi(\omphi)
	\end{align*}
	for all $\gamma \in \GaL$ and where $\mathrm{Fr}$ is the Frobenius on $W(\Cpb)_L$
	(cf.\ \cite[Lemma 2.1.13, p.\,92--93]{schn1} for the Frobenius and
	\cite[Lemma 2.1.15, p.\,95]{schn1} for the $\GaL$-action). Similar to the construction of $\EEL$,
	sending $X$ to $\omphi$ then defines an inclusion $\AL \hookrightarrow W(\EEL)_L$ (
	cf.\ \cite[p.\,94]{schn1}). In Particular, it
	gives us a commutative square (loc. cit.)
	\[ \begin{xy} \xymatrix{
		 & \AL \ar[rr]^-{X \mapsto \omphi} \ar[d] & & W(\EEL)_L \ar[d]\\
		 & \kL((X)) \ar[rr]^-{X \mapsto \omega} & & \EEL.
	} \end{xy} \]
\label{defAAL}Let \gls{AAL} denote the image
	of the inclusion $\AL \hookrightarrow W(\EEL)_L$. In addition, define
	\[ \gls{AALp} \coloneqq \OL \llbracket \omphi \rrbracket = \AAL \cap W(\EELp)_L.\]
	 $\AAL$ is endowed with the weak topology, i.e., induced by that from $W(\Cpb)_L$ and the isomorphism $\AL \cong \AAL$  is topological for the weak topologies on both
	sides (cf.\ \cite[Proposition 2.1.16, p.\,95--96]{schn1}). Furthermore, this topological
	isomorphism respects the $\GaL$-actions on both sides, where $\AAL$-carries a $\GaL$-action induced
	from the $G_L$-action of $W(\Cpb)_L$ (cf.\ \cite[p.\,94]{schn1}) and   what is $\phL$ on $\AL$ is
	the Frobenius on $\AAL$, which again is induced from the Frobenius on $W(\Cpb)_L$
	(cf.\ \cite[Proposition 2.1.16, p.\,95--96]{schn1}). We therefore denote the Frobenius on $\AAL$ also by
	$\phL$. An immediate consequence then is, that the $\GaL$-action and $\phL$ are continuous
	on $\AAL$.\\
	This   is the coefficient ring  for Schneider's $(\phL,\GaL)$-modules (cf.\
	\cite[Definition 2.2.6, p.\,100--101]{schn1}) but since we want to establish $(\varphi,\Gamma)$-modules
	over a finite extension $K|L$ as it was done in the classical way (cf.\ \cite[Definition 4.21, p.\,81]{fouy})
	for finite extensions of $\Qp$, we transfer this construction to our situation. Let for this
	$\gls{AALnr} \subseteq W(\EELsep)_L$ be the maximal unramified extension of $\AAL$ inside
	$W(\EELsep)_L$. In particular \cite[Lemma 3.1.3, p.\,112--113]{schn1} says that for every finite,
	separable extension $F|\EEL$ inside $\EELsep$, there exists a unique ring
	$\AAL(F) \subseteq W(\EELsep)_L$ containing $\AAL$ such that $\AALnr$ is the colimit
	of the family $\AAL(F)$. The ring \gls{AAA} is defined as the
	closure of $\AALnr$ inside $W(\EELsep)$ with respect to the $\piL$-adic topology and one has
	(cf.\ \cite[p.\,113 and Remark 3.14, p.\,114]{schn1}) \label{aaacompl}
	\[ \AAA \cong \varprojlim_n \AALnr/\piL^n\AALnr.\]
	  $\AALnr$ and $\AAA$ have an action from $G_L$,   the Frobenius
	on $W(\EELsep)$ preserves both rings,  they are discrete valuation rings with prime element $\piL$,
	where $\AAA$ is even complete and
	  their residue class field is $\EELsep$ (cf.\ \cite[p.\,113--114]{schn1}). \label{elsepresa}
	In fact, the $G_L$-action on both
	$\AALnr$ and $\AAA$ is continuous for the weak topologies, since the
	$G_L$ action on $W(\Cpb)_L$ is continuous for the weak topology
	(cf.\ \cite[Lemma 1.4.13, p.\,48--49]{schn1} and \cite[Lemma 1.5.3, p.\,65--66]{schn1}) and both, the weak
	topology and the $G_L$ action on $\AALnr$ respectively $\AAA$, are induced form $W(\Cpb)_L$.
Furthermore
	we have the relation (cf.\ \cite[Lemma 3.1.6, p.\,115--116]{schn1})
	\[ (\AAA)^{H_L} = \AAL.\]
	This   leads   to the definition
	\[ \AAK\coloneqq (\AAA)^{H_K}.\]
	In addition, define
	\begin{align*}
		\gls{AAAp} & \coloneqq \AAA \cap W(\EELsepp)_L \\
		\gls{AALnrp} & \coloneqq \AALnr \cap W(\EELsepp)_L\\
		\AAKp &\coloneqq \AAK \cap W(\EELsepp)_L.
	\end{align*}
	Then, since by definition it is $\AAL \subseteq \AAK \subseteq W(\EELsep)_L$, the ring $\AAK$ is
	a complete nonarchimedean discrete valuation ring with prime element $\piL$ and the restriction
	of the Frobenius from $W(\EELsep)_L$ gives a ring endomorphism $\phKL$ of $\AAK$, which then also
	commutes with $\phL$ (cf.\ \cite[Lemma 3.1.3, p.\,112--113]{schn1}).
	  Furthermore, since $\AAA$ carries an action from
	$G_L$ and therefore also one from $G_K$, the ring $\AAK$ carries an action from $\GaK$.
	Next, we want to define a weak topology on $\AAK$, deduce some properties and see that
	$\phKL$ and the action from $\GaK$ are continuous for this topology.
	
	\begin{mydef}
	The {\bfseries weak topology} on any of the rings $\AAA$, $\AALnr$, $\AAK$ and $\AAL$
	\index{weak topology!on AAA@on $\AAA$} \index{weak topology!on AALnr@on $\AALnr$}
	\index{weak topology!on AAK@on $\AAK$} \index{weak topology!on AAL@on $\AAL$} is
	defined as the induced topology of the weak topology of $W(\Cpb)_L$ (for the latter
	see \cite[p.\,64--65]{schn1}).
	\end{mydef}
	
	\begin{rem} \label{remweak}
	The weak topology on $W(\Cpb)_L$ is complete and Hausdorff (cf.\ \emph{\cite[Lemma 1.5.5, p.\,67--68]{schn1}})
	and $W(\Cpb)_L$ is a topological ring with respect to its weak topology
	(cf.\ \emph{\cite[Lemma 1.5.4, p.\,66--67]{schn1}}). Therefore, the induced topology on any
	of the rings $\AAA$, $\AALnr$, $\AAK$ and $\AAL$ is Hausdorff and these rings are topological rings.
	\end{rem}
	
	The question
	now is, wether $\phKL$ and the action from $\GaK$ are continuous for the weak topology on $\AAK$. For
	this, we want to recall a well-known fact.
	
	\begin{lem} \label{subtop}
	Let $X$ and $Y$ be topological spaces, $f\colon X \to Y$ be a continuous map and let
	$Z \subseteq Y$ be a subspace with $\im(f) \subseteq Z$. Then $f \colon X \to Z$ is continuous.
	\end{lem}
	
	\begin{prop}
	The from $W(\EELsep)_L$ induced $\GaK$-action and the induced Frobenius $\phKL$ on
	$\AAK$ are continuous.
	\begin{proof}
	This now is an immediate consequence of \hyperref[subtop]{Lemma \ref*{subtop}} and
	the fact, that $G_L$ acts continuously on $W(\EELsep)_L$ (cf.\ \cite[Lemma 1.5.3, p.\,65--66]{schn1})
	as well as that
	$\Fr$ is continuous on $W(\EELsep)$ with respect to the weak topology:\\
	Since the maps
	\[ \begin{xy} \xymatrix{
		& G_L \times \AAK \ar[r] & G_L \times W(\EELsep)_L \ar[r] & W(\EELsep)_L
	} \end{xy} \]
	and
	\[ \begin{xy} \xymatrix{
		& \AAK \ar@{^{(}->}[r] & W(\EELsep)_L \ar[r]^{\mathrm{Fr}} & W(\EELsep)_L		
	} \end{xy} \]
	are continuous as composite maps of continuous maps and their image is inside $\AAK$
	(for the latter see \cite[Lemma 3.1.3, p.\,112--113]{schn1}) the claim follows.
	\end{proof}
	\end{prop}

	We want to end this subsection by fixing some notation, defining weak topologies on modules over
	any of the above rings and calculating the residue class field of $\AAK$.
	We  denote
	by $\BBL$,   \gls{BBB},
	\glssymbol{BBK} and
	\gls{BBLnr} the quotient fields of $\AAL$, $\AAA$, $\AAK$ and $\AALnr$, respectively. Furthermore,
	set $\glssymbol{EEK}\coloneqq (\EELsep)^{H_K}$ and let $\glssymbol{EEKp}$ denote the integral closure of
	$\EELp$ inside $\EEL$. In \hyperref[resaak]{Lemma \ref*{resaak}} we will see that $\EEK$ is the
	residue class field of $\AAK$. Beforehand, we define weak topologies for modules.

	\begin{lem}
	Let $R \in \{ \AAA, \AALnr, \AAK, \AAL \}$ and $M$ be a finitely generated $R$-module. If
	$k,l \in \NN$ such that $R^k \twoheadrightarrow M$ and $R^l \twoheadrightarrow M$ are
	surjective homomorphisms, then the induced quotient topologies on $M$ coincide (where
	$R^k$ and $R^l$ carry the product topology of the weak topology on $R$).
	\begin{proof}
	This is \cite[Lemma 3.2.2 (i), p.\,100--102]{kley}.
	There, in fact, is no proof for $\AAK$, but in his proof, the author only uses that the coefficient ring is
	a topological ring with respect to the weak topology, what we stated in the above
	\hyperref[remweak]{Remark \ref*{remweak}}.
	\end{proof}
	\end{lem}
	
	\begin{mydef}
	Let $R \in \{ \AAA, \AALnr, \AAK, \AAL \}$ and $M$ be a finitely generated $R$-module. The
	{\bfseries weak topology} \index{weak topology!on modules} on $M$ is defined as the quotient topology for any
	surjective homomorphism $R^k \twoheadrightarrow M$, where $R^k$ carries the product topology
	of the weak topology on $R$.
	\end{mydef}
	
	\begin{lem}
	Let $R \in \{ \AAA, \AALnr, \AAK, \AAL \}$ and $M$ be a finitely generated $R$-module.
	Then $M$ with its weak topology is a topological $R$-module and if $M = M_1 \oplus M_2$, then
	the weak topology on $M$ coincides with the direct product of the weak topologies on the $M_1$ and
	$M_2$.\\
	Furthermore, if $N$ is another finitely generated $R$-module and $f\colon M \to N$ is an $R$-module
	homomorphism, then $f$ is continuous with respect to the weak topologies on both $M$ and $N$.
	\begin{proof}
	This is \cite[Lemma 3.2.2 (ii)-(iv), p.\,100--102]{kley}.
	Again, there is no proof for $\AAK$, but the property used is that of a discrete valuation ring,
	which $\AAK$ also fulfills.
	\end{proof}
	\end{lem}
	
	\begin{prop}[Relative Ax-Sen-Tate] \label{relast}
	Let $\CKK$ be a nonarchimedean valued field of characteristic $0$, $\OCKK$ an algebraic
	closure with completion
	$\CCC$ and $\CLL|\CKK$ a Galois extension within $\OCKK$ with completion $\HCLL$. Let furthermore
	$H \leq \GCLK$ be a closed subgroup. Then it holds
	\[ (\HCLL)^H = (\CLL^H)^{\wedge}.\]
	\begin{proof}
	This is an immediate consequence of the usual Ax-Sen-Tate theorem (cf.\
	\cite[Proposition 3.8, p.\,43--44]{fouy} ):
	Since $\CLL|\CKK$ is algebraic, $\OCKK$ is also an algebraic closure for $\CLL$ and then
	we deduce (loc. cit.)
	\[ \CCC^{\GCLL} = \HCLL.\]
	Infinite Galois theory then says that we have
	$ H = \Gal(\CLL|\CLL^H) \cong \GCLLH/\GCLL $.
	Together with Ax-Sen-Tate we then deduce
	\[ (\CLL^H)^{\wedge}= \CCC^{\GCLLH} = (\CCC^{\GCLL})^H = (\HCLL)^H.\]
	\end{proof}
	\end{prop}
	
	For our purposes the following integral version of the above Relative Ax-Sen-Tate Theorem will
	be the interesting one.
	
	\begin{cor} \label{intrelast}
	Let $\CKK$ be a nonarchimedean valued field of characteristic $0$, $\OCKK$ an algebraic
	closure with completion
	$\CCC$ and $\CLL|\CKK$ a Galois extension within $\OCKK$ with completion $\HCLL$. Denote by
	$\OO_{?}$ the ring of integers of any of the above fields $?$. Let furthermore
	$H \leq \GCLK$ be a closed subgroup. Then it holds
	\[ (\OO_{\HCLL})^H = ((\OO_{\CLL})^H)^{\wedge}.\]
	\begin{proof}
	For an element $x \in \CCC$ we have
	\[ \begin{xy} \xymatrix@C-0.9pc{
		x \in (\OO_{\HCLL})^H \ar@{<=>}[r] & x \in (\HCLL)^H \text{ with } |x| \leq 1
		\ar@{<=>}[r]^-{\hyperref[relast]{\ref*{relast}}} & x \in (\CLL^H)^{\wedge} \text{ with } |x| \leq 1
		\ar@{<=>}[r] & x \in ((\OO_{\CLL})^H)^{\wedge},
	} \end{xy} \]
	where the last equivalence comes from the fact that the integers of the completion are the completion
	of the integers.
	\end{proof}
	\end{cor}
	
	\begin{lem} \label{pgm2.4}
	It holds $(\AALnr)^{H_K} = \AAK$.
	\begin{proof}
	This is a direct consequence of the above \hyperref[intrelast]{Corollary \ref*{intrelast}}. This namely says that
	\[ \AAK = (\AAA)^{H_K} = ((\AALnr)^{H_K})^{\wedge}.\]
	But since $(\AALnr)^{H_K}|\AAL$ is finite and $\AAL$ is complete, $(\AALnr)^{H_K}$ itself is complete,
	i.e. it is
	\[ (\AALnr)^{H_K} = ((\AALnr)^{H_K})^{\wedge} = \AAK.\]
	\end{proof}
	\end{lem}
	
	\begin{lem} \label{resaak}
	$\EEK$ is the residue class field of $\AAK$.
	\begin{proof}
	We have an exact sequence
	\[ \begin{xy} \xymatrix{
		&0 \ar[r] & \AALnr \ar[r]^-{\cdot \piL} & \AALnr \ar[r] & \AALnr/\piL\AALnr \ar[r] &0.
	} \end{xy}\]
	By taking $H_K$-invariants and using $(\AALnr)^{H_K} = \AAK$ from
	\hyperref[pgm2.4]{Lemma \ref*{pgm2.4}} we obtain the exact sequence
	\[ \begin{xy} \xymatrix{
		&0 \ar[r] & \AAK \ar[r]^-{\cdot \piL} & \AAK \ar[r] & (\EELsep)^{H_K} \ar[r] & H^1(H_K,\AALnr).
	} \end{xy} \]
	Since $\BBLnr|\BBL$ is unramified, and therefore also tamely ramified, we get from
	\cite[(6.1.10) Theorem, p.\,342--342]{nswo} that $\AALnr$ is a cohomologically trivial
	$H_L$-module. Therefore the right term in the latter sequence is equal to zero and we get
	the exact sequence
	\[ \begin{xy} \xymatrix{
		&0 \ar[r] & \AAK \ar[r]^-{\cdot \piL} & \AAK \ar[r] & \EEK \ar[r] & 0
	} \end{xy} \]
	which ends the proof.
	\end{proof}
	\end{lem}

	\subsubsection{Concrete description of Weak Topologies}
	\label{secweaktop}

	As the title says, the goal of this chapter is to give a concrete description of both, the ring
	$\AAK$ and its weak topology. We will start with the topology and first we want the recall
	the description of the weak topology of $\AAL$ and recall that a similar description
	holds true on $W(\Cpb)_L$
	
	\begin{rem} \label{toponaal}
	\emph{\cite[Proposition 2.1.16 (i), p.\,95--96]{schn1}} says that the weak topology $\AAL$ has an analogous
	description as the description above. Concretely, a fundamental system of open neighbourhoods
	of $0$ for the weak topology on $\AAL$ is given by
	\[ \omphi^m \AALp + \piL^m \AAL, \ m \geq 1.\]
	\end{rem}

	\begin{rem} \label{pgm2.17}
	A fundamental system of open neighbourhoods of $0$ for the weak topology on
	$W(\Cpb)_L$ is given by the $W(\OCpb)_L$-submodules
	\[ \omphi^m W(\OCpb)_L + \piL^m W(\Cpb)_L, \ m \geq 1.\]
	\begin{proof}
	Because of $|\Phi_0(\omphi)|_\flat = |\omega|_\flat = |\piL|^{\qL/\qL-1} < 1$
	(cf.\ \cite[Lemma 2.1.13 (i), p.\,92--93]{schn1} for the first equality and
	\cite[Lemma 1.4.14, p.\,50]{schn1} for the second) this is exactly \cite[Remark 2.1.5 (ii), p.\,86--87]{schn1}.
	\end{proof}
	\end{rem}

	  In   the following Proposition we will show, that the above description of the weak topology on $\AAL$ extends to
	unramified, integral extensions. Its proof is a generalization of
	\cite[Proposition 2.1.16 (i), p.\,95--96]{schn1}.
	
	\begin{prop} \label{topunram}
	Let $B |\BBL$ be an unramified extension, $A \subseteq B$ the integral closure of $\AAL$ in $B$ and
	set $\Ap \coloneqq A \cap W(\EELsepp)_L$.\\
	Then the family
	\[ \omphi^m \Ap + \piL^m A, \ m \geq 1\]
	of $\Ap$-submodules of $A$ forms a fundamental system of open neighbourhoods of $0$ for
	the weak topology on $A$.
	\begin{proof}
	Since we have $\omphi^m \Ap \subseteq \omphi^m W(\OCpb)_L$ and
	$\piL^m A \subseteq \piL^m W(\Cpb)_L$ for all $m \geq 1$, we also get
	\[ \omphi^m \Ap + \piL^m A \subseteq (\omphi^m W(\OCpb)_L + \piL^m W(\Cpb)_L) \cap A\]
	for all $m \geq 1$, i.e. the topology on $A$ generated by the family
	$(\omphi^m \Ap + \piL^m A)_m$ is finer then the topology induced from $W(\Cpb)_L$.\\
	To see that it is also coarser, let $E |\EEL$ be the residue class field of $A$ and
	$\Epp$ be the integral closure of $\EELp$ in $E$ and consider the following families of
	$W(\OCpb)_L$-submodules of
	$W(\Cpb)_L$:
	\begin{align*}
	 \Vnm &\coloneqq \left\{(b_0,b_1,\dots) \in W(\OCpb)_L \mid b_0,\dots,b_{m-1} \in \omega^{n} \OCpb\right\}, \\
	 \Unm &\coloneqq \left\{(b_0,b_1,\dots) \in W(\Cpb)_L \mid b_0,\dots,b_{m-1} \in \omega^{n} \OCpb\right\}.
	 \end{align*}
	These are introduced in \cite[Section 1.5, p.\,64--68]{schn1} to define the weak topology on
	$W(\Cpb)_L$. In particular, the $\Unm$ give a fundamental system of open neighbourhoods of $0$ in
	$W(\Cpb)_L$ (loc. cit.) and the $\Vnm$ give one of $W(\OCpb)_L$. Since $\omphi$ is topologically
	nilpotent (cf.\ \cite[Lemma 2.1.6, p.\,87]{schn1}) we can find for any $k \in \NN$ an element
	$n \in \NN$ such that $\omphi^n \in \Vkm$. But since $\Phi_0(\omphi) = \omega$, i.e.
	$\omphi = (\omega,\dots)$, the condition $\omphi^n \in \Vkm$ implies $n \geq k$. Therefore we can find
	an increasing sequence of natural numbers $m \leq l_1 < \cdots < l_m$ such that
	\[ \omphi^{q_L^{l_i}} \in \Vqliem \ \text{ for all } 2 \leq i \leq m. \]
	Since $\Ap$ only contains positive powers of $\omphi$, this then implies, that
	for all $2 \leq i \leq m$ we have
	\[ \omphi^{q_L^{l_i}} \Ap \subseteq  \Vqliem. \]
	We will now show that
	\[ \Uqlmm \cap A \subseteq \omphi^m \Ap + \piL^m A.\]
	For this let $f_m \in \Uqlmm \cap A$. We then have
	\[ \Phi_0 (f_m) \in \omega^{q_L^{l_m}} \OCpb \cap E = \omega^{q_L^{l_m}} \Epp.\]
	Since by \cite[Lemma 3.1.3 (b), p.\,112--113]{schn1} the diagram
	\[ \begin{xy} \xymatrix{
		& A \ar@[(->][rr] \ar[dr]_{\mathrm{pr}} & & W(E)_L \ar[dl]^{\Phi_0} \\
			& & E &
	} \end{xy} \]
	commutes, we can find $g_m \in \omphi^{q_L^{l_m}} \Ap$ and $f_{m-1} \in A$ such that
	\[ f_m = g_m + \piL f_{m-1}.\]
	Recall $\omphi^{q_L^{l_m}} \Ap \subseteq  \Vqlmem$ from above and obtain
	\[ \piL f_{m-1} = f_m - g_m \in (\Uqlmm + \Vqlmem) \cap A = \Uqlmem \cap A. \]
	Then \cite[Proposition 1.1.18 (i), p.\,16--17]{schn1} says that, if $f_{m-1}=(b_0,b_1,\dots)$
	for some $b_j \in \Cpb$ then we have $ \piL f_{m-1} = (0,b_0^{q_L},b_1^{q_L},\dots)$. This then
	immediately implies $f_{m-1} \in \Uqlmmem \cap A$. This means that we can do a decreasing
	induction for $m \geq i \geq 1$ and find for every such $i$ elements
	$g_i \in \omphi^{q_L^{l_i}} \Ap$ and $f_{i-1} \in A$ such that
	\[ f_i = g_i + \piL f_{i-1}.\]
	Putting all this together, we get
	\[ f_m = \sum_{i=1}^m \piL^{m-i} g_m + \piL^m f_0. \]
	In particular we have
	\[ \sum_{i=1}^m \piL^{m-i} g_m \in \omphi^{q_L^{l_1}} \Ap \subseteq \omphi^m \Ap. \]
	Therefore we have $f_m \in \omphi^m \Ap + \piL^m A$ which was exactly the statement we wanted to
	see to end the proof.
	\end{proof}
	\end{prop}

	
	\begin{cor} \label{toponaak}
	A fundamental system of open neighbourhoods of $0$ for the weak topology on
	$\AAK$ (resp. $\AALnr$) is given by the $\AAKp$- (resp. $\AALnrp$-) submodules
	\begin{align*}
		\omphi^m \AAKp + \piL^m \AAK, \ &m \geq 1, \text{ respectively} \\
		\omphi^m \AALnrp + \piL^m \AALnr, \ &m \geq 1.
	\end{align*}
	\begin{proof}
	This is an application of \hyperref[topunram]{Proposition \ref*{topunram}}.
	\end{proof}
	\end{cor}

	\begin{prop} \label{topaakmod}
	The weak topology on $\AAK$ coincides with the weak topology of $\AAK$ considered as
	$\AAL$-module.
	\begin{proof}
	If $(u_i)_i$ is an $\AAL$-basis of $\AAK$, then $(\omphi^ku_i)_i$ is so for all $k \geq 0$. Therefore
	$\AAK$ has an $\AAL$-basis consisting of elements of $\AAKp$. The claim then follows from
	the above \hyperref[toponaak]{Corollary \ref*{toponaak}} together with
	\hyperref[toponaal]{Corollary \ref*{toponaal}}.
	\end{proof}
	\end{prop}
	
	\begin{prop} \label{topaakaaa}
	The canonical inclusion $\AAK \hookrightarrow \AAA$ is a topological embedding. Furthermore, for
	every $n \in \NN$ the induced inclusion $\AAK/\piL^n\AAK \hookrightarrow \AAA/\piL^n\AAA$ is a
	topological embedding as well.
	\begin{proof}
	Because of
	\[ \AAK \cap \AAA = \AAK \cap \AAA \cap W(\OCpb)_L = \AAK \cap W(\OCpb)_L \]
	the first part of the assertion follows from the definition of the weak topology. The second then follows from
	the commutative diagram
	\[ \begin{xy} \xymatrix{
			& \AAK\ar@{->>}[d] \ar@{^{(}->}[r] & \AAA \ar@{->>}[d]  \\
			 & \AAK/\piL^n\AAK \ar@{^{(}->}[r] & \AAA/\piL^n\AAA.
	} \end{xy} \]
	\end{proof}
	\end{prop}

	\begin{prop} \label{topaaa}
	The weak topology on $\AAA$ coincides with the topology of the projective limit
	$\varprojlim_n \AALnr/\piL^n\AALnr$ where each factor carries the quotient topology of the
	weak topology on $\AALnr$.
	Moreover, a fundamental system of open neighbourhoods of $0$ for the weak topology on $\AAA$ is
	given by the sets
	\[ \omphi^m\AALnrp + \piL^m \AAA, \ m \geq 1.\]
	Note that, by definition, $\AAAp = \AALnrp$.
	\begin{proof}
	For this proof, we will refer to the latter topology of the Proposition's formulation as the
	\emph{projective limit topology}.\\
	As in the above \hyperref[topaakaaa]{Proposition \ref*{topaakaaa}} the inclusion
	$\AALnr \hookrightarrow \AAA$ clearly is a topological embedding and since the diagram
	\[ \begin{xy} \xymatrix{
			& \AALnr\ar@{->>}[d] \ar@{^{(}->}[r] & \AAA \ar@{->>}[d]  \\
			 & \AALnr/\piL^n\AALnr \ar@{=}[r] & \AAA/\piL^n\AAA.
	} \end{xy} \]
	for every $n \in \NN$ is commutative, the quotient topology on $\AALnr/\piL^n\AALnr$ with respect to
	the weak topology on $\AALnr$ coincides with its quotient topology with respect to the weak
	topology on $\AAA$. Therefore the canonical projections
	\[ \begin{xy} \xymatrix{
		\AAA = \varprojlim_n \AALnr/\piL^n\AALnr \ar@{->>}[r] & \AALnr/\piL^n\AALnr
	} \end{xy} \]
	are continuous for the weak topology on $\AAA$. This means that the weak topology of $\AAA$ is
	finer than its projective limit topology.\\
	From \hyperref[topunram]{Proposition \ref*{topunram}} we deduce that a fundamental system of
	open neighbourhoods of $0$ for the quotient topology of the weak topology on
	$\AALnr/\piL^n\AALnr$ is given by the sets
	\[ \omphi^m \AALnrp + \piL^n\AALnr, \ m \geq 1.\]
	Then the sets
	\[ \omphi^m\AALnrp + \piL^n \AAA, \ m,n \geq 1\]
	form a fundamental system of open neighbourhoods of $0$ for the projective limit topology on $\AAA$. But
	clearly the sets with $m=n$ define the same topology. Since the weak topology is defined by the sets
	\[ \left( \omphi^m W(\OCpb)_L + \piL^m W(\Cpb)_L\right) \cap \AAA, \ m \geq 1\]
	(cf.\ \hyperref[pgm2.17]{Remark \ref*{pgm2.17}}) and we clearly have
	\[ \omphi^m \AALnrp + \piL^m \AAA \subseteq \left( \omphi^m W(\OCpb)_L + \piL^m W(\Cpb)_L\right) \cap \AAA\]
	for all $m \geq 1$, the projective limit topology is finer than the weak topology.
	\end{proof}
	\end{prop}
	
	\begin{lem} \label{finlaur}
	Let $k$ be a finite field and $E|k((X))$ be a finite, separable extension. Then there exists a finite
	extension $\kappa|k$ and $Y\in E$ such that $E \cong \kappa((Y))$.
	\begin{proof}
	This is \cite[Lemma 1.38, p.\,20]{kup}.
	\end{proof}
	\end{lem}

	\begin{lem} \label{pgm2.13}
	Let $k'|k$ be an extension of finite fields and $k'((Y))|k((X))$ be a finite, separable extension. Then
	the $Y$-adic and the $X$-adic topologies on $k'((Y))$ coincide.\\
	In particular, there exists an $l \in \NN$ such that for all $n \in \NN$ it holds
	\[ X^{ln} \kpY \subseteq Y^{ln} \kpY \subseteq X^n \kpY.\]
	\begin{proof}
	Since $\kX$ is a discrete valuation ring with respect to its $X$-adic topology and $\kpY$ is so as
	well with respect to its $Y$-adic topology, we deduce from usual ramification theory, that there exists
	an $l \in \NN$ such that
	\[ Y^l \kpY = X\kpY.\]
	Since $Y\kpY$ is the maximal ideal of $\kpY$ it clearly is $X \kpY \subseteq Y \kpY$ and therefore
	we get for all $n \in \NN$
	\[ X^{ln} \kpY \subseteq Y^{ln} \kpY \subseteq X^n \kpY.\]
	\end{proof}
	\end{lem}
	
	\begin{lem} \label{topexel}
	Let $E|\EEL$ be a finite and separable extension. Then
	the subspace topology on $E$ induced from the topology of $\Cpb$ coincides with the
	extension from the $\omega$-adic topology on $\EEL$.
	Note that the latter topology is the $\omega$-adic topology on $E$, due to the above
	\emph{\hyperref[pgm2.13]{Lemma \ref*{pgm2.13}}}.\\
	In particular, the integral closure $E^+$ of $\EELp$ inside $E$ consists of exactly those elements of $E$
	whose absolute value in $\Cpb$ is less or equal to $1$.
	\begin{proof}
	We denote the absolute value induced from $\Cpb$ by $|\cdot|_\flat$ as in
	\cite[Lemma 1.4.6, p.\,44--45]{schn1} and we use the identifications
	$E \cong \kappa((Y))$ as well as $\EEL\cong \kL((X))$ (cf.\ \hyperref[finlaur]{Lemma \ref*{finlaur}}), where
	$\kappa|\kL$ is a finite extension.\\
	The maximal unramified intermediate field of $\kappa((Y))|\kL((X))$ is $\kappa((X))$ and therefore it
	exists an $l \in \NN$ and $g_i \in \kapX$ for $0 \leq i < l$ with $X \mid g_i$ and $X^2 \nmid g_0$ such that
	(cf.\ \cite[Chapter I, \S 6, Proposition 17, p.\,19]{ser1})
	\[ \sum_{i=0}^{l-1} g_i Y^{i} + Y^l = 0.\]
	Since $|X|_\flat <1$ and $|x|_\flat =1$ for $x \in \kappa$ (in particular, every nonzero element coming
	from a finite field has absolute value $1$ in $\OCpb$ with respect to $|\cdot|_\flat$)
	we have $|g_i|_\flat \leq 1$ for all $0 \leq i < l$ and we can deduce
	\[ |Y^l|_\flat \leq \underset{0 \leq i <l}{\mathrm{max}} |g_i|_\flat |Y^{i}|_\flat
			\leq \underset{0\leq i < l}{\mathrm{max}} |Y^{i}|_\flat\]
	and therefore $|Y|_\flat \leq 1$. Furthermore, since we have $Y^l\kapY = X\kapY$ we can find a
	$g \in \kapY$ such that $Y^l = X g$ and since $|Y|_\flat \leq 1$ we then deduce $|g|_\flat \leq 1$ and
	\[ |Y^l|_\flat = |X|_\flat |g|_\flat \leq |X|_\flat <1.\]
	But this then immediately implies
	\[ |Y|_\flat < 1. \]
	Since $X\mid g_0$ and $X^2\nmid g_0$ it is $|g_0|_\flat = |X|_\flat $ and because of $X \mid g_i$ for all
	$0 \leq i <l$ we also have
	\[ |g_0|_\flat \geq |g_i|_\flat \ \text{ for all }\ 0 < i < l.\]
	Since $|Y|_\flat <1$ we deduce from the above
	\[ |g_0|_\flat > |g_i|_\flat|Y^{i}|_\flat \ \text{ for all }\ 0 < i < l\]
	and therefore
	\[ |Y^l|_\flat = |g_0|_\flat = |X|_\flat \]
	because $|\cdot|_\flat$ is a nonarchimedean absolute value.\\
	Denote by $|\cdot|$ the extension of the absolute value of $\EEL$ (which corresponds with
	the $\omega$-adic topology) to $E$. Then we deduce from \cite[Chapter 2, \S 2, Corollary 4, p.\,29]{ser1}
	that
	\[ |Y^l| = |\mathrm{Nor}(Y)|_\flat,\]
	where $\mathrm{Nor}$ denotes the norm of the extension $\kappa((Y))|\kappa((X))$. From the polynomial
	we started with we then can deduce $\mathrm{Nor}(Y)=g_0$ and therefore
	\[ |Y^l| = |g_0|_\flat = |X|_\flat.\]
	This means that $|\cdot|$ and $|\cdot |_\flat $ coincide on $E$.\\
	From the identification above we deduce $E^+=\kapY$. But since $|Y|_\flat <1$, these are exactly the
	elements of $E$ whose absolute value is less or equal to $1$.
	\end{proof}
	\end{lem}
	
	\begin{cor}
	$\EEK$ is, with respect to the topology induced from $\Cpb$, a complete, nonarchimedean
	discretely valued field of characteristic $p$ with residue class field $\kK$ and ring of integers
	$\EEKp$.
	\end{cor}
	
	\begin{lem} \label{topcol}
	Let $X$ be a topological space and $(Y_n)_n$ a family of subsets of $X$ with
	$Y_{n} \subseteq Y_{n+1}$. Set $Y \coloneqq \varinjlim_n Y_n=\bigcup_n Y_n$.
	Then, the subset topology on $Y$
	coincides with the final topology of the inductive limit with respect to the subset topologies on the $Y_n$.
	\begin{proof}
	First we show that the canonical injections $f_n \colon Y_n \hookrightarrow Y$ are continuous
	for the subset topology on $Y$. This then implies that the subspace topology on $Y$ is coarser than
	the projective limit topology since the latter is the finest topology such that all injections $f_n$ are continuous
	(cf.\ \cite[Chapter I, \S 2.4, Proposition 6, p.\,32]{bour4}).\\
	Let $U \subseteq Y$ be open and $V \subseteq X$ open such that $U = V \cap Y$. Then it
	is
	\[ f_n^{-1}(U) = U \cap Y_n = V \cap V \cap Y_n = V \cap Y_n, \]
	i.e. $f_n^{-1}(U) \subseteq Y_n$ is open.\\
	It is left to show, that the subspace topology is finer then the direct limit topology. For this,
	let $U \subseteq Y$ be open with respect to the direct limit topology, i.e. it is
	$U = \bigcup_n f_n^{-1}(U)$, where for every $n \in \NN$ it exists an open $V_n \subseteq X$ such that
	$f^{-1}(U)=V_n \cap Y_n$. We set $V \coloneqq \bigcup V_n$ and claim $U = V \cap Y$. To see this,
	let $u \in U$. Then it exists $n \in \NN$ such that $u \in V_n \cap Y_n$ and in particular
	$u \in V$. Conversely let $u \in V \cap Y$. Then, by definition, there exist $n_1,n_2 \in \NN$ such that
	$u \in V_{n_1}$ and $u \in Y_{n_2}$. For $n \coloneqq \mathrm{max}\{n_1,n_2\}$ we then deduce
	$u \in V_n \cap Y_n$ and therefore $u \in U$.
	\end{proof}
	\end{lem}

	\begin{prop} \label{pgm2.16}
	The integral closure $\EELsepp$ of $\EELp$ inside $\EELsep$ consists of exactly those
	elements with absolute value
	$|\cdot|_\flat$ less or equal to $1$. \\
	Furthermore, the topology on $\EELsep$ induced from $\Cpb$ coincides with the
	final topology with respect to the colimit
	\[ \EELsep = \bigcup_{\underset{\mathrm{\ fin, \ sep}}{E|\EEL}} E\]
	where each $E$ carries the topology induced from $\Cpb$. \\
	In particular, the $\EELsepp$-submodules
	\[ \omega^n \EELsepp \]
	form a fundamental system of open neighbourhoods of $0$ for this topology on $\EELsep$.
	\begin{proof}
	This now is an immediate consequence of \hyperref[topexel]{Lemma \ref*{topexel}} and
	\hyperref[topcol]{Lemma \ref*{topcol}}.
	\end{proof}
	\end{prop}

	

%
%
		
	\subsubsection{Structure of Coefficient Rings (unramified case)} 
	\label{secstructunram}
	
	For this subsection, let $K|L$ be an unramified extension.
	Then this a Galois extension and its Galois group is isomorphic to
	the Galois group of the respective residue class fields. It therefore is
	cyclic and generated by the lift of the $\qL$-Frobenius $x \mapsto x^{\qL}$. We will denote
	this lift by \glssymbol{sigKL} and call it Frobenius on $K$. Recall also from
	\hyperref[pgm1.2]{Remark \ref*{pgm1.2}} that the groups $\GaL$ and $\GaK$ are isomorphic and
	for every $n \in \NN$ the groups $\GaLn$ and $\Gamma_{K_n|K}$ are isomorphic as well.
	
	
	\begin{rem} \label{indh}
	We have $(H_L:H_K)=[K:L]$.
	\begin{proof}
	Since $\GaL \cong \GaK$ (cf.\ \hyperref[pgm1.2]{Remark \ref*{pgm1.2}}) we have
	$ (H_L:H_K)=[\Kin:\Lin]=[K:L]$.
	\end{proof}
	\end{rem}
	
	\begin{lem} \label{resun}
	We have $\kK\EEL = \EEK$.
	\begin{proof}
	Since $\kK$ is fixed by $H_K$ it clearly is $\kK \EEL \subseteq \EEK$. Since $K|L$ is unramified
	we have $[K:L]=[\kK:\kL]$ and
	therefore
	\[ [\kK\EEL : \EEL]=[\kK:\kK\cap\EEL]=[\kK:\kL]=[K:L]=(H_L:H_K)=[\EEK:\EEL].\]
	\end{proof}
	\end{lem}
	
	\begin{lem} \label{descak}
	We have $\AAK=\OK\otimes_{\OL} \AAL$ and $\BBK=K\BBL$.
	\begin{proof}
	Since $K|L$ is unramified $\OK \otimes_{\OL} \AAL$ is unramified over $\AAL$ and since $K$ is fixed by
	$H_K$ we deduce $\OK \otimes_{\OL} \AAL \subseteq \AAK$. Since both are free $\AAL$-modules of
	rank $[K:L]=(H_L:H_K)$ they coincide.\\
	The statement for the fields of fractions then follows immediately.
	\end{proof}
	\end{lem}
	
	In order to understand how the operations of $\GaK$ and the Frobenius look on $\AAK$ respectively
	$\BBK$ it now suffices to understand the corresponding operations on $\OK$ respectively $K$. Note, that
	since $K|L$ is unramified, we clearly have $W(\kK)_L=\OK$.
	
	\begin{lem} \label{frobonk}
	Let $\Fr$ denote the (restriction of the) $\qL$-Frobenius on $\kK$. Then
	the automorphism $\sigKL$ on $\OK$ coincides with the restriction of $W(\Fr)_L$.
	\begin{proof}
	Due to the functoriality of the Witt construction, $W(\Fr)_L$ is an automorphism on $\OK$ which fixes
	$\OL$, it induces also an automorphism on $K$ which fixes $L$ and it's reduction modulo
	$\piL$ is $\Fr$. The
	first observation says, that the restriction of $W(\Fr)_L$ is an element of $\Gal(K|L)$ and
	since $\Gal(K|L)$ and $\Gal(\kK|\kL)$ are
	isomorphic via $\sigma \mapsto \sigma \bmod \piL$, the second observation says that the restriction of
	$W(\Fr)_L$ is a lift of $\Fr$. Since this lift is unique we get the desired equality
	$W(\Fr)_L=\sigKL$ on $K$ respectively $\OK$.
	\end{proof}
	\end{lem}
	
	Before we give explicit descriptions of the operations on $\AAK$ we want to fix some
	notation.

	\begin{mydef} \label{endonak}
	Let $\vartheta$ be an $\OL$-linear endomorphism of $\OK$ and $f \in \AAK$ we denote by
	\gls{fthet} the element, on which $\vartheta$ is applied to the coefficients of $f$, i.e.
	if $f(\omphi) = \sum a_i\omphi^{i}$ then
	\[ \fthet(\omphi)=\sum_{i \in \NN_0} \vartheta(a_i)\omphi^{i}.\]
	\end{mydef}

	\begin{prop} \label{descopak}
	Let $f = f(\omphi) = \sum a_i \omphi^{i} \in \AAK$ and $\gamma \in \GaK$. We then have
	\[ \gamma \cdot f = \sum_{i \in \ZZ} a_i [\chLT(\gamma)]_\phi(\omphi^{i}).\]
	For the Frobenius $\phKL$ we have
	\[ \phKL(f) = \sum_{i \in \ZZ} \sigKL(a_i) [\piL]_\phi(\omphi^{i}).\]
	Together with the above \emph{\hyperref[endonak]{Definition \ref*{endonak}}}, we then have the
	description
		\[ \phKL(f(\omphi))=\fsig(\phKL(\omphi)).\]
	\begin{proof}
	This is an immediate consequence of \hyperref[pgm1.2]{Remark \ref*{pgm1.2}},
	\hyperref[descak]{Lemma \ref*{descak}} and \hyperref[frobonk]{Lemma \ref*{frobonk}}.
	\end{proof}
	\end{prop}

	\subsubsection{Structure of Coefficient Rings (general case)} 
	\label{secstructgen}	
	%
%
%

	\begin{prop} \label{finAL}
	Let $B|\BBL$ be a finite, unramified extension and $A \subseteq B$ be the integral closure of
	$\AAL$. Then there exists a finite, unramified extension $E|L$ and an element
	$\nuphi \in W(\EELsep)_L$ with $\nuphi^j=\omphi$ for some $j > 0$ such that
	\[ A \cong \varprojlim_n \OO_E/\piL^n\OO_E ((\nuphi)).\]
	\begin{proof}
	Let $\kappa$ be the residue class field of $A$ and recall that the residue class field of $\AAL$ is
	$\EEL=\kL((\omega))$. Since $B|\BBL$ is unramified, we then have
	\[ [B:\BBL] = [\kappa:\kL((\omega))].\]
	Since $\kappa|\kL((\omega))$ is finite and separable ($B|\BBL$ is unramified), we deduce from
	\hyperref[finlaur]{Lemma \ref*{finlaur}} that $\kappa \cong k ((\nu))$ for some finite extension
	$k|\kL$ and $\nu \in \EELsep$ with $\nu^{j}=\omega$ for some $j>0$.
	But then there exists a unique finite and unramified extension $E|L$ with $k_E=k$.
	In particular, we have $j [E:L] = [B:\BBL]$. Furthermore, since $\BBL$ is a complete discrete valuation field,
	and $B|\BBL$ is a finite extension,
	the henselian lemma in the sense of \cite[II \S 4, (4.6) Henselsches Lemma, p.\,135--136]{neuk1} holds true
	and therefore we can find a $\nuphi \in B \subseteq W(\EELsep)_L$ which is a
	root of the polynomial $X^j-\omphi$ and for which we have
	\[ \nuphi \bmod \piL = \nu.\]
	Since $X^j-\omphi$ is irreducible over $\BBL$ and $E|L$ is unramified, we deduce
	\begin{align*}
		[E\BBL(\nuphi):\BBL] & = [E\BBL(\nuphi) : \BBL(\nuphi)] \cdot [\BBL(\nuphi):\BBL]  \\
						&=[E : E \cap \BBL(\nuphi)] \cdot j \\
					&= [E:L] \cdot j
	\end{align*}
	and therefore
	\[ B = E \BBL(\nuphi).\]
	In particular, we have
	\begin{align*}
	B = \Bigg\{\sum_{i \in \ZZ} a_i \nuphi^{i}\, \bigg| \, & a_i \in E, \ \lim_{i \to - \infty} a_i = 0
	\text{ and it exists } n \in \NN\\
	& \text{ such that } \piL^na_i \in \OO_E \text{ for all } i \in \ZZ \Bigg\}
	\end{align*}
	Then $A$ consists of those elements of $B$ with $\piL$-adic absolute value $\leq 1$. Since
	this absolute value is nonarchimedean, these are exactly those elements \linebreak
	$\sum_{i \in \ZZ} a_i \nuphi^{i} \in B$ with
	$a_i \in \OO_E$ for all $i \in \ZZ$, i.e.
	$ A \cong \varprojlim_n \OO_E/\piL^n\OO_E ((\nuphi)).$
	\end{proof}
	\end{prop}

	\subsection{$(\phKL,\GaK)$-modules and Galois representations}

	If not otherwise stated, all continuity statements refer to the corresponding
	weak topology.
	
	\begin{mydef}
	Let $M$ be an $\AAK$-module. We regard $M$ as a left-$\AAK$-module and $\AAK$ itself as a
	right-$\AAK$-module via $\phKL$. For the tensor product in this situation we write
	$\AAK \glssymbol{tenphaaK} M$, which is per definition an abelian group, but since $\AAK$ is also a left-
	$\AAK$-module (with the standard multiplication) this tensor product is also a (left)-$\AAK$-module.
	\end{mydef}

	\begin{mydef}
	Let $M$ be a finitely generated $\AAK$-module equipped with a $\phKL$-linear endomorphism $\phM$.
	Then \glssymbol{phMlin} denotes the homomorphism
	\[ \begin{xy} \xymatrix@R-2pc{
		& \phMlin\colon \AAK \tenphaaK M \ar[rr] & & M \\
			&f \otimes m \ar@{|->}[rr] & & f \phM(m).
	} \end{xy} \]
	\end{mydef}

	\begin{mydef} \label{defpgmod}
	A finitely generated $\AAK$-module $M$ is called $(\phKL,\GaK)${\bfseries -module}
	\index{phklgakmodule@ $(\phKL,\GaK)$-module} if it is equipped
	with a $\phKL$-linear endomorphism $\phM$ and a continuous, semilinear action from $\GaK$, which
	commutes with the endomorphism $\phM$.
	A $(\phKL,\GaK)$-module is called {\bfseries \etale}
	\index{phklgakmodule@ $(\phKL,\GaK)$-module! \'{e}tale @\etale} if the homomorphism $\phMlin$ is bijective.\\
	A morphism of $(\phKL,\GaK)$-modules $f\colon M\to N$ is an $\AAK$-module homomorphism,
	which respects the actions from $\GaK$ and the endomorphisms $\phM$ and $\phN$.
	We will denote the category of \etale \ $(\phKL,\GaK)$-modules by \glssymbol{ModpGKe}.
	\end{mydef}

	\begin{thm} \label{equivcat}
	The exact tensor  categories $\RepGKOLfg$ and $\ModpGKe$ are equivalent to each other. The equivalence
	is given by the quasi inverse functors
	\[ \begin{xy} \xymatrix@R-2pc{
		 & \MAK\colon \RepGKOLfg \ar[r] & \ModpGKe \\
		 & V \ar@{|->}[r] & \left( \AAA \otimes_{\OL} V \right)^{H_K}
	} \end{xy} \]
	and
	\[ \begin{xy} \xymatrix@R-2pc{
		& \VK \colon \ModpGKe \ar[r] & \RepGKOLfg \\
		& M \ar@{|->}[r] & \left( \AAA \otimes_{\AAK} M \right)^{\mathrm{Fr} \otimes \phM = 1}.
	} \end{xy} \]
	\end{thm}

\section{Iwasawa Cohomology } \label{cwh}

	In this section   we want to
	list the results from \cite[Section 5, p.\,23--31]{SV15} which we will need later on. Note that \cite[Remark 5.1, p.\,23]{SV15} was also
	proven here (cf.\ \hyperref[galc1.1]{Lemma \ref*{galc1.1}}).
	
	\begin{mydef}
	Let $M$ be a topological $\OL$-module.  The {\bfseries Pontrjagin dual} \index{dual!Pontrjagin} of $M$ is defined
	as
	\[ \gls{Mvee}\coloneqq \HomcOL(M,L/\OL)\cong \HomcOK(M,K/\OK).\]
	It is always equipped with the compact-open topology.
	\end{mydef}
	
	\begin{prop}[Pontrjagin duality] \label{pontrjdual}
	The functor $-^{\vee}$ defines an involuntary contravariant autoequivalence of the
	category of (Hausdorff) locally compact linear-topological $\OL$-mo\-du\-les.
	In particular, for such a module $M$ there is a canonical isomorphism
	\[ M \overset{\cong}{\longrightarrow} (M^{\vee})^{\vee}.\]
	\begin{proof}
	This is explained in \cite[Proposition 5.4, p.\,25--26]{SV15}.
	\end{proof}
	\end{prop}
	
	\begin{rem}
	Let $M_0 \overset{\alpha}{\to} M \overset{\beta}{\to}M_1$ be a sequence of locally compact linear-topological $\OK$-modules such that $\im(\alpha)=\ker(\beta)$ and $\beta$ is topologically strict with
	closed image. Then the dual sequence
	\[ M_1^{\vee}\overset{\beta^{\vee}}{\to} M^{\vee} \overset{\alpha^{\vee}}{\to} M_0^{\vee}\]
	is exact as well.
	\begin{proof}
	The proof is similar to the one of \cite[Remark5.5, p.\,27]{SV15}.
	\end{proof}
	\end{rem}

	\begin{rem} \label{remdual}
	Let $V \in \RepGKOLfg$ of finite length and $n\geq 1$ such that $\piL^nV=0$. Then there is a
	natural isomorphism of topological groups:
	\[ \MAK(V)^{\vee} \cong \MAK(V^{\vee}(\chLT)).\]
	This isomorphism identifies $\psi_{\MAK(V^{\vee}(\chLT))}$ with $\varphi_{\MAK(V)}^{\vee}$.
	\begin{proof}
	This is \cite[Remark 5.6, p.\,27]{SV15}
	\end{proof}
	\end{rem}

	\begin{prop}[Local Tate duality] \label{loctat7}
	Let $V \in \RepGKOLfg$, $n \geq 1$ such that $\piL^nV=0$ and $E$ a finite extension of
	$K$. Then the cup product and the local invariant map induce perfect pairings of
	finite $\OL$-modules
	\[ H^{i}(G_E,V) \times H^{2-i}(G_E,\Hom_{\Zp}(V,\Qp/\Zp(1))) \to H^2(G_E,\Qp/\Zp(1))=\Qp/\Zp\]
	and
	\[ H^{i}(G_E,V) \times H^{2-i}(G_E,\Hom_{\OL}(V,L/\OL(1))) \to H^2(G_E,L/\OL(1))=L/\OL.\]
	There $-(1)$ denotes the twist by the cyclotomic character.\\
	This means that there are conical isomorphisms
	\[ H^{i}(G_E,V) \cong H^{2-i}(K,V^{\vee}(1))^{\vee}.\]
	\begin{proof}
	This is \cite[Proposition 5.7, p.\,27--28]{SV15}, where \cite[Theorem 2, p.\,91--92]{ser3} is applied.
	\end{proof}
	\end{prop}
	
	\begin{mydef} \label{iwcohdef}
	Let $V \in \RepGKOLfg$. The {\bfseries generalized Iwasawa cohomology of} \index{cohomology! Iwasawa}
	V is defined by
	\[ \glssymbol{HIW}(\Kin|K,V)\coloneqq\varprojlim_{K \subseteq E \subseteq \Kin} H^{i}(G_E,V).\]
	We always consider these modules as $\GaK$-modules.
	\end{mydef}

	\begin{rem} \label{Iwfinext}
	Let $E|K$ be a finite extension contained in $\Kin$. Then there is an isomorphism of $\OL$-modules:
	\[ \varprojlim_{E \subseteq E' \subseteq \Kin} H^{i}(G_{E'},V) \cong \HIW^{i}(\Kin|K,V).\]
	\begin{proof}
	The claim follows immediately from the fact, that the set
	$\{E'|E \text{ finite} \mid E'\subseteq \Kin \}$ is cofinal in the set $\{E'|K \text{ finite} \mid E'\subseteq \Kin \}$.
	\end{proof}
	\end{rem}
	
	\begin{lem} \label{loctat8}
	Let $V \in \RepGKOLfg$. Then we have
	\[ \HIW^{i}(\Kin|K,V)\cong H^{i}(G_K,\OL\llbracket \GaK \rrbracket \otimes_{\OL} V).\]
	\begin{proof}
	The proof is similar to the one of \cite[Lemma 5.8, p.\,28--29]{SV15}.
	\end{proof}
	\end{lem}
	
	\begin{lem}
	$V \mapsto \HIW(\Kin|K,V)$ defines a $\delta$-functor on $\RepGKOLfg$.
	\begin{proof}
	Replace $\GaL$ by $\GaK$ in the proof of \cite[Lemma 5.9, p.\,29]{SV15}.
	\end{proof}
	\end{lem}
	
	\begin{rem} \label{loctat10}
	Let $V, V_0  \in \RepGKOLfg$ such that $V_0$ is $\OL$-free and $G_K$ acts through
	its factor $\GaK$ on $V_0$. Then there is a natural isomorphism
	\[ \HIW^{i}(\Kin|K,V \otimes_{\OL} V_0) \cong \HIW^{i}(\Kin|K,V) \otimes_{\OL} V_0.\]
	\end{rem}
	
	\begin{rem}
	Let $V \in \RepGKOLfg$ be of finite length. Then there is an isomorphism
	\[ \HIW^{i}(\Kin|K,V) \cong H^{i}(H_K,V^{\vee}(1))^{\vee}.\]
	Note that $H_K=G_{\Kin}$.
	\begin{proof}
	From \hyperref[loctat7]{Proposition \ref*{loctat7}} we deduce
	\[ H^{i}(G_{K_n},V) \cong H^{2-i}(G_{K_n},V^{\vee}(1))^{\vee}\]
	for every $n \in \NN$. Taking projective limits gives us
	\begin{align*}
		\HIW^{i}(\Kin|K,V)&= \varprojlim H^{i}(G_{K_n},V) \\
				&= \varprojlim H^{2-i}(G_{K_n},V^{\vee}(1))^{\vee} \\
				&= \varprojlim \HomcOL ( H^{2-i}(G_{K_n},V^{\vee}(1)),L/\OL) \\
				&= \HomcOL(\varinjlim H^{2-i}(G_{K_n},V^{\vee}(1)),L/\OL) \\
				&= \HomcOL(H^{2-i}(\varprojlim G_{K_n},V^{\vee}(1)),L/\OL) \\
				&=\HomcOL(H^{2-i}(H_K,V^{\vee}(1)),L/\OL) \\
				&= H^{2-i}(H_K,V^{\vee}(1))^{\vee}.
	\end{align*}
	\end{proof}
	\end{rem}
	
	\begin{lem} $ $
	\begin{enumerate}[itemsep=0pt,topsep=0pt]
	\item $\HIW^{i}(\Kin|K,V)=0$ for $i \neq 1,2$.
	\item $\HIW^2(\Kin|K,V)$ is finitely generated as $\OL$-module.
	\item $\HIW^1(\Kin|K,V)$ is finitely generated as $\OL\llbracket \GaK \rrbracket$-module.
	\end{enumerate}
	\begin{proof}
	The proof is similar to the one of \cite[Lemma 5.12, p.\,29--30]{SV15}.
	\end{proof}
	\end{lem}
	
	\begin{thm}[{\cite[Theorem 5.13]{SV15}}] \label{loctat13}
	Let $V $ be in $\RepGKOLfg$, $\tau=\chcyc\chLT^{-1}$ and $\psi=\psi_{\MAK(V(\tau^{-1}))}$. Then we
	have an exact sequence
	\[ 0 \longrightarrow \HIW^1(\Kin|K,V) \longrightarrow \MAK(V(\tau^{-1})) \overset{\psi-1}{\longrightarrow}
		\MAK(V(\tau^{-1})) \longrightarrow \HIW^2(\Kin|K,V) \longrightarrow 0,\]
	which is functorial in $V$.
	Furthermore, each occurring map is continuous and
	$\OL \llbracket \GaK \rrbracket$-equivariant.
	\end{thm}

\begin{rem}
A version in the derived category is shown in \emph{Proposition \ref{iwpsi}} and, unfortunately, is not a direct consequence of this theorem, of course.
\end{rem}

\section{Galois cohomology in terms of Lubin-Tate $(\varphi,\Gamma)$-modules}
\label{galco}

	We keep the notation from \hyperref[chpgm]{Section \ref*{chpgm}}. Recall from
	\hyperref[elsepresa]{Theorem \ref*{elsepresa}}
	(resp. from \cite[p.\,113-114]{schn1}) that $\EELsep$ is the residue class field of $\AAA$
	and $\EELsepp$ is the residue class field of $\AAAp$.

	\subsection{Description with $\varphi$} 

	The goal of this subsection is to compute Galois cohomology from the generalized
	$\varphi$-Herr complex, which is related to $\phKL$ and $\GaK$.

	\renewcommand{\theenumi}{\emph{\arabic{enumi}.}}

	\begin{lem} \label{galc1.1} $ $
	\begin{enumerate}[itemsep=0pt,topsep=0pt]
	\item The following sequences are exact:
	\[ \begin{xy} \xymatrix{
		0 \ar[r] & \OL \ar[r] & \AAA \ar[rr]^-{\Fr-\id} & & \AAA \ar[r] &0. \\
		0 \ar[r] & \OL \ar[r] & \AAAp \ar[rr]^-{\Fr-\id} & & \AAAp \ar[r] &0.
	} \end{xy} \]
	\item Let $E \mid L$ be a finite extension. For every $n \in \NN$ the maps
	\[ \begin{xy} \xymatrix{
		&\phEL- \id \colon \omphi^n\EEp_E \ar[r] & \omphi^n \EEp_E, \\
		&\Fr-\id \colon  \omphi^n \EELsepp \ar[r] & \omphi^n\EELsepp
	} \end{xy} \]
	are isomorphisms.
	\item For every $n \in \NN$ the map
	\[ \begin{xy} \xymatrix{
	&\Fr - \id \colon \omphi^n\AAAp \ar[r] & \omphi^n\AAAp
	} \end{xy} \]
	is an isomorphism.
	\end{enumerate}
	\begin{proof} $ $ \renewcommand{\theenumi}{{\arabic{enumi}.}}
	\begin{enumerate}[itemsep=0pt,topsep=0pt]
	\item We start with the sequence
	\[ \begin{xy} \xymatrix{
		0 \ar[r] & \kL \ar[r] & \EELsep \ar[rr]^-{x \mapsto x^{q_L}-x} & & \EELsep \ar[r] & 0,
	} \end{xy} \]
	and claim that it is exact. Recall that $\Fr(x) \equiv x^{q_L} \bmod \piL$ holds for all $x \in \AAA$
	by definition. The inclusion $\OL \hookrightarrow \AAA$ induces the
	inclusion $\kL \hookrightarrow \EELsep$ and we have
	\[ \ker(\Fr- \id) = \{ x \in \EELsep \mid x^{q_L} - x\} = \kL.\]
	It remains to check that $\Fr-\id$ is surjective on $\EELsep$. But since the polynomial \linebreak
	$X^{q_L}-X-\alpha$ is separable for every $\alpha \in \EELsep$ and $\EELsep$ is separably closed
	by definition this follows immediately.\\
	Now suppose that the sequence
	\[ \begin{xy} \xymatrix{
		0 \ar[r] & \OL/\piL^n\OL \ar[r] & \AAA/\piL^n\AAA \ar[rr]^-{\phL-\id} & &
		\AAA/\piL^n\AAA \ar[r] & 0
	} \end{xy} \]
	is exact for $n \geq 1$ and consider the following commutative diagram
	\[ \begin{xy} \xymatrix{
		0 \ar[r] & \OL/\piL^n\OL \ar[r] & \AAA/\piL^n\AAA \ar[rr]^-{\phL-\id} & &
		\AAA/\piL^n\AAA \ar[r] & 0 \\
		0 \ar[r] & \OL/\piL^{n+1}\OL \ar[r]  \ar@{->>}[u]& \AAA/\piL^{n+1}\AAA
		\ar[rr]^-{\phL-\id} \ar@{->>}[u]& & \AAA/\piL^{n+1}\AAA \ar[r] \ar@{->>}[u] & 0.
	} \end{xy} \]
	Our aim is to show that the second sequence is exact. The kernel of the homomorphism
	$\OL \hookrightarrow \AAA \twoheadrightarrow \AAA/\piL^{n+1}\AAA$ is $\piL^{n+1}\OL$, i.e.
	we have exactness at the first position. Since we have $\phL(x) = x$ for all $x \in \OL$, we also have
	$\OL/\piL^{n+1}\OL \subseteq \ker(\Fr-\id)$. So let $x \in \AAA$ such that $\Fr(x) - x \equiv 0 \bmod
	\piL^{n+1} \AAA$. Then we also have $\Fr(x) - x \equiv 0 \bmod \piL^{n} \AAA$ and because the first
	sequence is exact, we obtain a $y \in \OL$ such that $y \equiv x \bmod \piL^n \AAA$. Then there is
	an $\alpha \in \AAA$ such that $x-y = \piL^n\alpha$, especially we have
	$x-y \equiv \piL^n\alpha \bmod \piL^{n+1}\AAA$. Since $\Fr(X) \equiv X^{\qL} \bmod \piL$
	we get $\Fr(\alpha) \equiv \alpha^{\qL} \bmod \piL\AAA$ and
	therefore \linebreak $\Fr(\piL^n \alpha) \equiv \piL^n \alpha^{\qL} \bmod \piL^{n+1}\AAA$ since $\Fr$
	is $\OL$-linear. Then we also get
	\[ 0 \equiv (\Fr-\id)(x-y) \equiv (\Fr-\id)(\piL^n\alpha) \equiv \piL^n(\alpha^{\qL} - \alpha)
	\bmod \piL^{n+1}\AAA.\]
	Since $\AAA$ is a domain, this then implies $\alpha^{\qL} \equiv \alpha \bmod \piL\AAA$. Since the
	sequence in question is exact for $n=1$ by the start of the proof, we then get a
	$z \in \OL$ such that $z \equiv \alpha \bmod\piL \AAA$, i.e. it exists $\beta \in \AAA$ such that
	$\alpha = z +\piL\beta$. We then get
	\[ x \equiv y + \piL^n\alpha \equiv y + \piL^n(z+\piL\beta) \equiv y + \piL^nz \bmod \piL^{n+1} \AAA,\]
	i.e. $\ker(\Fr-\id) \subseteq \OL/\piL^{n+1}\OL$.\\
	It remains to check that $\Fr -\id$ is surjective on $\AAA/ \piL^{n+1} \AAA$. So let $x \in \AAA$.
	Because $\Fr - \id$ is surjective on $\AAA /\piL^n \AAA $ we get a $y \in \AAA$ such that
	$\phL(y) - y \equiv x \bmod \piL^n \AAA$. As before there is now an $\alpha \in \AAA$ such that
	$\phL(y) - y \equiv x + \piL^n \alpha \bmod \piL^{n+1} \AAA$. Again, since the sequence for $n=1$ is
	exact we can find $z \in \AAA$ such that \linebreak $\phL(z)-z \equiv \alpha \bmod \piL\AAA$ and therefore
	we can find $\beta \in \AAA$ such that \linebreak $\phL(z)-z + \piL\beta = \alpha$. We then get
	\begin{align*}
	\Fr(y-\piL^nz)-(y-\piL^nz) &= \Fr(y)-y - \piL^n(\Fr(z)-z) \\&\equiv x + \piL^n\alpha - \piL^n\alpha +
	\piL^{n+1}\beta \equiv x \bmod \piL^{n+1}\AAA,
	\end{align*}
	i.e. $y-\piL^nz$ is $\bmod \piL^{n+1}\AAA$ a preimage of $x$ under $\phL-\id$.\\
	Since the transition maps $\OL/\piL^{n+1}\OL \to \OL/\piL^n\OL$ are surjective, the inverse system
	$(\OL/\piL^n\OL)_n$ is a Mittag-Leffler System and therefore we have
	$\varprojlim^1 \OL/\piL^n\OL=0$ (cf.\ \hyperref[galc1.13]{Remark \ref*{galc1.13}}). By taking
	the inverse limit of the sequence
	\[ \begin{xy} \xymatrix{
		0 \ar[r] & \OL/\piL^n\OL \ar[r] & \AAA/\piL^n\AAA \ar[rr]^-{\Fr-\id} & &
		\AAA/\piL^n\AAA \ar[r] & 0
	} \end{xy} \]
	we then get the exact sequence
	\[ \begin{xy} \xymatrix{
		0 \ar[r] & \OL \ar[r] & \AAA \ar[rr]^-{\Fr-\id} & &
		\AAA \ar[r] & 0.
	} \end{xy} \]
	
	The proof of the exactness of the second sequence is similar to the prove above. Just replace
	$\EELsep$ by $\EELsepp$ which is the separable closure of $\EEp_L$
	in $\EELsep$.
	\item As before we have $\Fr(x) = x^{q_L}$ for all $x \in \EELsep$. Especially this equation holds for
	elements in $\EELsepp$ and $\EEp_E$. The injectivity of the above maps then is easy
	to see: \\
	Let $0 \neq x \in \omphi^n\EELsepp$. So, in particular we have $\deg_{\omphi}(x) \geq n >0$ and therefore also
	$\deg_{\omphi}(\Fr(x)) > \deg_{\omphi} (x)$, i.e. $\Fr(x) - x \neq 0$ and so $\Fr-\id$ is injective
	on $\omphi^n\EEp$.
	Because of $\EEp_E\subseteq \EELsep$ the homomorphism $\phEL-\id$ is also injective on $\omphi^n \EEp_E$.\\
	For the surjectivity let $\alpha$ be an element of $\omphi^n\EELsepp$ or of $\omphi^n\EEp_E$. Then
	the series $(\Fr(\alpha)^{i})_i$ converges to zero and therefore
	\[ \beta:=\sum_{i=0}^{\infty} - \Fr(\alpha)^{i} \]
	is also an element of $\omphi^n\EELsepp$ or of $\omphi^n\EEp_E$ and clearly is a preimage of
	$\alpha$ under $\Fr-\id$.
	\item Let $n,l \in \NN$ be fixed and note that there is a canonical identification
	\[ (\omphi^n\AAAp)/(\piL^l\omphi^n\AAAp)\cong \omphi^n(\AAAp/\piL^l\AAAp)\]
	since $\omphi^n$ is not a zero divisor in both $\AAAp$ and $\AAAp/\piL^l\AAAp$. Now assume that
	\[ \begin{xy} \xymatrix{
		&\Fr-\id \colon \omphi^n (\AAAp/\piL^k \AAAp) \ar[r] & \omphi^n(\AAAp/\piL^k\AAAp)
	}\end{xy}\]
	for all natural numbers $k \leq l$ is an isomorphism. Note that we just proved this for $l=1$. Consider the
	commutative diagram:
	\[ \begin{xy} \xymatrix{
		&\Fr-\id \colon \omphi^n (\AAAp/\piL^l \AAAp) \ar[r] & \omphi^n(\AAAp/\piL^l\AAAp) \\
		&\Fr-\id \colon \omphi^{n} (\AAAp/\piL^{l+1} \AAAp) \ar[r] \ar@{->>}[u]
		& \omphi^{n}(\AAAp/\piL^{l+1}\AAAp) \ar@{->>}[u]
	}\end{xy}\]
	Our aim is to show, that the latter horizontal homomorphism is also an isomorphism.\\
	Let $x \in \AAAp$ such that $\omphi^n x \not\equiv 0 \bmod \piL^{l+1}\AAAp$. The degree $n$-term
	(with respect to $\omphi$) of
	$\Fr(\omphi^nx)-\omphi^nx$ is $\omphi^n(\piL-1)x$ and therefore it is unequal to zero modulo
	$\piL^{l+1}$. To see this, we assume $\omphi^n(\piL-1)x \equiv 0 \bmod \piL^{n+1}$ and let $j$
	be the smallest integer such that $2^j \geq n+1$ and multiply this congruence with
	$(1+\piL)(1+\piL^2)\cdots(1+\piL^{2^{j-1}})$. Then we get
	\[ 0 \equiv \omphi^n(\piL^j-1)x \equiv - \omphi^n x \bmod \piL^{n+1}\AAAp\]
	what we excluded, i.e. it has to be $\omphi^n(\piL-1)x \not\equiv 0 \bmod \piL^{n+1}\AAAp$ and therefore
	$\Fr-\id$ is injective on $\omphi^n(\AAAp/\piL^{l+1} \AAAp)$.\\
	Let $x \in \omphi^n\AAAp$. Then there exists $y \in \omphi^n\AAAp$ such that
	$\phL(y)-y \equiv x \bmod \piL^l\AAAp$ (because we
	assumed the surjectivity for all values $\leq l$), i.e.
	there exists $\alpha \in \omphi^n\AAAp$ such that $\Fr(y)-y=x+\piL^l\alpha$. Then again there exists
	$\beta \in \omphi^n\AAAp$ such that $\Fr(\beta)-\beta \equiv \alpha \bmod \piL$, i.e. there
	exists some $\eta \in \omphi^n\AAAp$ such  that $\Fr(\beta)-\beta=\alpha + \piL\eta$. We then get
	\begin{align*} (\Fr-\id)(y -\piL^l\beta) & = (\Fr-\id)(y) - \piL^l(\Fr-\id)(\beta) \\
		&= x + \piL^l\alpha - \piL^l(\alpha+\piL\eta) \equiv x \bmod \piL^{l+1}\AAAp,
	\end{align*}
	i.e. the map $\Fr-\id$ is surjective on $\omphi^n(\AAAp/\piL^{l+1} \AAAp)$. Since these maps are all
	isomorphisms, passing to the projective limit gives that the map $\Fr-\id$ is an isomorphism on
	$\omphi^n\AAAp$
	\end{enumerate}
	\end{proof}
	\end{lem}
	
	\begin{cor} \label{galc1.2}
	For every $n \in \NN$ the following sequence is exact:
	\[ \begin{xy} \xymatrix{
		0 \ar[r] & \OL \ar[r] & \AAA/\omphi^n\AAAp \ar[rr]^-{\Fr-\id} & & \AAA/\omphi^n \AAAp \ar[r] &0.
	} \end{xy} \]
	\begin{proof}
	In \hyperref[galc1.1]{Lemma \ref*{galc1.1}} we showed that
	\[ \begin{xy} \xymatrix{
		0 \ar[r] & \OL \ar[r] & \AAA \ar[rr]^-{\Fr-\id} & & \AAA \ar[r] &0.
	} \end{xy} \]
	is an exact sequence and that
	\[ \begin{xy} \xymatrix{
		&\Fr-\id \colon  \omphi^n \AAAp \ar[r] & \omphi^n \AAAp
	} \end{xy} \]
	is an isomorphism for every $n \in \NN$. Since every element of the image of
	$\OL \hookrightarrow \AAA$ has degree $0$ (with respect to $\omphi$) the homomorphism
	$\OL \to \AAA/\omphi^n \AAAp$ is still injective. Since $\Fr$ fixes $\OL$ it is clear that we have
	$\OL \subseteq \ker(\Fr-\id)$. For the other inclusion let $x \in \ker(\Fr-\id)$. Then there exists
	an $\alpha \in \AAA$ such that $\alpha \bmod \omphi^n\AAAp = x$ and
	$\Fr(\alpha)-\alpha \in \omphi^n\AAAp$. But since $\Fr-\id$ is an isomorphism
	on $\omphi^n\AAAp$ there exists also a $\beta \in \omphi^n\AAAp \subseteq \AAA$ such that
	$\Fr(\beta)-\beta=\Fr(\alpha)-\alpha$.  Because of the exactness of
	\[ \begin{xy} \xymatrix{
		0 \ar[r] & \OL \ar[r] & \AAA \ar[rr]^-{\Fr-\id} & & \AAA \ar[r] &0
	} \end{xy} \]
	it then exists $ \eta \in \OL$ such that $\eta = \alpha-\beta$. This implies
	$\eta \equiv \alpha \bmod \omphi^n\AAAp$, i.e. $\eta = x$ which means $\ker(\Fr-\id) \subseteq \OL$.
	This proves the exactness in the middle. For the surjectivity of $\Fr-\id$ recall that
	$\AAA \twoheadrightarrow \AAA/\omphi^n\AAAp$ and
	$\Fr-\id \colon \AAA \to \AAA$ are surjective and consider the commutative diagram
	\[ \begin{xy} \xymatrix{
		&\AAA \ar@{->>}[rr]^{\Fr-\id} \ar@{->>}[d] & & \AAA \ar@{->>}[d] \\
		& \AAA/\omphi^n\AAAp \ar[rr]_{\Fr-\id} & & \AAA/\omphi^n\AAAp.
	} \end{xy} \]
	This implies that the homomorphism
	$\Fr-\id\colon \AAA/\omphi^n\AAAp \to \AAA/\omphi^n\AAAp$ is also surjective.
	\end{proof}
	\end{cor}

	\begin{lem} \label{galc1a.3}
	Let $A|\AAA_L$ be a finite, unramified extension. Then, for every $m \in \NN$, the canonical
	projection $A/\piL^{m+1} A \to A/\piL^m A$ has a
	continuous, set theoretical section with respect to the weak topology on $A$.
	\begin{proof}
	From \hyperref[finAL]{Proposition \ref*{finAL}} we deduce that
	\[ A \cong \varprojlim_n \OO_E/\piL^n \OO_E((X)) \]
	for some finite, unramified extension $E|L$. Therefore we have
	\[ A/\piL^m A = \OO_E/\piL^m\OO_E((X))\]
	for every $m \in \NN$. Therefore it is enough to give a continuous set theoretical section of the
	canonical projection  $\OO_E/\piL^{m+1}\OO_E((X)) \to \OO_E/\piL^m\OO_E((X))$ with respect
	to the $X$-adic topology.
	Since the $\OO_E/\piL^m\OO_E$ are finite discrete, there exists for
	every $m \in \NN$ a continuous map
	\[ \begin{xy} \xymatrix{
		& \iota_m\colon \OO_E/\piL^m\OO_E \ar[r] & \OO_E/\piL^{m+1}\OO_E
	} \end{xy} \]
	which is a set theoretical section of the canonical projection. We then define a map
	\[ \begin{xy} \xymatrix@R-2pc{
		& \alpha_m\colon \OO_E/\piL^m\OO_E((X)) \ar[rr] & & \OO_E/\piL^{m+1}\OO_E ((X)), \\
			&\sum_{i >> -\infty} \lambda_i X^{i} \ar@{|->}[rr] & & \sum_{i >>-\infty} \iota_m(\lambda_i) X^{i}.
	} \end{xy} \]
	This then clearly is a set theoretical section of the canonical projection. We have to check continuity.\\
	So let $f \in \OO_E/\piL^{m+1} \OO_E ((X))$ and $n \in \NN_0$. If then
	$\alpha_m^{-1}(f + X^n \OO_E/\piL^{m+1} \OO_E \llbracket X \rrbracket)$ is empty, there is nothing to prove.
	So assume there is $g \in \alpha_m^{-1}(f + X^n \OO_E/\piL^{m+1} \OO_E \llbracket X \rrbracket)$ and let
	$h \in X^n \OO_E/\piL^{m} \OO_E \llbracket X \rrbracket$. Then $g$ and $g+h$ coincide in degrees
	$< n$ and therefore, by definition, also $\alpha_m(g)$ and $\alpha_m(g+h)$ coincide in
	degrees $< n$, i.e.
	\[ \alpha_m(g+h) \in \alpha_m(g) + X^n \OO_E/\piL^{m+1} \OO_E \llbracket X \rrbracket
			= f+ X^n \OO_E/\piL^{m+1} \OO_E \llbracket X \rrbracket\]
	since $\alpha_m(g) \in f+ X^n \OO_E/\piL^{m+1} \OO_E \llbracket X \rrbracket$. It then follows
	\[  g + X^n \OO_E/\piL^{m} \OO_E \llbracket X \rrbracket
		\subseteq \alpha_m^{-1}(f + X^n \OO_E/\piL^{m+1} \OO_E \llbracket X \rrbracket)\]
	and therefore that $\alpha_m$ is continuous.
	\end{proof}
	\end{lem}
	
	\begin{cor} \label{galc1a.4}
	For every $m \in \NN$ the canonical projection
	$\AAA/\piL^{m+1} \AAA \to \AAA/\piL^m \AAA$ has a continuous, set theoretical section.
	\begin{proof}
	Since $\AAA$ is the $\piL$-adic completion of $\AALnr$ it is
	\[ \AAA/\piL^m\AAA = \AALnr/\piL^m \AALnr\]
	for every $m \in \NN$. Since colimits are exact it is
	\[ \AALnr/\piL^m\AALnr = \bigcup_{A|\AAA_L \text{ fin, nr}} A/\piL^mA\]
	for every $m \in \NN$ and since we have for every $A|\AAA_L$ finite and unramified and every
	$m \in \NN$ a
	continuous, set theoretical section of the canonical projection \linebreak
	$A/\piL^{m+1}A \to A/\piL^mA$ (cf.\ \hyperref[galc1a.3]{Lemma \ref*{galc1a.3}})
	this induces for every $m \in \NN$ a set theoretical section of the canonical projection
	$\AALnr/\piL^{m+1} \AALnr \to \AALnr/\piL^m\AALnr$, which then is continuous, since
	$\AALnr$ carries the topology of the colimit and then so does
	$\AALnr/\piL^m\AALnr$ for every $m \in \NN$.
	\end{proof}
	\end{cor}

	\begin{lem} \label{galc1a.5}
	Let $V \in \RepGKOLfg$, set $M := \MOEK(V)$ and $V_m:= V/\piL^mV$ as well as
	$M_m:=M/\piL^mM$ for $m \in \NN$.
	Then the transition maps of the inverse systems $(V_m)_m$, $(M_m)_m$
	and $(\AAA \otimes_{\OL} V_m)_m$ are surjective and they have a continuous, set
	theoretical section. In particular, the short sequences
	\[ \begin{xy} \xymatrix@R-2pc{
		& 0 \ar[r] & \AAA \otimes_{\OL} V_m \ar[r]^-{\id \otimes \cdot \piL} & \AAA \otimes_{\OL} V_{m+1}
		\ar[r] & \AAA \otimes_{\OL} V_1 \ar[r] & 0, \\
		&0 \ar[r] & M_m \ar[r]^-{\cdot \piL} & M_{m+1} \ar[r] & M_1 \ar[r] & 0
	} \end{xy} \]
	are exact and have continuous, set theoretical sections.
	\begin{proof}
	Since $\MOEK$ is exact as an equivalence of categories (cf.\ \hyperref[equivcat]{Theorem \ref*{equivcat}}) and
	the tensor product is right exact, it is immediately clear that the transition maps
	of the systems $(M_m)_m$ and $(\AAA \otimes_{\OL} V_m)_m$ are surjective since the transition maps
	of $(V_m)_m$ are.\\
	Since the $V_m$ are finite and discrete one can define a set theoretical section
	of the canonical projection $V_{m+1} \to V_m$ by choosing a preimage for every element in $V_m$.
	Since $M_m$ is a finitely generated $\AAK$-module, there are for every $m \in \NN$ isomorphisms of
	topological $\AAK$-modules
	\[ M_m \cong \bigoplus_{i=1}^{n^{(m)}} \AAK/\piL^{n_i^{(m)}} \AAK\]
	such that $n_i^{(m)} \leq n_{i+1}^{(m)}$ and the canonical projection $M_{m+1} \to M_m$ maps
	the $i$-th component of $\oplus_{i=1}^{n^{(m+1)}} \AAK/\piL^{n_i^{(m+1)}} \AAK$ to the $i$-th
	component of $\oplus_{i=1}^{n^{(m)}} \AAK/\piL^{n_i^{(m)}} \AAK$ for $i \geq n^{(m)}$ and is
	zero on the $i$-th component with $i > n^{(m)}$. With \hyperref[galc1a.3]{Lemma \ref*{galc1a.3}}
	we then obtain a continuous, set theoretical section for every component, which then also gives a
	continuous set theoretical section for $M_{m+1} \to M_m$.\\
	As topological $\OL$-module we have
	\[ \AAA \otimes_{\OL} V_m \cong \bigoplus_{i=0}^{k^{(m)}} \AAA/\piL^{k_i^{(m)}} \AAA\]
	and therefore we see that there exists a continuous, set theoretical section of the
	canonical projection $\AAA \otimes_{\OL} V_{m+1} \to \AAA \otimes_{\OL} V_m$ as above
	using \hyperref[galc1a.4]{Corollary \ref*{galc1a.4}} instead of \hyperref[galc1a.3]{Lemma \ref*{galc1a.3}}.\\
	The statement on the short exact sequences then follows immediately.
	\end{proof}
	\end{lem}

	\begin{lem} \label{galc1.9l2}
	Let $E|L$ be a finite extension and $H_E=\Gal(\oQp|\Ein)$ as usual. Then the operation of $H_E$
	on $\EELsep$ is continuous with respect to the discrete topology on $\EELsep$.
	\begin{proof}
	Let $x \in \EELsep$. Then there exists a finite extension $\FF|\EE_E$ such that $x \in \FF$.
	Then $x$ is fixed by $U:=\Gal(\EELsep|\FF)$ which is an open subgroup of $H_E$. If
	then $\tau \in H$ and $y \in \EELsep$ are such that $\tau(y)=x$, then
	$U\tau \times \{y\}$ is an open neighbourhood of $\{\tau\}\times \{y\}$ in
	$H_E\times \EELsep$ with $\sigma(\tau(y))=x$ for all $\sigma \in U$.
	\end{proof}
	\end{lem}
	
	\begin{lem} \label{galc1.9l1}
	Let $V$ be a finite dimensional $\kL$-representation of $G_K$. Then there exists a
	finite Galois extension $E|K$ such that $H_E$ acts trivially on $V$.
	\begin{proof}
	Since the action of $G_K$ on $V$ is continuous, the homomorphism
	$G_K \to \Aut_{\kL} (V)$ is continuous and since
	$V$ is a finite dimensional $\kL$-vector space, it is finite and so $\Aut_{\kL}(V)$ carries
	the discrete topology, i.e. the kernel of the upper homomorphism is open, which means that there exists a
	finite Galois extension $E|K$ such that $G_E$ acts trivially on $V$. With $G_E$
	also $H_E$ acts trivially on $V$.
	\end{proof}
	\end{lem}
	
	\begin{lem} \label{galc1.9}
	Let $V$ be a finite dimensional $\kL$-representation of $G_K$ and
	$E|K$ a finite Galois extension, such that $H_E$ acts trivially on $V$ and set
	$\Delta:= \Gal(\Ein|\Kin)$. Then $\Delta$ acts on the short exact sequence
	\[ 0 \to \omphi^n \EEp_E \otimes_{\kL} V \to \EE_E \otimes_{\kL} V \to \EE_E/\omphi^n \EEp_E \otimes_{\kL} V \to 0\]
	and it holds
	\begin{enumerate}
	\item $H^j(\Delta,\EE_E\otimes_{\kL} V)=0$ for all $j > 0$.
	\item There exists $r \geq 0$ such that $\omphi^r H^j(\Delta,\omphi^n \EEp_E \otimes_{\kL} V)=0$ for
	all $j >0$ and $n \in \ZZ$.
	\end{enumerate}
	\begin{proof}
	The proof is literally the same as the one of \cite[Lemma 2.2.10, p.20]{scholl}
	\end{proof}
	\end{lem}
	
	\begin{lem} \label{galc1.12}
	Let $V$ be a finite dimensional $\kL$-representation of $G_K$ and
	$E|K$ a finite Galois extension, such that $H_E$ acts trivially on $V$ and set
	$\Delta:= \Gal(\Ein|\Kin)$. Then we have
	\begin{enumerate}
	\item $ (\EELsep \otimes_{\kL} V)^{H_K} \cong (\EE_E \otimes_{\kL} V)^{\Delta}$.
	\item $(\omphi^n\EELsepp\otimes_{\kL} V)^{H_K} \cong (\omphi^n\EEp_E \otimes_{\kL} V)^{\Delta}$
	for all $n \geq 0$.
	\end{enumerate}
	\begin{proof}
	In both cases the proof is the same. So let $X$ be $\EELsep$ or $\omphi^n\EELsepp$ for some $n \geq 0$.
	Note that $H_K/H_E \cong \Delta$. We then get
	\[ (X \otimes_{\kL} V)^{H_K} = ((X \otimes_{\kL} V)^{H_E})^{H_K/H_E}
					= (X^{H_E} \otimes_{\kL} V)^{\Delta},\]
	where the last equation is true, since $H_E$ acts trivial on $V$.
	\end{proof}
	\end{lem}
	
	Before stating a corollary, we should introduce some notation. Since all projective systems which
	appear here are indexed by the natural numbers, we will make the following definitions only
	for projective systems indexed by natural numbers.

	\begin{prop} \label{galc1.11}
	Let $V$ be a finite dimensional $\kL$-representation of $G_K$ and
	$E|K$ a finite Galois extension, such that $H_E$ acts trivially on $V$ and set
	$\Delta:= \Gal(\Ein|\Kin)$. Let in addition $M=\MOEK(V)$ and
	\[M_n:= M /\left(\omphi^n \EELsepp \otimes_{\kL} V \right)^{H_K}.\]
	Then we have
	\begin{enumerate}
	\item The inverse systems $(H^j(\Delta,\omphi^n \EEp_E \otimes_{\kL} V))_n $ and
	$(H^j(\Delta,\EE_E/\omphi^n\EEp_E \otimes_{\kL} V))_n$ are ML-zero
	for all $j >0$.
	\item The map of inverse systems $(M_n)_n \to (H^0(\Delta,\EE_E/\omphi^n\EEp_E \otimes_{\kL} V))_n$ is an
	ML-isomorphism.
	\end{enumerate}
	\begin{proof} $ $ \renewcommand{\theenumi}{{\arabic{enumi}.}}
	\begin{enumerate}
	\item Since $V$ is a finite dimensional $\kL$-vector space, it's flat and therefore
	the homomorphism
	$\omphi^{n+1} \EEp_E \otimes_{\kL} V \subseteq \omphi^n\EEp_E \otimes_{\kL} V$ is injective and
	induces a homomorphism
	\[H^j(\Delta,\omphi^{n+1}\EEp_E \otimes_{\kL} V) \to H^j(\Delta,\omphi^n\EEp_E  \otimes_{\kL} V).\]
	The image of this last homomorphism is a subset of $\omphi H^j(\Delta,\omphi^n\EEp_E \otimes_{\kL} V)$, i.e. the
	maps $H^j(\Delta,\omphi^{k}\EEp_E \otimes_{\kL} V) \to H^j(\Delta,\omphi^n\EEp_E \otimes_{\kL} V)$
	are zero for $k \geq n+r$ (cf.\ \hyperref[galc1.9]{Lemma \ref*{galc1.9}, 2.}), i.e. the inverse system
	$(H^j(\Delta,\omphi^n \EEp_E \otimes_{\kL} V))_n $ is ML-zero for $j >0$.\\
	Since every class in $\EE_E/\omphi^n\EEp_E$ has a unique representative of
	highest degree $\leq n-1$ in $\omphi$ the homomorphism $\EE_E \to \EE_E/\omphi^n \EEp_E$ has
	a set theoretical splitting (by sending a class to this representative). This map is
	continuous, since the preimage of a subset of $\EE_E$ in $\EE_E/\omphi^n \EEp_E$ is equal to the
	image under the canonical projection, which is open by definition. Since $V$ is flat, the
	sequence
	\[ 0 \to \omphi^n \EEp_E \otimes_{\kL} V \to \EE_E \otimes_{\kL} V \to \EE_E/\omphi^n\EEp_E \otimes_{\kL} V \to 0\]
	is exact and we can deduce a long
	exact cohomology sequence (cf.\ \cite[(2.3.2) Lemma, p.106]{nsw}) and since
	$H^j(\Delta, \EE_E\otimes_{\kL} V)=0$ for $j >0$ (cf.\ \hyperref[galc1.9]{Lemma \ref*{galc1.9}, 1.}), the
	homomorphism
	\[ H^j(\Delta,\EE_E/\omphi^n \EEp_E \otimes_{\kL} V) \to H^{j+1}(\Delta, \omphi^n \EEp_E \otimes_{\kL} V)\]
	is an isomorphism for all $j >0$ and the diagram
	\[ \begin{xy} \xymatrix{
		&H^j(\Delta,\EE_E/\omphi^n\EEp_E \otimes_{\kL} V) \ar[r] &H^{j+1}(\Delta, \omphi^n \EEp_E \otimes_{\kL} V) \\
		& H^j(\Delta,\EE_E/\omphi^{n+1}\EEp_E \otimes_{\kL} V) \ar[u] \ar[r]
		&H^{j+1}(\Delta, \omphi^{n+1} \EEp_E \otimes_{\kL} V) \ar[u]
	} \end{xy} \]
	commutes. This means that the transition map
	\[ H^j(\Delta,\EE_E/\omphi^{k}\EEp_E \otimes_{\kL} V) \to H^j(\Delta,\EE_E/\omphi^{n}\EEp_E \otimes_{\kL} V)\]
	is zero for $k \geq n+r$ and therefore the inverse system
	$(H^j(\Delta,\EE_E/\omphi^{n}\EEp_E \otimes_{\kL} V))_n$ is ML-zero.
	\item As seen before, for every $n \geq 0$ we have an exact sequence
	\[ 0 \to \omphi^n \EEp_E \otimes_{\kL} V \to \EE_E \otimes_{\kL} V \to \EE_E/\omphi^n\EEp_E \otimes_{\kL} V \to 0.\]
	Taking $\Delta$-invariants then gives an exact sequence
	\[ \begin{xy} \xymatrix@R-2pc{
	0 \ar[r] &(\omphi^n \EEp_E \otimes_{\kL} V)^{\Delta} \ar[r] & (\EE_E \otimes_{\kL} V)^{\Delta} \ar@{-}[r] &  \cdots \\
	\cdots \ar[r] & (\EE_E/\omphi^n\EEp_E \otimes_{\kL} V)^{\Delta} \ar[r]
	&H^{1}(\Delta, \omphi^n \EEp_E \otimes_{\kL} V) \ar[r] & 0,
	} \end{xy} \]
	where the last term is zero because $H^j(\Delta, \EE_E\otimes_{\kL} V)=0$ for $j >0$
	(cf.\ \hyperref[galc1.9]{Lemma \ref*{galc1.9}, 1.}). With
	\hyperref[galc1.12]{Lemma \ref*{galc1.12}} this sequences becomes
	\[ \begin{xy} \xymatrix@R-2pc{
	0 \ar[r] & (\omphi^n \EELsepp \otimes_{\kL} V)^{H_K} \ar[r] & (\EELsep \otimes_{\kL} V)^{H_K} \ar@{-}[r] & \cdots \\
	\cdots \ar[r] & (\EE_E/\omphi^n\EEp_E \otimes_{\kL} V)^{\Delta} \ar[r]
	& H^{1}(\Delta, \omphi^n \EEp_E \otimes_{\kL} V) \ar[r] & 0
	} \end{xy} \]
	and then gives the following short exact sequence
	\[ \begin{xy} \xymatrix@R-2pc@C-1pc{
	 0 \ar[r] & (\EELsep \otimes_{\kL} V)^{H_K}/(\omphi^n \EELsepp \otimes_{\kL} V)^{H_K}  \ar@{-}[r] &
	 \cdots & \\
	\cdots \ar[r] & (\EE_E/\omphi^n\EEp_E \otimes_{\kL} V)^{\Delta} \ar[r]
	& H^{1}(\Delta, \omphi^n \EEp_E \otimes_{\kL} V) \ar[r] & 0.
	} \end{xy} \]
	In particular, $H^{1}(\Delta, \omphi^n \EEp_E \otimes_{\kL} V)$ is the cokernel of the homomorphism \linebreak
	$M_n \to (\EE_E/\omphi^n\EEp_E \otimes_{\kL} V)^{\Delta}$.
	According to the first part of the proof the inverse system
	$(H^1(\Delta,\omphi^n \EEp_E \otimes_{\kL} V))_n $ is ML-zero, and since the kernel
	of $M_n \to (\EE_E/\omphi^n\EEp_E \otimes V)^{\Delta}$ is zero it is also ML-zero, which then ends the
	proof.	
	\end{enumerate}
	\end{proof}
	\end{prop}

		
	\begin{thm} \label{galc1.3}
	Let $V \in \RepGKOLfg$ and set $M = \MOEK(V)$. Then there are isomorphisms
	\[ \begin{xy} \xymatrix@R-1pc{
		&\Hctss(G_K,V) \ar[r]^{\cong} &\HfphKLs(\GaK,M),\\
		&\Hctss(H_K,V) \ar[r]^{\cong} &\HfphKLs(M).
	} \end{xy} \]
	These isomorphisms are functorial in $V$ and compatible with restriction and corestriction. On the level of complexes these isomorphisms are induced from quasi-isomorphisms
\[ \begin{xy} \xymatrix@R-1pc{
	& \Cctsb(G_K,V) \ar[r] &\varprojlim_{n,m}\CFrb(G_K, (\AAA/\omphi^n\AAAp) \otimes_{\OL} V/\piL^mV) & \CphKLb(\GaK,M),\ar[l] , \\
	& \Cctsb(H_K,V) \ar[r] &\varprojlim_{n,m}\CFrb(H_K, (\AAA/\omphi^n\AAAp) \otimes_{\OL} V/\piL^mV) & \CphKLb(M).\ar[l] .
	} \end{xy} \]

If $A$ denotes a cofinitely generated $\mathcal{O}_L$-module with continuous $G_K$-action, we obtain similar quasi-isomorphisms
\[ \begin{xy} \xymatrix@R-1pc{
	& \Cctsb(G_K,A) \ar[r] &\varinjlim_m \varprojlim_{n}\CFrb(G_K, (\AAA/\omphi^n\AAAp) \otimes_{\OL} A_m) & \CphKLb(\GaK,M)\ar[l] . \\
	} \end{xy} \]
where $A_m:=A[\pi_L^m]$ denotes the kernel of multiplication by $\pi_L^m$ and $M=\varinjlim_m \MOEK(A_m).$ An analogous statement for $H_K$ is stated in \eqref{f:quasiH}, but compare also with \emph{Proposition \ref{qind}} below.

	\begin{proof} In this proof, we follow closely the proof of \cite[Theorem 2.2.1, p.702-706]{scholl}.

	 {\bfseries \underline{Step 1:}} Explaining the strategy.\\
	First, for $m \in \NN$ set $V_m:=V/\piL^mV$ and $M_m:=M/\piL^mM$. Since $\MOEK$ is an equivalence
	of categories (cf.\ \hyperref[equivcat]{Theorem \ref*{equivcat}}) it is exact and therefore we have
	$M_m = \MOEK(V_m)$.
	The open subgroups
	\[ M_m \cap \left(\omphi^n \AAAp \otimes_{\OL} V_m\right) =
	\left(\omphi^n \AAAp \otimes_{\OL} V_m \right)^{H_K}\]
	form a basis of neighbourhoods of $0$ in $M_m$. 
	These subgroups are clearly stable under the operation of $\GaK$ and since $\phKL$ commutes
	with the operation of $G_K$ on $\left(\omphi^n \AAAp \otimes_{\OL} V_m \right)$ these subgroups are also
	stable under $\phKL$.
	We then set
	\[ \Mmn := M_m/\left(\omphi^n \AAAp \otimes_{\OL} V_m \right)^{H_K}.\]
	Since $\left(\omphi^n \AAAp \otimes_{\OL} V_m \right)^{H_K}$ is an open subgroup, this is a discrete
	$\GaK$-module and we have topological isomorphisms
	\begin{align*}
		M_m &\cong \varprojlim_n \Mmn \\
		M &\cong \varprojlim_m M_m.
	\end{align*}
	In \hyperref[galc1.2]{Corollary \ref*{galc1.2}} we proved that the sequence
	\[ \begin{xy} \xymatrix{
		0 \ar[r] & \OL \ar[r] & \AAA/\omphi^n\AAAp \ar[rr]^-{\Fr-\id} & & \AAA/\omphi^n \AAAp \ar[r] &0
	} \end{xy} \]
	is exact and since $\AAA/\omphi^n\AAAp$ is a free $\OL$-module, it is flat and therefore the sequence
	\[ \begin{xy} \xymatrix{
		0 \ar[r] & V_m \ar[r] & \AAA/\omphi^n\AAAp \otimes_{\OL} V_m\ar[rr]^-{\Fr-\id} & &
		\AAA/\omphi^n \AAAp \ar[r] \otimes_{\OL} V_m &0
	} \end{xy} \]
	is also exact.
	Then \hyperref[galc1.3l3]{Lemma \ref*{galc1.3l3}} says that for
	every $m,n\geq 1$ we have a quasi isomorphism
	\[ \begin{xy} \xymatrix@R-1pc{
	& \Cctsb(G_K,V_m) \ar[r] &\CFrb(G_K, (\AAA/\omphi^n\AAAp) \otimes_{\OL} V_m). \\
	} \end{xy} \]
	The inverse systems $(V_m)_m$ and
	$((\AAA/\omphi^n\AAAp) \otimes_{\OL} V_m)_{n,m}$ have surjective transition maps. From
	\hyperref[pre6.10]{Corollary \ref*{pre6.10}} we then can deduce that also the
	inverse systems of complexes $(\Cctsb(G_K,V_m))_m$ and
	$\Cctsb(G_K,((\AAA/\omphi^n\AAAp) \otimes_{\OL} V_m))_{n,m}$ have surjective transition maps and
	\hyperref[galc1.8l2]{Lemma \ref*{galc1.8l2}} then says that the system
	$\CFrb(G_K,((\AAA/\omphi^n\AAAp) \otimes_{\OL} V_m))_{n,m}$ has surjective transition maps
	as well.\\
	From the quasi isomorphism $\Cctsb(G_K,V_m) \to \CFrb(G_K,(\AAA/\omphi^n\AAAp \otimes V_m))$
	we then can deduce with \hyperref[galc1.16]{Proposition \ref*{galc1.16}} that the cohomologies of the
	complexes
	$\varprojlim_{n,m} \CFrb(G_K,(\AAA/\omphi^n\AAAp \otimes_{\OL} V_m))$ and
	$\varprojlim_m \Cctsb(G_K,V_m)$ coincide. Since
	$\varprojlim_m \Cctsb(G_K,V_m) \cong \Cctsb(G_K,V)$, the cohomology of $\varprojlim_m \Cctsb(G_K,V_m)$
	 is  $\Hctss(G_K,V)$, which then is also computed by\linebreak
	 $\varprojlim_{n,m} \CFrb(G_K,(\AAA/\omphi^n\AAAp \otimes_{\OL} V_m))$.\\
	On the other hand, since the canonical inclusion
	$\iota\colon\Mmn \hookrightarrow (\AAA/\omphi^n\AAAp) \otimes_{\OL} V_m$
	commutes with $\phKL$ and since together with
	the canonical projection $\pr\colon G_K \twoheadrightarrow \GaK$ it holds
	\[ \iota(\pr(\sigma) x) = \sigma \iota(x)\]
	for all $\sigma \in G_K$ and $x \in \Mmn$ and since the operations of $\phKL$ and $G_K$
	respectively $\GaK$ commute wo get an induced morphism of complexes
	\[ \alphmn\colon \CphKLb(\GaK,\Mmn) \to \CFrb(G_K,(\AAA/\omphi^n\AAAp) \otimes_{\OL} V_m)\]
	(cf.\ \cite[I \S 5, p45]{nswo}, the additional properties concerning $\phKL$ we noted above, ensure
	that we get the morphism
	of the above total complex with respect to $\phKL$ on the left hand side and $\Fr$ on the right hand side).\\
	We now want to see that $\varprojlim_{n,m} \alphmn$ is a quasi isomorphism. Because of\linebreak
	$\varprojlim_{n,m} \CphKLb(\GaK,\Mmn)= \CphKLb(\GaK,M)$ (cf.\ \hyperref[galc1.8]{Lemma \ref*{galc1.8}}),
	this then says that the cohomology of $\CphKLb(\GaK,M)$ and
	$\varprojlim_{n,m} \CFrb(G_K,(\AAA/\omphi^n\AAAp \otimes_{\OL} V_m))$ coincide. But then
	the cohomologies of $\CphKLb(\GaK,M)$ and $\Cctsb(G_K,V)$ coincide, what is exactly what we want to prove.\\
	To see that $\varprojlim_{n,m} \alphmn$ is a quasi isomorphism, it is enough to
	see, that $\varprojlim_n \alphmn$ is a quasi isomorphism for every $m \geq 1$. Because if this is shown,
	one knows that the inverse systems of complexes $(\CphKLb(\GaK,M_m))_m$ and
	$(\CphKLb(G_K,\AAA\otimes_{\OL}V_m))_m$  are quasi isomorphic. Since
	the transition maps $M_{m+1}\to M_m$ as well as
	$\AAA \otimes_{\OL} V_{m+1} \to \AAA \otimes_{\OL} V_m$
	are surjective and have a continuous section (cf.\ \hyperref[galc1a.5]{Lemma \ref*{galc1a.5}}),
	one can see as before,
	using \hyperref[pre6.10]{Corollary \ref*{pre6.10}} and \hyperref[galc1.8l2]{Lemma \ref*{galc1.8l2}},
	that the inverse systems of complexes $(\CphKLb(\GaK,M_m))_m$ and
	$(\CphKLb(G_K,\AAA\otimes_{\OL}V_m))_m$ have surjective transition maps. As before with
	\hyperref[galc1.16]{Proposition \ref*{galc1.16}} respectively \hyperref[galc1.16r]{Remark \ref*{galc1.16r}}
	one then sees that $\varprojlim_m\CphKLb(\GaK,M_m)$ and
	$\varprojlim_m \CphKLb(G_K,\AAA\otimes_{\OL}V_m)$ are quasi isomorphic.\\
	So, what is still to show, is that $\varprojlim_{n} \alphmn$ is a quasi isomorphism for every
	$m \geq 1$. This will be the rest of the proof.


	 {\bfseries \underline{Step 2:}} Reduction to the case $m=1$. \\
	Since for every $m \geq 1$ the sequence
	\[ \begin{xy} \xymatrix{
	 &0 \ar[r] & V_m \ar[r] & V_{m+1} \ar[r] & V_1 \ar[r] &0.
	} \end{xy}\]
	is exact and $\MOEK$ is an exact functor (since it is an equivalence), this
	implies that for every $m \geq 1$ there is a short exact sequence
	\[ \begin{xy} \xymatrix{
	 &0 \ar[r] & M_m \ar[r] & M_{m+1} \ar[r] &M_1 \ar[r] &0.
	} \end{xy}\]
	By the definition of the topology on the $M_m$ it is clear, that the topology of
	$M_m$ is induced from that of $M_{m+1}$ and from \hyperref[galc1a.5]{Lemma \ref*{galc1a.5}}
	we deduce that it has a continuous set theoretical section.
	Therefore \hyperref[pre7.35]{Proposition \ref*{pre7.35}}
	says that we get a long exact sequence of cohomology.\\
	Now assume the result is shown for $m=1$. Then $\HfphKLs(\GaK,M) \to \Hctss(G_K,V)$ is
	an isomorphism for
	every $V$ with $\piL V=0$. Induction on $m$ and the $5$-lemma applied to the following
	diagram which arises from the long exact cohomology sequences
	(where we write $\Gamma = \GaK$ and $G =G_K$ and $\varphi=\phKL$)
	\[ \begin{xy} \xymatrix@C-0.77pc{
	\Hfph^{l-1}(\Gamma,M_1) \ar[r]^{\delta} \ar[d]_{\cong}& \Hfph^l(\Gamma,M_m) \ar[r] \ar[d]_{\cong}
	& \Hfph^l(\Gamma,M_{m+1}) \ar[r] \ar[d]
	 & \Hfph^l(\Gamma,M_1) \ar[r]^{\delta}\ar[d]_{\cong} &\Hfph^{l+1}(\Gamma,M_m) \ar[d]_{\cong}\\
	\Hcts^{l-1}(G,V_1) \ar[r]^{\delta} &\Hcts^l(G,V_m) \ar[r] & \Hcts^l(G,V_{m+1}) \ar[r]
	&  \Hcts^l(G,V_1) \ar[r]^{\delta}  &\Hcts^{l+1}(G,V_m)
	} \end{xy} \]
	then implies the result for all $m \geq 1$.

	 {\bfseries \underline{Step 3:}} Splitting $\alpha_{1,n}$ up.\\
	For the rest of the proof we may assume
	$\piL V=0$ and therefore also $\piL M = 0$, but we will still write
	$M_{1,n}$ to avoid confusion. Note that this implies
	\[ \AAA \otimes_{\OL} V \cong \EELsep \otimes_{\kL} V,\]
	\[ \omphi^n \AAAp \otimes_{\OL} V \cong \omphi^n \EELsepp \otimes_{\kL} V \]
	as well as the correspondingly isomorphism with respect to the fixed modules of $H_K$.\\
	Now fix a finite Galois extension $E|K$ such that
	$H_E$ acts trivially on $V$ (cf.\ \hyperref[galc1.9l1]{Lemma \ref*{galc1.9l1}}). Then, the
	canonical inclusion
	\[ \begin{xy} \xymatrix{
		&\Men = \frac{(\EELsep \otimes_{\kL} V)^{H_K}}{(\omphi^n\EELsepp \otimes_{\kL} V)^{H_K}} \ar@{^{(}->}[r] 
		&\frac{(\EELsep \otimes_{\kL} V)^{H_E}}{(\omphi^n\EELsepp \otimes_{\kL} V)^{H_E}} =
		\EE_E/\omphi^n \EEp_E \otimes_{\kL} V
	} \end{xy} \]
	induces together with the canonical projection $\Gal(\Ein|K) \twoheadrightarrow \GaK$,
	as in step 1 for $\alphmn$, for all $n \in \NN$ a morphism of complexes
	\[ \begin{xy} \xymatrix{
		& \beta_n \colon \CphKLb(\GaK,\Men) \ar[r]
		& \CFrb(\Gal(\Ein| K),\EE_E/\omphi^n \EEp_E \otimes_{\kL} V).
	} \end{xy} \]
	Simultaneously, the canonical inclusion $\EE_E/\omphi^n\EEp_E \otimes_{\kL} V \hookrightarrow
	\EELsep/\omphi^n\EELsepp \otimes_{\kL} V$ together with the canonical projection
	$G_K \twoheadrightarrow \Gal(\Ein|K)$ induces for all $n \in \NN$ a morphism of complexes
	\[ \begin{xy} \xymatrix{
		 \gamma_n \colon \CFrb(\Gal(\Ein| K),\EE_E/\omphi^n \EEp_E \otimes_{\kL} V)
		\ar[r] & \CFrb(G_K,\EELsep/\omphi^n\EELsepp \otimes_{\kL} V).
	} \end{xy} \]
	Since both diagrams
	\[ \begin{xy} \xymatrix{
		& \Men \ar@{^{(}->}[r] \ar@{^{(}->}[dr] & \EE_E/\omphi^n\EEp_E \otimes_{kL} V \ar@{^{(}->}[d] 
		& & \GaK & \Gal(\Ein|K) \ar@{->>}[l] \\
		& & \EELsep/\omphi^n\EELsepp \otimes_{\kL} V,  & & & G_K \ar@{->>}[ul] \ar@{->>}[u]
	} \end{xy} \]
	are commutative, where all the arrows in the left diagram are canonical inclusions and the ones in
	the right diagram are canonical projections, it is immediately clear that also the diagram
	\[ \begin{xy} \xymatrix{
	 & \CphKLb(\GaK,\Men) \ar[rr]^-{\beta_n}  \ar[rrd]_-{\alpha_{1,n}}
	&  & \CFrb(\Gal(\Ein| K),\EE_E/\omphi^n \EEp_E \otimes_{\kL} V) \ar[d]^-{\gamma_n} \\
		&  &   & \CFrb(G_K, \EELsep/\omphi^n\EELsepp \otimes_{\kL} V)
	} \end{xy} \]
	commutes. So, to prove that $\varprojlim_n \alpha_{1,n}$ is a quasi-isomorphism it is enough to prove that
	$\varprojlim_n \beta_n$ and $\varprojlim_n \gamma_n$ are quasi-isomorphisms. In addition,
	we will also show that $\gamma_n$ is a quasi-isomorphism for every $n \geq 1$.

	 {\bfseries \underline{Step 4:}} $\varprojlim_n \gamma_n$ is a quasi-isomorphism.\\
	Due to \hyperref[pre7.34]{Lemma \ref*{pre7.34}} there is an $E_2$-spectral sequence converging
	to the cohomology of the source of $\gamma_n$
	\[ \begin{xy} \xymatrix@R-2pc{
		 H^{a}(\Gal(\Ein| K),\HFr^b(\EE_E/\omphi^n \EEp_E \otimes_{\kL} V)) \ar@{=>}[r]
		 	&  \\  \HFr^{a+b}(\Gal(\Ein| K),\EE_E/\omphi^n \EEp_E \otimes_{\kL} V) &
	} \end{xy} \]		
	as well as en $E_2$-spectral sequence converging to the target of $\gamma_n$
	\[ \begin{xy} \xymatrix@R-2pc{
		H^{a}(\Gal(\Ein| K),\HFr^b(H_E,\EELsep/\omphi^n \EELsepp \otimes_{\kL} V)) \ar@{=>}[r]
			& \\   \HFr^{a+b}(G_K, \EELsep/\omphi^n\EELsepp \otimes_{\kL} V). &
	} \end{xy} \]
	The canonical inclusion $\EE_E/\omphi^n\EEp_E \otimes_{\kL} V
	\hookrightarrow \EELsep/\omphi^n \EELsepp \otimes_{\kL} V$ together with the trivial map
	$H_E \to 1$ then induces a homomorphism on the above $E_2$-pages. Together with
	the from $\gamma_n$ induced map on cohomology this then gives a morphism
	of spectral sequences.
	So, to show that $\gamma_n$ induces an isomorphism on cohomology it is enough to show that the
	induced homomorphism on the above $E_2$ pages is an isomorphism. And for this it is enough, that
	the homomorphism between the coefficients
	$\HFr^b(\EE_E/\omphi^n \EEp_E \otimes_{\kL} V)$ and
	$\HFr^b(H_E,\EELsep/\omphi^n \EELsepp \otimes_{\kL} V)$
	is an isomorphism. Since $H_E$ acts trivially on $V$ it is
	\begin{align*}
		\HFr^b(\EE_E/\omphi^n \EEp_E \otimes_{\kL} V)
		&= \HFr^b(\EE_E/\omphi^n \EEp_E) \otimes_{\kL} V\\
		\HFr^b(H_E,\EELsep/\omphi^n \EELsepp \otimes_{\kL} V)
		&= \HFr^b(H_E,\EELsep/\omphi^n \EELsepp) \otimes_{\kL} V
	\end{align*}
by \cite[(3.4.4) Proposition, p.\,66--67]{neksc}.
	Therefore it is enough to show that there is an isomorphism between
	$\HFr^b(\EE_E/\omphi^n \EEp_E)$ and  $\HFr^b(H_E,\EELsep/\omphi^n \EELsepp)$.
	To see this, consider the commutative square
	\[ \begin{xy} \xymatrix{
	 &\HFr^b(\EE_E) \ar[r] \ar[d]&\HFr^b(\EE_E/\omphi^n \EEp_E) \ar[d]\\
	  &\HFr^b(H_E,\EELsep) \ar[r] &\HFr^b(H_E,\EELsep/ \omphi^n \EELsepp),
	} \end{xy} \]
	where $\EELsep$ is regarded as discrete $H_E$-module (cf.\ \hyperref[galc1.9l2]{Lemma \ref*{galc1.9l2}}) and
	where the horizontal maps are induced from the respective canonical projections and the vertical
	maps from the respective canonical inclusions.\\
	First we want to see, that the upper horizontal map is an isomorphism.
	$\HFr^b(\EE_E)$ is computed by $\EE_E \overset{\phL-\id}{\longrightarrow} \EE_E$ and
	$\HFr^b(\EE_E/\omphi^n \EEp_E)$ by the corresponding complex and the square
	\[ \begin{xy} \xymatrix{
		&\EE_E \ar[rr]^-{\Fr-\id} \ar@{->>}[d] &  &\EE_E \ar@{->>}[d]\\
		& \EE_E/\omphi^n \EEp_E  \ar[rr]^-{\Fr-\id} & &\EE_E/\omphi^n \EEp_E
	}\end{xy} \]
	is commutative. Denote the kernel and image of the upper horizontal map by $\kappa_1$ and $\im_1$
	and the ones of the lower
	vertical map by $\kappa_2$ and $\im_2$ respectively. By \hyperref[galc1.1]{Lemma \ref*{galc1.1}} the map
	$\omphi^n\EEp_E \overset{\Fr-\id}{\longrightarrow} \omphi^n\EEp_E$
	is an isomorphism, especially is $\omphi^n \EEp_E \subseteq \im_1$ and so we see immediately
	$\im_2 \subseteq \im_1/\omphi^n \EEp_E$. For the other inclusion let $\overline{x} \in \im_1/\omphi^n
	\EEp_E$ and $x \in \EE_E$ a preimage under the canonical projection. Because of $\omphi^n\EEp_E
	\subseteq \im_1$ we deduce $x \in \im_1$. If $y \in \EE_E$ is a preimage of $x$ under $\Fr-\id$, then
	because of the commutativity of the latter diagram we get
	$(\Fr-\id)(\overline{y}) = \overline{x}$, i.e. $\overline{x} \in \im_2$.
	Therefore $\HFr^1(\EE_E)$ and
	$\HFr^1(\EE_E/\omphi^n\EEp_E)$ coincide.\\
	For the term in degree zero let $x \in \kappa_1$ such that $x \in \omphi^n \EEp_E$. Since
	$\Fr-\id$ is an isomorphism on $\omphi^n \EEp_E $ and $(\Fr-\id)(x)=0$, $x$ itself
	is zero, i.e. the canonical homomorphism $\kappa_1 \to \kappa_2$ is injective. Let now
	$\eta \in \kappa_2$ and $y' \in \EE_E$ be a preimage under the canonical projection. By
	commutativity it is $\overline{ (\Fr-\id)(y')}=0$ and therefore $(\Fr-\id)(y') \in \omphi^n \EEp_E$.
	Again since $\Fr-\id$ is an isomorphism on $\omphi^n \EEp_E$
	we find an element $y'' \in \omphi^n \EEp_E$ with $(\Fr-\id)(y')=(\Fr-\id)(y'')$.
	Set $y := y' - y''$. Then $\overline{y}=\overline{y'}-\overline{y''}=\overline{y'}=\eta$
	and $(\Fr-\id)(y)=0$, i.e. $\kappa_1 \to \kappa_2$ is also surjective and therefore an
	isomorphism. Since every other cohomology group is zero, we conclude that
	\[ \HFr^b(\EE_E) \cong \HFr^b(\EE_E/\omphi^n\EEp_E)\]
	for all $b \geq 0$.\\
	For the lower horizontal map in the upper square, recall that \hyperref[galc1.1]{Lemma \ref*{galc1.1}}
	also says that $\Fr-\id$ is on $\omphi^n\EELsepp$ an isomorphism.
	Therefore one sees with a similar argument as above that the canonical projection
	$\EELsep \twoheadrightarrow \EELsep/\omphi^n\EELsepp$ induces an isomorphism between
	the cohomology groups
	$\HFr^{b'}(\EELsep)$ and $\HFr^{b'}(\EELsep/\omphi^n\EELsepp)$ for all $b' \geq 0$.
	\hyperref[pre7.21]{Lemma \ref*{pre7.21}} states that there are two $E_2$-spectral sequences
	converging to $\HFrs(H_E,\EELsep)$ respectively $\HFrs(H_E,\EELsep/\omphi^n\EELsepp)$ (recall
	from the beginning of Step 4 that $\EELsep$ is considered as discrete $H_E$-module):
	\begin{align*}
		H^{a'}(H_E,\HFr^{b'}(\EELsep)) & \Rightarrow \HFr^{a'+b'}(H_E,\EELsep) \\
		H^{a'}(H_E,\HFr^{b'}(\EELsep/\omphi^n\EELsepp))
		& \Rightarrow \HFr^{a'+b'}(H_E,\EELsep/\omphi^n\EELsepp).
	\end{align*}
	We conclude as before: The canonical projection
	$\EELsep \twoheadrightarrow \EELsep/\omphi^n\EELsepp$ induces
	a morphism of spectral sequences and since the induced homomorphism is an
	isomorphism
	on the $E_2$-pages, we obtain an
	isomorphism between the limit terms $\HFr^b(H_E,\EELsep)$ and $\HFr^b(H_E,\EELsep/\omphi^n\EELsepp)$ for
	all $b \geq 0$.\\
	To see that the left vertical arrow in the first square is an isomorphism we consider
	the $E_2$-spectral sequence (cf.\ \hyperref[pre7.21]{Lemma \ref*{pre7.21}})
	\[\HFr^{a'}(H^{b'}(H_E,\EELsep)) \Rightarrow \HFr^{a'+b'}(H_E,\EELsep).\]
	Since $\EELsep$ is a separabel closure of $\EE_E$ with Galois group isomorphic to
	$H_E$ it is $H^{b'}(H_E,\EELsep)=0$ for all $b'>0$. Then \cite[Chapter II \S 1, (2.1.4) Proposition, p.100]{nswo}
	says that we have an isomorphism $\HFr^{b}(\EE_E) \cong \HFr^{b}(H_E,\EELsep)$ for
	all $b \geq 0$ (here we identified $H^0(H_E,\EELsep)=(\EELsep)^{H_E}=\EE_E$), which is induced from
	the canonical inclusion, i.e.
	the left vertical arrow in the first square also is an isomorphism.
	Then also the right vertical arrow is an isomorphism (since all other arrows are isomorphisms)
	and so is the map on $E_2$-terms from which we started.
	Hence $\gamma_n$ is a quasi-isomorphism for all $n$.\\
	To see that $\varprojlim_n \gamma_n$ is an isomorphism, it remains to check that the
	transition maps are surjective (cf.\ \hyperref[galc1.16]{Proposition \ref*{galc1.16}} respectively
	\hyperref[galc1.16r]{Remark \ref*{galc1.16r}}).
	Since the transition maps
	\[ \begin{xy} \xymatrix@R-2pc{
	&\EE_E/\omphi^{n+1}\EEp_E \otimes_{\kL} V \ar@{->>}[r] &\EE_E/\omphi^n\EEp_E \otimes_{\kL} V, \\
	&\EELsep/\omphi^{n+1}\EELsepp \otimes_{\kL} V \ar@{->>}[r] &\EELsep/\omphi^n\EELsepp \otimes_{\kL} V
	} \end{xy} \]
	are surjective and the groups carry the discrete topology, \hyperref[pre6.10]{Corollary \ref*{pre6.10}}
	says that also the transition maps
	\[ \begin{xy} \xymatrix@R-2pc{
	\Cctsb(\Gal(\Ein|K),\EE_E/\omphi^{n+1}\EEp_E \otimes_{\kL} V) \ar@{->>}[r] &
	\Cctsb(\Gal(\Ein|K),\EE_E/\omphi^n\EEp_E \otimes_{\kL} V), \\
	\Cctsb(G_K,\EELsep/\omphi^{n+1}\EELsepp \otimes_{\kL} V) \ar@{->>}[r]
	&\Cctsb(G_K,\EELsep/\omphi^n\EELsepp \otimes_{\kL} V)
	} \end{xy} \]
	are surjective. But then \hyperref[galc1.8l2]{Lemma \ref*{galc1.8l2}} says that the transition maps
	\[ \begin{xy} \xymatrix@R-2pc{
	\CFrb(\Gal(\Ein|K),\EE_E/\omphi^{n+1}\EEp_E \otimes_{\kL} V) \ar@{->>}[r] &
	\CFrb(\Gal(\Ein|K),\EE_E/\omphi^n\EEp_E \otimes_{\kL} V), \\
	\CFrb(G_K,\EELsep/\omphi^{n+1}\EELsepp \otimes_{\kL} V) \ar@{->>}[r] &\CFrb(G_K,\EELsep/\omphi^n\EELsepp
	\otimes_{\kL} V)
	} \end{xy} \]
	are surjective, too. Then \hyperref[galc1.16]{Proposition \ref*{galc1.16}} respectively
	\hyperref[galc1.16r]{Remark \ref*{galc1.16r}} say that $\varprojlim_n \gamma_n$ is a quasi
	isomorphism.
	
	 {\bfseries \underline{Step 5:}} $\varprojlim_n \beta_n$ is a quasi-isomorphism.\\
	Now let $\Delta:= \Gal(\Ein \mid \Kin)$. \hyperref[pre7.34]{Lemma \ref*{pre7.34}}
	then says that there is an $E_2$-spectral sequence
	of inverse systems of abelian groups given by
	\[ \begin{xy} \xymatrix{
		  \HFr^{i}(\GaK,H^j(\Delta,\EE_E/\omphi^n \EEp_E \otimes_{\kL} V))
		\ar@{=>}[r] & \HFr^{i+j}(\Gal(\Ein| K), \EE_E/\omphi^n \EEp_E \otimes_{\kL} V).
	} \end{xy} \]
	We will write $\leftidx{_n}\EEE^{ij}_2$ for second page of this $E_2$-spectral sequence,
	$\leftidx{_n}\EEE^{k}$ for its limit term and
	$\EEE^{ij}_2 = \varprojlim_ n \leftidx{_n}\EEE^{ij}_2$ as wells as $\EEE^{k} = \varprojlim_n \leftidx{_n}\EEE^{k}$.
	\hyperref[galc1.11]{Proposition \ref*{galc1.11}} says that the system
	$ (H^j(\Delta,\EE_E/\omphi^n \EEp_E \otimes_{\kL} V))_n$ is ML-zero for $j >0$, i.e. for every $n \in \NN$ there
	is an $m(n) \in \NN$ such that the transition map
	\[ \begin{xy} \xymatrix{
		& H^j(\Delta,\EE_E/\omphi^{m(n)} \EEp_E \otimes_{\kL} V) \ar[r]
		& H^j(\Delta,\EE_E/\omphi^n \EEp_E \otimes_{\kL} V)
	} \end{xy} \]
	is the zero map. For fixed $n \in \NN$ and $m(n) \in \NN$ as above, we then obtain that the transition map
	\[ \begin{xy} \xymatrix{
		 \Ccts^{i}(\GaK,H^j(\Delta,\EE_E/\omphi^{m(n)} \EEp_E \otimes_{\kL} V)) \ar[r]
		& \Ccts^{i}(\GaK,H^j(\Delta,\EE_E/\omphi^n \EEp_E \otimes_{\kL} V))
	} \end{xy} \]
	is also zero for all $i \geq 0$ and $j > 0$. Then clearly the transition map
	\[ \begin{xy} \xymatrix{
		\CFr^{i}(\GaK,H^j(\Delta,\EE_E/\omphi^{m(n)} \EEp_E \otimes_{\kL} V)) \ar[r]
		& \CFr^{i}(\GaK,H^j(\Delta,\EE_E/\omphi^n \EEp_E \otimes_{\kL} V))
	} \end{xy} \]
	is zero for all $i \geq 0$ and $j > 0$, too. And so is the induced map on cohomology, i.e the
	inverse systems
	$(\leftidx{_n}\EEE^{ij}_2)_n$ are ML-zero for all $i \geq 0$ and $j >0$.
	But then the edge homomorphism $\EEE_2^{i0} \to
	\EEE^{i}$ is an isomorphism, since $\EEE_2^{ij}=0$ for all $i \geq 0$ and $j >0$
	(cf.\ \cite[Chapter II, \S 1, (2.1.4) Corollary, p.100]{nswo}). Recall that this edge homomorphism
	is induced from both, the canonical
	projection $\Gal(\Ein|K) \twoheadrightarrow \GaK$ and the canonical inclusion
	$(\EE_E/\omphi^n\EEp_E \otimes_{\kL} V)^\Delta
	\hookrightarrow \EE_E/\omphi^n \EEp_E \otimes_{\kL} V$. \\
	\hyperref[galc1.11]{Proposition \ref*{galc1.11}} says that
	$(\eta_n)_n\colon (M_{1,n})_n \to (H^0(\Delta,\EE_E/\omphi^n \EEp_E \otimes_{\kL} V))_n$ is an
	ML-iso\-mor\-phism.
	Therefore the inverse systems $(\ker (\eta_n))_n$ and $(\coker (\eta_n))_n$ are ML-zero. As above,
	we then deduce that also the systems $(\CfphKL^{i}(\GaK,\ker(\eta_n)))_n$ and
	$(\CFr^{i}(\GaK,\coker(\eta_n)))_n$ are ML-zero for all $i \in \NN_0$. Since
	\mbox{$H^0(\Delta,\EE_E/\omphi^n\EEp_E \otimes_{\kL} V)$} and $\Men$ carry the discrete topology for all
	$n \in \NN$, we deduce from \hyperref[pre7.23]{Lemma \ref*{pre7.23}}, which says that
	$\CFr^{i}(\GaK,-)$ is for discrete modules an exact functor, the exact sequence
	\[ \begin{xy} \xymatrix@R-2pc@C-0.9pc{
		0 \ar[r] & \CfphKL^{i}(\GaK,\ker(\eta_n)) \ar[r] & \CfphKL^{i}(\GaK,\Men) \ar@{-}[rr]^-{\CFr^{i}(\GaK,\eta_n)}
		& &  \cdots \\
		\cdots \ar[r] & \CFr^{i}(\GaK,H^0(\Delta,\EE_E/\omphi^n \EEp_E \otimes_{\kL} V)) \ar[r]
		& \CFr^{i}(\GaK,\coker(\eta_n)) \ar[rr] & &  0.
	} \end{xy} \]
	Taking inverse limits then gives us an isomorphism of complexes
	\[ \CfphKLb(\GaK,M_1) \cong \CFrb(\GaK, H^0(\Delta,\EE_E\otimes_{\kL} V)),\]
	which, by construction, is induced from the canonical inclusion $M_1 \hookrightarrow
	\EE_E \otimes_{\kL} V$ and
	which then prolongs to an isomorphism of its respective cohomology groups, i.e. for
	all $i \in \NN_0$ we get
	\[ \HfphKL^{i}(\GaK,M_1) \cong \HFr^{i}(\GaK,H^0(\Delta,\EE_E\otimes_{\kL} V)).\]
	Together with the observation from above, that the edge homomorphism
	$\EEE_2^{i0} \to \EEE^{i}$ is an isomorphism for all $i \in \NN_0$ we deduce for all
	$i \in \NN_0$ the isomorphism
	\[ \HfphKL^{i}(\GaK,M_1) \cong \HFr^{i}(\Gal(\Ein|K),\EE_E\otimes_{\kL} V),\]
	which by construction is $\varprojlim_n \beta_n$.

%
\end{proof}
	\end{thm}
	
	\subsection{Description with $\psi$} 
	
	In this subsection we want to give a description of the Galois cohomology groups of a representation
	using a $\psi$-operator.
	
	\begin{mydef}
	Let $A$ be an $\OL$-module. We say that $A$ is {\bfseries cofinitely generated}
	\index{cofinitely generated} if its Pontrjagin dual $A^\vee=\HomctsOL(A,L/\OL)$ is
	finitely generated.
	\end{mydef}
	
	\begin{rem} \label{remcoftop} $ $
	\begin{enumerate}
	\item Since finitely generated $\OL$-modules together with their natural topology are compact, cofinitely
	generated $\OL$-modules are discrete, which means that \linebreak
	$\HomctsOL(-,L/\OL)=\Hom_{\OL}(-,L/\OL)$ for
	both, finitely and cofinitely generated $\OL$-modules.
	\item For $n \in \NN$ we have an isomorphism
	\[ \begin{xy} \xymatrix@R-2pc{
			\OL/\piL^n\OL \ar[r] & (\OL/\piL^n\OL)^\vee \\
			x \bmod \piL^n\OL \ar@{|->}[r] & [1 \bmod \piL^n\OL \mapsto \piL^{-n} x \bmod \OL]
	} \end{xy} \]
	which then also implies a non-canonical isomorphism $T \cong T^\vee$ for a finitely generated
	torsion $\OL$-module, since
	$(-)^\vee$ is compatible with finite direct sums. These isomorphisms are clearly topological, since all these
	objects carry the discrete topology.
	\item Due to Pontrjagin duality (cf.\ \emph{\hyperref[pontrjdual]{Proposition \ref*{pontrjdual}}}) a cofinitely
	generated $\OL$-module is always the Pontrjagin dual of a finitely generated $\OL$-module.
	\item If $T \in \RepGKOLfg$ is torsion, then $T^\vee$ also is a finitely generated torsion $\OL$-module with
	a continuous action from $G_K$.
	\end{enumerate}
	\end{rem}
	
	\begin{mydef} \label{deftorcof}
	Let $A$ be a cofinitely generated $\OL$-module and $n \in \NN$. We denote by
	\gls{An} the kernel of the multiplication $m_{\piL^n}$ with $\piL^n$ on $A$, i.e.
	\[ A_n= \ker(m_{\piL^n}\colon A \to A).\]
	\end{mydef}
	
	\begin{prop} \label{structcof}
	Let $A$ be a cofinitely generated $\OL$-module. Then we have $A = \varinjlim_n A_n$.\\
	In particular, if $A$ is torsion, say with $\piL^mA = 0 $ for some $m \in \NN$, then we have $A = A_m$.
	\begin{proof}
	Let $T$ be a finitely generated $\OL$-module such that $A=\HomctsOL(T,L/\OL)$, let
	$e_1,\dots,e_m$ be a set of generators of $T$ and
	let $f \in A$.
	Then for every $i \in \{1,\dots,m\}$ there exists an $n_i \in \NN$ such that $\piL^{n_i} f(e_i)=0$.
	Set $n:=\mathrm{max}_i n_i$. Then it is $\piL^n f(\alpha)=0$ for every $\alpha \in A$, i.e.
	$f \in A_n$.\\
	In particular, if there exists $m \in \NN$ such that $\piL^m g = 0 $ for every $g \in A$, then
	the above shows $A=A_m$.
	\end{proof}
	\end{prop}
		
	\begin{lem} \label{lemctsdis}
	Let $T \in \RepGKOLfg$ such that $\piL^m T=0$. Then $H_K$ acts continuously on
	$\AAA \otimes_{\OL} T$ equipped with the discrete topology.
	\begin{proof}
	Recall from page \pageref{aaacompl} that
	\[ \AAA \cong \varprojlim_n \AALnr/\piL^n\AALnr\]
	and that $H_L$ is the Galois group of $\AALnr|\AAL$. The latter means, that $H_L$ acts continuously on
	$\AALnr$ with respect to the discrete topology because if $x \in \AALnr$, then $\BBL(x)|\BBL$ is a finite
	extension and therefore it exists an open subgroup $U \leq H_L$ which fixes $x$. But then
	$(U,x)$ is an open subset of the preimage of $x$ under the operation
	\[ H_L \times \BBL \to \BBL.\]
	Then $H_L$ also clearly acts continuously on $\AALnr/\piL^n\AALnr$ for all $n \in \NN$ equipped
	with the discrete topology. Since $H_K$ is an open subgroup of $H_L$ it then also acts continuously
	on $\AALnr/\piL^n\AALnr$ for all $n \in \NN$ equipped with the discrete topology.
	Because of $\piL^nT=0$ we have $T=T \otimes_{\OL} \OL/\piL^n\OL$ and therefore
	\[ \AAA \otimes_{\OL} T = \AAA \otimes_{\OL} \OL/\piL^n\OL \otimes_{\OL} T = \AAA/\piL^n\AAA \otimes_{\OL} T
			= \AALnr/\piL^n\AALnr \otimes_{\OL} T.\]
	Since $H_K$ acts continuously on both $T$ and $\AALnr/\piL^n\AALnr$ with respect to the discrete topology
	it does so on the tensor product equipped with the linear topological structure, which then again is
	discrete.
	\end{proof}
	\end{lem}
	
	\begin{lem} \label{lemcomtriv}
	Let $T \in \RepGKOLfg$ such that $\piL^m T=0$. Then we have\linebreak
	$\Hcts^{i}(H_K,\AAA\otimes_{\OL} T) = 0$ for all $i >0$.
	\begin{proof}
	This is \cite[Lemma 5.2, p.\,23--24]{SV15}, since it is even $\Hcts^{i}(U,\EELsep)=0$ for all $i >0$ and
	open subgroups $U \leq H_L$.
	\end{proof}
	\end{lem}
	
	\begin{cor} \label{cofcomtriv}
	Let $A$ be a cofinitely generated $\OL$-module with a continuous action from $G_K$.
	Then $H_K$ acts continuously
	on $\AAA \otimes_{\OL} A$ equipped with the discrete topology and we have
	$\Hcts^{i}(H_K,\AAA \otimes_{\OL} A)=0$ for all $i>0$.
	\begin{proof}
	If $A$ is torsion, then \hyperref[remcoftop]{Remark \ref*{remcoftop}} says
	that this is just \hyperref[lemctsdis]{Lemma \ref*{lemctsdis}}
	and \hyperref[lemcomtriv]{Lemma \ref*{lemcomtriv}}.\\
	If $A$ is general, then with \hyperref[structcof]{Proposition \ref*{structcof}} we can write
	$A = \varinjlim_n A_n$, where the $A_n$ are torsion $\OL$-modules. Since tensor products
	commute with colimits we have
	\[ \varinjlim_n \AAA \otimes_{\OL} A_n \cong \AAA \otimes_{\OL} A\]
	algebraically. But the direct limit topology of $ \varinjlim_n \AAA \otimes_{\OL} A_n$ again is discrete
	and so the above isomorphism is also topological. Then, $\AAA \otimes_{\OL} A$ is a discrete $H_L$-module
	and therefore we deduce from \cite[(1.5.1) Proposition, p.\,45--46]{nswo}
	\[ H^{i}(H_K,\AAA \otimes_{\OL} A )= \varinjlim_n H^{i} (H_K,\AAA \otimes_{\OL} A_n)\]
	for all $i \geq 0$. Since $H^{i} (H_K,\AAA \otimes_{\OL} A_n) = 0$ for all $i >0$ and $n \in \NN$ we
	also have \linebreak $H^{i}(H_K,\AAA \otimes_{\OL} A)=0$ for all $i >0$.
	\end{proof}
	\end{cor}
	
	\begin{lem} \label{lemcofseq}
	Let $A$ be a cofinitely generated $\OL$-module. Then the sequence
	\[ \begin{xy} \xymatrix{
		0 \ar[r] &A \ar[r] & \AAA \otimes_{\OL} A \ar[rr]^-{\Fr\otimes\id - \id} & & \AAA \otimes_{\OL} A
			\ar[r] & 0.
	} \end{xy} \]
	is exact and has a continuous set theoretical splitting, where all terms are equipped with the
	discrete topology.
	\begin{proof}
	Since $\AAA$ is a flat $\OL$-module the first assertion comes from
	\hyperref[galc1.1]{Lemma \ref*{galc1.1}}, the second is obvious since all terms carry
	the discrete topology.
	\end{proof}
	\end{lem}

	\begin{prop} \label{torcofquas}
	Let $A$ be a cofinitely generated $\OL$-module with a continuous action from $G_K$. Then
	the exact sequence
	\[ \begin{xy} \xymatrix{
		0 \ar[r] &A \ar[r] & \AAA \otimes_{\OL} A \ar[rr]^-{\Fr\otimes\id - \id} & & \AAA \otimes_{\OL} A
			\ar[r] & 0.
	} \end{xy} \]
	and the canonical homomorphism
	\[ \begin{xy} \xymatrix{
			(\AAA \otimes_{\OL} A)^{H_K} \ar@{^{(}->}[r] &
			\Cctsb(H_K,\AAA \otimes_{\OL} A)		
	} \end{xy} \]
	induce quasi isomorphisms
	\begin{equation}\label{f:quasiH}
	    \begin{xy} \xymatrix{
		& \Cctsb(H_K,A) \ar[r]^-{\simeq} & \CFrb(H_K,\AAA \otimes_{\OL} A)
	 & \ar[l]_-{\simeq} \CphKLb(\MAK(A)).
	} \end{xy}
\end{equation}
	\begin{proof}
	Since $\mathrm{Fr}$ commutes with the action from $H_K$, the exact sequence
	\[ \begin{xy} \xymatrix{
		0 \ar[r] &A \ar[r] & \AAA \otimes_{\OL} A \ar[rr]^-{\Fr\otimes\id - \id} & & \AAA \otimes_{\OL} A
			\ar[r] & 0.
	} \end{xy} \]
	clearly is an exact sequence of (discrete) $H_K$-modules. Then
	\hyperref[cohomendo]{Corollary \ref*{cohomendo}} says that
	\[ \Hcts^{i}(H_K,A) \cong \HFr^{i}(H_K,\AAA \otimes_{\OL} A),\]
	which is exactly the first quasi isomorphism. For the second quasi isomorphism it is
	with \hyperref[pre7.21]{Proposition \ref*{pre7.21}} enough to show
	\[ \Hcts^{i}(H_K,\AAA \otimes_{\OL} A) =
		\begin{cases}
		\MAK(A) &,\text{ if } i=0\\
		0 &,\text{ else }.
		\end{cases} \]
	But this is exactly the above \hyperref[cofcomtriv]{Corollary \ref*{cofcomtriv}}.
	\end{proof}
	\end{prop}
	
	\begin{cor} \label{torcofseq}
	Let $A$ be a cofinitely generated $\OL$-module with a continuous action from $G_K$.
	Then the following sequence is exact
	\[ \begin{xy} \xymatrix{
		0 \ar[r] & \Hcts^0(H_K,A) \ar[r] & \MAK(A) \ar[rr]^{\phKL-\id} & & \MAK(A) \ar[r] & \Hcts^1(H_K,A) \ar[r] &0.
	} \end{xy} \]
	\begin{proof}
	This is the long exact cohomology sequence of
	\[ \begin{xy} \xymatrix{
		0 \ar[r] &A \ar[r] & \AAA \otimes_{\OL} A \ar[rr]^-{\Fr\otimes\id - \id} & & \AAA \otimes_{\OL} A
			\ar[r] & 0.
	} \end{xy} \]
	combined with $\Hcts^1(H_K,\AAA \otimes_{\OL} A)=0$ from \hyperref[cofcomtriv]{Corollary \ref*{cofcomtriv}}.
	\end{proof}
	\end{cor}
	
	
	In the next step, we want to replace the above exact sequence with a sequence of
	\glssymbol{LamK}-modules. An idea how to do this gives \Nekovar in \cite[(8.3.3) Corollary, p.\,159]{neksc} but
	unfortunately the modules we are working with are not ind-admissible, since $\AAA$ is no direct limit
	of finitely generated $\OL[G_K]$-modules. As in the proof of \hyperref[galc1.3]{Theorem \ref*{galc1.3}} we
	use limits and colimits to reduce to the case of discrete coefficients.
	
	
%
	
	We want to recall the notation from \cite[(8.1.1), p.\,148; (8.2.1), p.\,157]{neksc} and from the beginning
	of \cite[(8.3) Infinite extensions, p.\,158--159]{neksc}.
	
	\begin{mydef}
	Let $G$ be a profinite group, $U \leq G$ an open subgroup and $M$ a discrete
	$\OL[U]$-module. We then define the induced module to be
	\[ \gls{IndUG} \coloneqq \{ f \colon G \to X\mid f(ug)=uf(g) \text{ for all } u \in U, g \in G\}.\]
	$\IndUG(M)$ carries a $G$-action by $(g \cdot f)(\sigma) \coloneqq f(\sigma g)$.
	Furthermore, if $M$ is a discrete $\OL[G]$-module define
	\[ \gls{UM} \coloneqq \Hom_{\OL}(\OL[G/U],M).\]
	${}_U M$ then again carries a $G$-action by $(\sigma \cdot (f))(x) \coloneqq \sigma( f(\sigma^{-1}(x)))$.
	Let now $H\triangleleft G$ be a closed, normal subgroup and
	\glssymbol{UK} be the open subgroups of $G$ containing $H$. Then, for $V, U \in \UK$ with
	$V \subseteq U$ the canonical map $G/V \twoheadrightarrow G/U$ induces $\OL$-linear
	maps ${}_U M \hookrightarrow {}_V M$ under which the system $({}_U M)_{U \in \UK}$ becomes
	a filtered directed system. We then set
	\[ \glssymbol{FGH} \coloneqq \varinjlim_{U \in \UK} {}_U M.\]
	Similar as above, $\FGH(M)$ then also carries an action from $G$.
	If $H=\{1\}$ we write \gls{UKE} instead of $\UK$ and \gls{FG} instead of $F_{G/\{1\}}(M)$. Furthermore,
	we set $\UKK\coloneqq \mathcal{U}(G_K;H_K)$ and we write \gls{FGK} instead of $F_{G_K/H_K}$. This
	can lead to an abuse of notation, but it will be clear from the context, which construction is chosen.
	\end{mydef}

	\begin{rem} \label{indiso}
	For the above situation, \Nekovar proves in \emph{\cite[(8.1.3), p.\,149]{neksc}} that
	\[ \begin{xy} \xymatrix{
		\IndUG(M) \ar[r] & {}_U M, \ f \ar@{|->}[r] & \left[gU \mapsto g(f(g^{-1})) \right]
	} \end{xy} \]
	is a $G$-equivariant isomorphism.
	\end{rem}
	
	\begin{rem}
	In the above situation, if $f \in \FGH(M)$ then it exists $U \in \UK$ such that $f \in {}_U M$. If then
	$V \in \UK$ with $V \subseteq U$ we also have $f \in {}_V M$ as well as
	\[ f(gV) = f (gU) \]
	for all $g \in G$.
	\end{rem}
	

	\begin{rem} \label{abnormal}
	Let $G$ be a group and $H \triangleleft G$ a normal subgroup such that $G/H$ is abelian.
	Then every subgroup $U \leq G$ with $H \subseteq U$ is normal as well.
	In particular, if additionally $G$ is profinite and $H$ is closed, then the elements of $\UK$ are normal, open
	subgroups of $G$ containing $H$. This is of great interest for us, since our application of this theory will be
	$G=G_K$ and $H=H_K$ and $G=\GaK$ and $H=\{1\}$. In both cases, the factor $G/H$ is $\GaK$ which
	is abelian.
	\end{rem}
	
	\begin{prop}
	Let $G$ be a profinite group, $H \triangleleft G$ a closed, normal subgroup, $M$ a discrete
	$\OL[G]$-module and let $U \in \UK$. Then the compact-open topology on ${}_U M$ is
	discrete and the $G$-action on ${}_U M$ is again continuous with respect to this topology.\\
	Furthermore, the transition maps ${}_V M \to {}_{V'} M$ for $V,V' \in \UK$ with $V' \subseteq V$ are
	injective,
	the direct limit topology on $\FGH(M)$ is discrete and its $G$-action is continuous.
	\begin{proof}
	Since $U \leq G$ is an open subgroup, the set of cosets $G/U$ is finite and therefore
	$\OL[G/U]$ is a finitely generated free $\OL$-module. So in particular, $\OL[G/U]$ is compact.
	Then ${}_U M = \Hom_{\OL}(\OL[G/U],M)$ is discrete with respect to the compact open topology since
	$M$ is discrete. To see that the action from $G$ is continuous on ${}_U M$ it is enough to see that
	for every $f \in {}_U M$ there exists an open subset $V \subseteq G$ under which $f$ is fixed. Note also
	that $G$ acts by left multiplication on $G/U$. So, let
	$f \in {}_U M$ and let $g_1, \dots, g_n \in G$ be a set of representatives of the cosets of $G/U$. Since
	the action of $G$ on $M$ is continuous and $M$ carries the discrete topology, there exist
	open subsets $V_1, \dots, V_n \subseteq G$ such that $g_i$ is fixed by $V_i$ for all
	$1 \leq i \leq n$. Then f is fixed by $V \coloneqq \cap_{i} V_i$.\\
	The statements on $\FGH(M)$ follow immediately by taking the direct limit. So the statement on the
	transition maps is left. Let $V,V' \in \UK$ with $V'\subseteq V$. Then the canonical map
	$G/V' \to G/V$ is surjective. Then $\OL[G/V']\to \OL[G/V]$ is a surjective $\OL$-linear homomorphism
	and since $\Hom_{\OL}(-,M)$ is left exact, the induced homomorphism
	${}_{V'} M \to {}_V M$ is injective.
	\end{proof}
	\end{prop}
	
	In the above situation, under the additional assumption that $U$ is normal in $G$,
	\Nekovar introduces in \cite[(8.1.6.3) Conjugation, p.\,151]{neksc} an action from $G/U$ on ${}_U M$ which will
	be important for us. We recall this action in the following Remark and we prove the statements.
	
	\begin{rem} \label{guact}
	Let $G$ be a profinite group, $U \triangleleft G$ be an open, normal subgroup and $M$ a discrete
	$\OL[G]$-module. For $g \in G$ and $f \in \IndUG(M)$ we define \glssymbol{Adt} to be
	\[ (\Adt(g)(f) )(\sigma) \coloneqq g(f(g^{-1}\sigma)).\]
	This is an action from $G$ on $\IndUG(M)$ which is trivial on $U$, i.e. it induces
	an action from $G/U$ on $\IndUG(M)$ which we will denote also by $\Adt$.
	Since both, $\IndUG(M)$ and $G/U$ carry the
	discrete topology, this action is continuous.\\
	Furthermore,
	this action commutes with the standard action from $G$ and under the isomorphism $\IndUG(M) \cong {}_U M$
	from \emph{\hyperref[indiso]{Remark \ref*{indiso}}} it corresponds to the $G/U$-action
	\[ (\Adt(gU)(f))(\sigma U) = f(\sigma g U)\]
	on ${}_U M$. Then clearly the $G$-action on ${}_U M$ commutes with this action from $G/U$
	and the latter is again continuous.
	\end{rem}
	
	\begin{lem} \label{fghact}  
	Let $G$ be a profinite group and $H \triangleleft G$ a closed, normal subgroup, such that $G/H$ is
	abelian. Then $\Adt$ induces a continuous action from $G/H$ on $\FGH(M)$.\\
	In particular, with this action $\FGH(M)$ becomes an $\OLGH$-module.
	\begin{proof}
	The action from $G/H$ on $\FGH(M)$ is given as follows:
	For $f \in \FGH(M)$ and $U \in \UK$ such that $f \in {}_U M$ and $g \in G$ we have
	\[  \Adt(gH)(f) = \Adt(gU)(f).\]
	This is well defined, since if $V \in \UK$ such that $V \subseteq U$
	then $f \in {}_V M$ and for $\sigma \in G$ we have
	\[
	 \Adt(gU)(f)(\sigma U) = f(\sigma g U)
	 				= f(\sigma g V)
	 				= \Adt(gV)(f)(\sigma V).
	 \]
	The action is continuous since the above $f$ is fixed under $U/H$, which is an open
	subgroup of $G/H$.\\
	If $f$ is as above, $x \in \OLGH$ and $\pr_U \colon \OLGH \to \OL[G/U]$ denotes the
	canonical projection, then we have
	\[ \Adt(x)(f) = \Adt(\pr_U(x))(f).\]
	This again is well defined and makes $\FGH(M)$ into an $\OLGH$-module.
	\end{proof}
	\end{lem}
	
	\begin{prop}
	Let $G$ be a profinite group and $H \triangleleft G$ a closed, normal subgroup such that
	$G/H$ is abelian. Then
	$\FGH$ is an exact functor, viewed as functor from discrete \linebreak $\OL[G]$-modules to discrete
	$\OLGH[G]$-modules.
	\begin{proof}
	The above \hyperref[fghact]{Lemma \ref*{fghact}} says that $\FGH$ is a functor from
	discrete $\OL[G]$-modules to discrete $\OLGH[G]$-modules. So it is left to check that it is
	exact. For fixed \linebreak$U \in \UK$ the functor $M \mapsto {}_U M$ from discrete $\OL[G]$-modules
	to discrete $\OL[G/U][G]$-modules is exact since $\OL[G/U]$ is a finitely generated, free
	$\OL$-module. Since taking direct limits is exact as well, $\FGH$ is exact.
	\end{proof}
	\end{prop}
	
	\begin{mydef}
	If $\mathbf{C}$ is an abelian category, we denote by $\glssymbol{DA}$ the corresponding derived
	category. As usual, we denote by $\glssymbol{DPA}$ the full subcategory whose objects
	are the complexes, which have no nonnegative entries and by \glssymbol{DBA} the full subcategory
	 whose objects are the bounded below complexes.\\
	If $\Cfrb$ is a complex in an abelian category $\mathbf{C}$, we denote as in \cite{neksc} by \gls{RGa} the
	corresponding complex as an object in the derived category $\mathbf{R}(\mathbf{C})$. \\
	In particular, if $G$ is a profinite group and $M$ is a topological $G$-module we set
	\[ \gls{RGacts}\coloneqq \RGa(\Cctsb(G,M))\]
	as an object in $\mathbf{R}(\mathbf{Ab})$.
	\end{mydef}
	
	\begin{rem} \label{adgh}
	Let $G$ be a profinite group, $H\triangleleft G$ a closed, normal subgroup,
	and $M$ a discrete $\OL[G]$-module.
	As in \emph{\cite[(3.6.1.4), p.\,72]{neksc}} we define an action from $G$ on $\Cctsb(H,M)$ by
	\[ \glssymbol{Ad} \coloneqq g(c(g^{-1}h_0g,\dots,g^{-1}h_ng)),\]
	where $c \in \Ccts^n(H,M)$. In loc. cit. \Nekovar also proves that for $h \in H$ this action
	is homotopic to the identity and therefore induces an action from $G/H$ on
	$\RGactsb(H,M)$ and $H^*(H,M)$ respectively.\\
	Similarly, by
	\[ \begin{xy} \xymatrix{
		\Cctsb(G,\FGH(M)) \ar[r]^{\Adt(g)_*} & \Cctsb(G,\FGH(M)) \ar[r]^{\Ad(g)} & \Cctsb(G,\FGH(M))
	} \end{xy} \]
	we can define an action from $G$ on $\Cctsb(G,\FGH(M))$. Note that in this situation\linebreak
	$ \Ad(g)\colon \Cctsb(G,\FGH(M)) \to \Cctsb(G,\FGH(M))$
	is homotopic to the identity and so
	the complex $\RGactsb(G,\FGH(M))$ becomes a complex of $\OLGH$-modules. See also
	\emph{\hyperref[lamact]{Remark \ref*{lamact}}} below.
 %
%
	\end{rem}

	\begin{prop} \label{shaind}
	Let $G$ be a profinite group, $H\triangleleft G$ a closed, normal subgroup and $M$ a discrete $\OL[G]$-module.
	Then there is a canonical morphism of complexes
	\[ \Cctsb(G,\FGH(M)) \to \Cctsb(H,M),\]
	which is a quasi isomorphism. Moreover, for $g \in G$ the diagram
	\[ \begin{xy} \xymatrix{
		\Cctsb(G,\FGH(M)) \ar[r] \ar[d]_{\Adt(g)}& \Cctsb(H,M) \ar[dd]^{\Ad(g)} \\
		\Cctsb(G,\FGH(M)) \ar[d]_{\Ad(g)}	& \\
		\Cctsb(G,\FGH(M)) \ar[r] & \Cctsb(H,M)
	} \end{xy} \]
	is commutative. So in particular, the corresponding isomorphism \linebreak
	$\RGactsb(G,\FGH(M)) \to \RGactsb(H,M)$ in the derived category
	$\DDp(\OLMod)$ is $G/H$-linear.
	\begin{proof}
	For the proof set $\UU\coloneqq \UK$. \cite[(1.5.1) Proposition, p.\,45--46]{nswo} says that we have
	\[ \Cctsb(G,\FGH(M)) \cong \Cctsb(G,\varinjlim_{U \in \UU} {}_U M) \cong
	\varinjlim_{U \in \UU} \Cctsb(G, {}_U M). \]
	With \hyperref[indiso]{Remark \ref*{indiso}} we then obtain
	\[ \varinjlim_{U \in \UU} \Cctsb(G, {}_U M) \cong \varinjlim_{U \in \UU} \Cctsb(G, \IndUG(M)). \]
	Shapiro's Lemma (cf.\ \cite[(1.6.4) Proposition, p.\,62--63]{nswo}) and again \cite[(1.5.1) Proposition, p.\,45--46]{nswo}
	then give us
	\[ \varinjlim_{U \in \UU} \Cctsb(G, \IndUG(M)) \simeq \varinjlim_{U \in \UU} \Cctsb(U, M)
		\cong \Cctsb(\varprojlim_{U \in \UU} U, M) = \Cctsb(H,M).\]
	\cite[(8.1.6.3), p.\,151]{neksc} says that for $U \in \UK$ and $g \in G$ the diagram
	\[ \begin{xy} \xymatrix{
		\Cctsb(G, \IndUG(M)) \ar[r] \ar[d]_{\Adt(g)}& \Cctsb(U,M) \ar[dd]^{\Ad(g)} \\
		\Cctsb(G,\IndUG(M)) \ar[d]_{\Ad(g)}	& \\
		\Cctsb(G,\IndUG(M)) \ar[r] & \Cctsb(U,M)
	} \end{xy} \]
	is commutative. Taking direct limits then proves the commutativity of the desired diagram.
	\end{proof}
	\end{prop}
	
	\begin{cor} \label{lamlinder}
	Let $M \in \ModpGKe$ such that $M$ is discrete as $\OL[G]$-module. Then the above
	\emph{\hyperref[shaind]{Proposition \ref*{shaind}}} together with
	\emph{\hyperref[pre7.21]{Proposition \ref*{pre7.21}}} induces
	the $\GaK$-linear isomorphism
	\[ \begin{xy} \xymatrix{
		\RGa(\CphKLb(\Gamma_K,\FGK(M))) \ar[r]^-{\cong} & \RGa(\CphKLb(M)).
	} \end{xy} \]
	\end{cor}

	\begin{rem} \label{lamact}
	In the situation of \emph{\hyperref[shaind]{Proposition \ref*{shaind}}}, the morphism
	\[ \begin{xy} \xymatrix{
		\Ad(g)\colon \Cctsb(G,\FGH(M)) \ar[r] & \Cctsb(G,\FGH(M))
	} \end{xy} \]
	for $g \in G$ is homotopic to the identity (cf.\ \emph{\cite[(3.6.1.4), p.\,72]{neksc}} respectively
	\emph{\hyperref[adgh]{Remark \ref*{adgh}}}) and therefore the diagram
	\[ \begin{xy} \xymatrix{
		\RGactsb(G,\FGH(M)) \ar[r] \ar[d]_{\Adt(g)_*} & \RGactsb(H,M) \ar[d]^{\Ad(g)} \\
		\RGactsb(G,\FGH(M)) \ar[r] & \RGacts(H,M)
	} \end{xy} \]
	is commutative. The corresponding diagram for cohomology groups
	\[ \begin{xy} \xymatrix{
		\Hctss(G,\FGH(M)) \ar[r] \ar[d]_{\Adt(g)_*} & \Hctss(H,M) \ar[d]^{\Ad(g)} \\
		\Hctss(G,\FGH(M)) \ar[r] & \Hctss(H,M)
	} \end{xy} \]
	then also is commutative. This then explains that the statement from
	\emph{\cite[p.\,65]{nswo}} coincides with the theory from \Nekovar.
	\end{rem}
\pagebreak
	\begin{prop} \label{qind}
	Let $A=\varinjlim_m A_m$ be a cofinitely generated $\OL$-module,
	where $A_m=\ker(\mu_{\piL^m})$ as usual, with a continuous action from $G_K$ and set
	\begin{align*}
		\Amn & \coloneqq \left(\AAA \otimes_{\OL} A_m\right)/\left(\piL^n\AAAp \otimes_{\OL} A_m\right) \\
		M_{mn} & \coloneqq \left(\AAA \otimes_{\OL} A_m\right)^{H_K}/\left(\piL^n\AAAp
		\otimes_{\OL} A_m\right)^{H_K}.
	\end{align*}
	Then the following diagram is commutative and each arrow in it is a quasi isomorphism. Moreover,
	the vertical arrows on the right hand side are homomorphisms of $\LamK$-modules.
	\[
\xymatrix{
			\Cctsb(H_K,A) & & \Cctsb(G_K,\FGK(A)) \ar[ll]_-{\simeq} \\
			\underset{m \in \NN}{\varinjlim} \Cctsb(H_K,A_m) \ar@{}[d]^(.15){}="a"^(.9){}="b" \ar^-{\simeq} "a";"b"
				\ar[u]_-{\simeq}& &
				\underset{m \in \NN}{\varinjlim} \Cctsb(G_K,\FGK(A_m)) \ar[ll]_-{\simeq}
				\ar@{}[d]^(.15){}="a"^(.9){}="b" \ar_-{\simeq} "a";"b"
				\ar[u]^-{\cong} \\
			\underset{m \in \NN}{\varinjlim} \underset{n \in \NN}{\varprojlim} \CFrb(H_K,\Amn)
	& & \underset{m \in \NN}{\varinjlim} \underset{n \in \NN}{\varprojlim} \CFrb(G_K, \FGK(\Amn)) \ar[ll]_-{\simeq}\\
			\underset{m \in \NN}{\varinjlim} \underset{n \in \NN}{\varprojlim} \CphKLb(M_{mn})
			\ar@{}[u]^(.1){}="a"^(.85){}="b" \ar_-{\simeq} "a";"b"
	& & \underset{m \in \NN}{\varinjlim} \underset{n \in \NN}{\varprojlim}  \CphKLb(\GaK,\FGK(M_{mn})) \ar[ll]_-{\simeq}
	\ar@{}[u]^(.1){}="a"^(.85){}="b" \ar^-{\simeq} "a";"b" \\
			\underset{m \in \NN}{\varinjlim} \CphKLb(\MAK(A_m)) \ar@{}[u]^(.1){}="a"^(.8){}="b" \ar_-{\cong} "a";"b"
			\ar@{}[d]^(.15){}="a"^(.9){}="b" \ar^-{\simeq} "a";"b"& &\\
			\CphKLb(\MAK(A)). & &
	}
  \]
	In particular, the induced isomorphism $\RGa(\CphKLb(\MAK(A))) \cong \RGactsb(G_K,\FGK(A))$
	in \linebreak $\DDp(\OLMod)$ is $\LamK$-linear, i.e. it is an isomorphism in
	$\DDp(\LamKMod)$.
	\begin{proof}
	We start with the left column and we consider the following diagram
	\[ \begin{xy} \xymatrix{
			\Cctsb(H_K,A) &   \\
			\underset{m \in \NN}{\varinjlim} \Cctsb(H_K,A_m)
			\ar@{}[d]^(.15){}="a"^(.9){}="b" \ar_-{(2)}^-{\simeq} "a";"b" \ar[u]^-{(1)}_-{\simeq}& \\
	\underset{m \in \NN}{\varinjlim} \CFrb(H_K,\AAA\otimes_{\OL} A_m)
	\ar@{}[dr]^(.15){}="a"^(.87){}="b" \ar^-{(5)}_-{\cong} "a";"b" &\\
			& \underset{m \in \NN}{\varinjlim} \underset{n \in \NN}{\varprojlim} \CFrb(H_K,\Amn) \\
			& \underset{m \in \NN}{\varinjlim} \underset{n \in \NN}{\varprojlim} \CphKLb(M_{mn})
			\ar@{}[u]^(.1){}="a"^(.85){}="b" \ar_-{(7)}^-{\simeq} "a";"b"  \\
			\underset{m \in \NN}{\varinjlim} \CphKLb(\MAK(A_m))
			\ar@{}[ur]^(.2){}="a"^(.67){}="b" \ar_-{(6)}^-{\cong} "a";"b"
			\ar@{}[d]^(.15){}="a"^(.9){}="b" \ar_-{(4)}^-{\simeq} "a";"b"
			\ar@{}[uuu]^(.05){}="a"^(.95){}="b" \ar^-{(3)}_-{\simeq} "a";"b" &\\
			\CphKLb(\MAK(A)). &
	} \end{xy} \]
	That the morphisms (1) and (4) are quasi isomorphisms is well known (cf.\ eg.
	\cite[(1.5.1) Proposition, p.\,45--46]{nswo}). (2) and (3) are quasi isomorphisms by
	\hyperref[torcofquas]{Proposition \ref*{torcofquas}}.
	\hyperref[galc1.8]{Proposition \ref*{galc1.8}} says that (5) and (6) are isomorphisms of complexes. But then
	(7) is also a quasi isomorphism. So, all the morphisms in the left column of the original diagram
	are at least quasi isomorphisms. The horizontal morphisms are quasi isomorphisms by
	\hyperref[shaind]{Proposition \ref*{shaind}} and therefore the morphisms in the right
	column are also quasi isomorphisms. So it is left to check that the induced isomorphism
	$\RGa(\CphKLb(\MAK(A))) \cong \RGactsb(G_K,\FGK(A))$ is $\LamK$-linear. But the morphisms
	\[ \begin{xy} \xymatrix{
		\underset{m \in \NN}{\varinjlim} \CphKLb(\MAK(A_m)) \ar[r] & \CphKLb(\MAK(A))
	} \end{xy} \]
	and
	\[ \begin{xy} \xymatrix{
		\underset{m \in \NN}{\varinjlim} \underset{n \in \NN}{\varprojlim} \CphKLb(M_{mn}) \ar[r] &
		\underset{m \in \NN}{\varinjlim} \CphKLb(\MAK(A_m))
	} \end{xy} \]
	are clearly $\LamK$-linear and so are all the morphisms in the right column of the original
	diagram with respect to the $\LamK$-action induced by $\Adt$ (which is
	the correct action in the derived category according to \hyperref[lamact]{Remark \ref*{lamact}}).
	Finally, the morphism
	\[ \begin{xy} \xymatrix{
		\underset{m \in \NN}{\varinjlim} \underset{n \in \NN}{\varprojlim} \RGa(\CphKLb(\Gamma_K,\FGK(M_{mn})))
		\ar[r] & \underset{m \in \NN}{\varinjlim} \underset{n \in \NN}{\varprojlim} \RGa(\CphKLb(M_{mn}))
	} \end{xy} \]
	is $\LamK$-linear by \hyperref[lamlinder]{Corollary \ref*{lamlinder}}.
	\end{proof}
	\end{prop}
This description now has the advantage that the objects of the complexes
	are $\LamK$-modules which allows us to apply the theory of Matlis duality. We give a brief
	overview of this theory.
	\begin{rem} \label{actonlam}
	We have to consider different types of group actions on $\LamK$. First, $\GaK$ acts by
	multiplication and $G_K$ acts by multiplication through the natural projection \linebreak
	$\pr\colon G_K \twoheadrightarrow \GaK$. Sometimes we also have to consider $\LamK$ as
	$\LamK$-module via the
	involution \gls{iota}, i.e. $\GaK$ then acts by $\gamma \cdot x \coloneqq \gamma^{-1} x$. If this is the case,
	we write \glssymbol{Miota}. Note that this does also affect the action from $G_K$, i.e. $G_K$ acts on
	$\LamK^{\iota}$ by $g \cdot x = \pr(g)^{-1}x$ and $\GaK$ acts by $\gamma \cdot x = \gamma^{-1} x$.
	\\
	Additionally, if $M$ is a $\LamK$-module, we denote by $M^{\iota}$ the $\LamK$-module
	$M$ where $\GaK$ acts via the involution $\iota$, i.e. for all $\gamma \in \GaK$ and $m \in M$ we
	have $\gamma \cdot m = \gamma^{-1} m$. If $N$ is another $\LamK$-module we clearly have
	\[ \Hom_{\LamK}(M,N^{\iota})=\Hom_{\LamK}(M^{\iota},N).\]
	\end{rem}
	
	\begin{mydef}
	A $\LamK$-module with a $\LamK$-semilinear action of $G_K$ is a $\LamK$-module $M$
	with an action from $G_K$ such that for all
	$\lambda \in \LamK$, $m \in M$ and $g \in G_K$ we have
	\[ g(\lambda m)= g(\lambda) g(m) =\pr(g)\lambda g(m),\]
	where $\pr\colon G_K \twoheadrightarrow \GaK$ denotes the canonical projection
	(cf.\ \hyperref[actonlam]{Remark \ref*{actonlam}}).
	\end{mydef}
	
	\begin{rem} \label{gkonlam}
	For us it feels more natural to consider $\LamK$-modules with a semilinear   $G_K$-action instead of
	$\LamK$-modules with a
	linear action from $G_K$, which are considered in \emph{\cite{neksc}}. The main reason for this is that if we consider
	modules with a linear action from $G_K$ we would have to consider $\LamK$ with the
	trivial action from $G_K$. But this feels nonintuitive. In the text below we will always compare
	our results to the results of \Nekovar in \emph{\cite{neksc}}. He considers $\LamK$ with the trivial action
	of $G_K$ (cf.\ \emph{\cite[(8.4.3.1) Lemma, p.\,161--162]{neksc}}).\\
	Both concepts are linked in the following sense: If $M$ is a $\LamK$-module with a (linear
	or semilinear) action from $G_K$, then for $n \in \ZZ$ denote by $M<n>$ the $\LamK$-module $M$
	with the $G_K$-action given by
	\[ g \cdot m = \pr(g)^n g(m),\]
	with $g \in G_K$ and $m \in M$ and where $g(m)$ denotes the given action of $G_K$ on $M$
	(cf.\ \emph{\cite[(8.4.2), p.\,161]{neksc}}). Then $M \mapsto M<1>$ induces a morphism from $\LamK$-modules
	with a linear action from $G_K$ to $\LamK$-modules with a semilinear action from $G_K$. Its inverse
	clearly is $M \mapsto M<-1>$.
	\end{rem}
	
	\begin{rem} \label{actonlamod}
	Let $M,N$ be $\LamK$-modules with a $\LamK$-semilinear action of $G_K$.
	Then \linebreak $\Hom_{\LamK}(M,N)$ also carries actions from
	both $G_K$ and $\GaK$ (respectively $\LamK$). The action from $\GaK$ is given by
	the multiplication of $\LamK$ on $N$ (respectively $M$ since the homomorphisms are $\LamK$-linear).
	The action from $G_K$ is given by
	\[ (g \cdot f)(m) \coloneqq g_N(f(g_M^{-1}(m))),\]
	for $f \in \Hom_{\LamK}(M,N)$ and $m \in M$ and where $g_M$ respectively $g_N$ denote the
	actions from $G_K$ on $M$ and $N$.
	\end{rem}
	
	\begin{rem} \label{actonhom}
	Let $T$ be a topological $\OL$-module with a continuous action from $G_K$ and let $M$
	be a $\LamK$-module with a $\LamK$-semilinear action of $G_K$. Then $\GaK$ acts on $\HomOL(T,M)$ by
	multiplication on the coefficients and $G_K$ as in the above
	\emph{\hyperref[actonlamod]{Remark \ref*{actonlamod}}}, i.e. by
	\[ (g \cdot f)(t) \coloneqq g_M(f(g_T^{-1}(t))),\]
	for $f \in \HomOL(T,M)$ and $m \in M$ and where $g_T$ and $g_M$ denote the
	actions from $G_K$ on $T$ and $M$ respectively.
	\end{rem}
	
	\begin{lem} \label{lemlamhom}
	Let $T$ be a topological $\OL$-module with a continuous action from $G_K$ and let $M$
	be a $\LamK$-module with a $\LamK$-semilinear action of $G_K$. Then the homomorphism of $\OL$-modules
	\[ \begin{xy} \xymatrix{
	\HomOL(T,M) \ar[r] & \Hom_{\LamK}(T \otimes_{\OL} \LamK, M), \ f \ar@{|->}[r] & \beta_f \coloneqq [t \otimes x
	\mapsto xf(t)]
	} \end{xy} \]
	is an isomorphism which respects the actions from $\GaK$ and $G_K$
	described in the above \emph{\hyperref[actonhom]{Remark \ref*{actonhom}}} for the left
	hand side and \emph{\hyperref[actonlam]{Remark \ref*{actonlam}}} for the right hand side.
	\begin{proof}
	The inverse homomorphism is given by
	\[ \begin{xy} \xymatrix{
	\Hom_{\LamK}(T\otimes_{\OL} \LamK,M) \ar[r] & \HomOL(T,M), \ h \ar@{|->}[r] &
	[t \mapsto h(t \otimes 1)].
	} \end{xy} \]
	So it is left to check that the above homomorphisms respects the actions from $\GaK$ and $G_K$, which we leave to the reader.
	\end{proof}
	\end{lem}
	
	\begin{rem} \label{gakonpodu}
	Let $M$ be a $\LamK$-module with a $\LamK$-semilinear action of $G_K$.
	Then $M^{\vee}=\HomctsOL(M,L/\OL)$ also carries actions from
	$G_K$ and $\GaK$. Both are given by
	\[ (g \cdot f)(m) = f(g^{-1}(m)),\]
	where $g \in G_K$ or in $\GaK$, $f \in M^{\vee}$ and $m \in M$.\\
	Note that \Nekovar considers the Pontrjagin dual of $M$ with the $\GaK$-action without the
	involution, i.e. by $(\gamma\cdot f)(m)=f(\gamma(m))$ (cf.\ the proof respectively the result
	of \emph{\cite[(8.4.3.1) Lemma, p.\,161--162]{neksc}}). In our notation the Pontrjagin
	dual of \Nekovar of $M$ is $(M^{\vee})^{\iota}=(M^{\iota})^{\vee}$.
	\end{rem}
	
	\begin{rem} \label{gkpondushi}
	Let $M$ be a $\LamK$-module with a $\LamK$-semilinear action of $G_K$ and $n \in \ZZ$. Then
	the identity of $M^{\vee}$ induces an isomorphism
	of $\LamK$-modules with a $\LamK$-semilinear action of $G_K$
	\[ (M<n>)^{\vee} \cong M^{\vee}<n>.\]
	\end{rem}
	
	\begin{mydef}
	Let $M$ be a $\LamK$-module. The {\bfseries Matlis dual} \index{Matlis dual} of $M$ is
	defined as
	\[ \glssymbol{ODK} \coloneqq \Hom_{\Lambda_K}(M,\Lambda_K^{\vee}).\]
	This is a contravariant functor of $\LamK$-modules and maps finitely generated $\LamK$-modules
	to cofinitely generated and vice versa.\\
	$\LamK$ acts on $\ODK(M)$ by multiplication and if $M$ has also a semilinear action from $G_K$, then
	$G_K$ acts on $\ODK(M)$ as described in the above
	\hyperref[actonlamod]{Remark \ref*{actonlamod}}
	\end{mydef}
	
	\begin{rem} \label{matduex}
	$\Lambda_K^{\vee}$ is an injective $\Lambda_K$-module. Moreover, it is an injective hull
	of the residue class field of $\LamK$ as $\LamK$-module. Therefore $\ODK$ is exact and for every finitely
	respectively cofinitely generated $\LamK$-module the canonical homomorphism
	$M \to \ODK(\ODK(M))$ is an isomorphism.
	\begin{proof}
	Since $\gamma \mapsto \gamma^{-1}$ defines an isomorphism of $\LamK$-modules $\LamK \to \LamK^{\iota}$,
	the first statement is \cite[(8.4.3.2) Corollary, p.\,162]{neksc}. For this, note that in
	\cite[(8.4.3.1) Lemma, p.\,161--162]{neksc} \Nekovar proves that $(\LamK^{\vee})^{\iota}=(\LamK^{\iota})^{\vee}$
	and \Nekovars dualizing module coincide and with
	$(\LamK^{\iota})^{\vee}$ also $\LamK^{\vee}$ is a dualizing module.
	The second statement is \cite[Theorem 3.2.12, p.\,105--107]{bruher}.
	\end{proof}
	\end{rem}

	\begin{rem} \label{homdu}
	As mentioned in \emph{\cite[(2.3.3, p.\,41)]{neksc}} $L/\OL$ is an injective hull for $k_L$. Therefore
	we have a canonical isomorphism $M \cong \HomOL(\HomOL(M,L/\OL),L/\OL)$ for every finitely
	or cofinitely generated $\OL$-module $M$ and $\HomOL(-,L/\OL)$ is an exact functor. As above,
	the proof for this is \emph{\cite[Theorem 3.2.12, p.\,105--107]{bruher}}.
	\end{rem}

	We need some more notation from \cite{neksc}.

	\begin{rem} \label{actfcgk}
	Let $T \in \RepGKOLfg$ and $U \in \UKK$. Then we have two group actions on \linebreak
	$T \otimes_{\OL} \OL[G_K/U]$.
	The first action, is the diagonal action from $G_K$
	\[ g \cdot (a \otimes xU) = (ga)\otimes (gxU).\]
	The second action is the following action from $G_K/U$:
	\[ \Adt(gU)(a \otimes xU) \coloneqq a \otimes xg^{-1}U.\]
	The homomorphism $\sum a_{xU} \otimes xU \mapsto \sum a_{xU} \delta_{xU}$ where $\delta_{xU}$ is the
	Kronecker delta-function on $G_K/U$ (i.e. it is $1$ for $xU$ and zero otherwise) defines
	an isomorphism between \linebreak $T \otimes_{\OL} \OL[G_K/U]$ and ${}_U T$
	(cf.\ \emph{\cite[(8.1.3), p.\,149; (8.2.1) p.\,157]{neksc}}) under which the actions described above coincide with
	the corresponding actions on ${}_U T$ (cf.\ \emph{\cite[(8.1.6.3), p.\,151]{neksc}}).
	\end{rem}

	\begin{mydef}
	Let $T \in \RepGKOLfg$. We set
	\[ \glssymbol{FCK}\coloneqq \varprojlim_{U \in \UKK} T \otimes_{\OL} \OL[G_K/U]\]
	together with the two actions from $G_K$ and $\GaK$ described in the above
	\hyperref[actfcgk]{Remark \ref*{actfcgk}}.
	With this, we define
	\[ \glssymbol{RGaiw}\coloneqq \RGactsb(G_K,\FCK(T)).\]
	Furthermore, by \gls{tender} we denote the derived tensor product over the ring $R$.	
	\end{mydef}
	
	\begin{rem}
	At \emph{\cite[p.\,201]{neksc}} \Nekovar proves
	\[ \HIW^*(\Kin|K,T) \cong H^*(\RGaiwb(\Kin|K,T)),\]
	i.e. that the cohomology of the above complex coincides with the Iwasawa cohomology
	defined in \emph{\hyperref[iwcohdef]{Definition \ref*{iwcohdef}}}.
	\end{rem}

	\begin{rem} \label{FCKLam}
	Let $T \in \RepGKOLfg$, then we have an isomorphism of $\LamK$-modules
	with a $\LamK$-semilinear action of $G_K$
	\[ \FCK(T) \cong T \otimes_{\OL} \LamK^{\iota}.\]
	\begin{proof}
	Since $T$ is finitely generated and $\OL$ is a discrete valuation ring, $T$ is finitely presented.
	Therefore we have
	\[ \varprojlim_{U \in \UKK} T \otimes_{\OL} \OL[G_K/U] = T \otimes_{\OL} \LamK\]
	as $\OL$-modules. $G_K$ acts on both sides diagonally and $\GaK$ acts on the left hand side
	via $\Adt$ (which technically means via the involution) on the right hand term $\OL[G_K/U]$. Since
	$\GaK$ acts on $\LamK^{\iota}$ also via the involution, the claim follows.
	\end{proof}
	\end{rem}
	
	\begin{lem}  \label{lamdu}
	We have an isomorphism of $\LamK$-modules with a $\LamK$-semilinear action of $G_K$
	\[ (\LamK^{\iota})^{\vee} \cong \FGK(L/\OL) \left(=
	\varinjlim_{U \in \UKK} \HomOL(\OL[G_K/U],L/\OL)\right).\]
	\begin{proof}
	$\OL[G_K/U]$ is compact for $U \in \UKK$, therefore $\HomOL(\OL[G_K/U],L/\OL)$ is discrete and so
	$\varinjlim_{U \in \UKK} \HomOL(\OL[G_K/U],L/\OL)$ is discrete too. This means that every map with source
	$\varinjlim_{U \in \UKK} \HomOL(\OL[G_K/U],L/\OL)$ into any topological space is continuous. We then compute
	(as $\OL$-modules)
	\begin{align*}
			\HomcOL(\FGK(L/\OL),L/\OL)
				&=\HomcOL(\varinjlim_{U \in \UKK} \HomOL(\OL[G_K/U],L/\OL),L/\OL)\\
				= &\HomOL(\varinjlim_{U \in \UKK} \HomOL(\OL[G_K/U],L/\OL),L/\OL) \\
				\cong &\varprojlim_{U \in \UKK}\HomOL( \HomOL(\OL[G_K/U],L/\OL),L/\OL) \\
				\cong&\varprojlim_{U \in \UKK} \OL[G_K/U]\\
				= &\LamK.
	\end{align*}
	At the third equation, we used the identification
	\[\OL[G_K/U] \cong \HomOL( \HomOL(\OL[G_K/U],L/\OL),L/\OL)\]
	from \hyperref[homdu]{Remark \ref*{homdu}}. Now we head towards the action from $\GaK$. For
	$\gamma \in \GaK$, \linebreak
	$f \in \HomcOL(\FGK(L/\OL),L/\OL)$ and $h\in \FGK(L/\OL)$ we have
	\[ (\gamma \cdot f)(h)=f(\gamma^{-1} \cdot h) = f(\Adt(\gamma^{-1})h)\]
	for all $x \in \FGK(L/\OL)$. Going through the above isomorphisms shows that this results in
	an action from $\GaK$ on $\LamK$ via the involution, i.e. we have an isomorphism of
	$\LamK$-modules
	\[ \HomcOL(\FGK(L/\OL),L/\OL) \cong \LamK^{\iota}.\]
	With the above notation, we have for $g \in G_K$
	\[ (g \cdot f)(h)= f(g^{-1}\cdot h ) = f(h \circ g),\]
	since $G_K$ acts trivial on $L/\OL$ by definition. Therefore the above isomorphism is also
	$G_K$-linear.
	\end{proof}
	\end{lem}

	\begin{rem} \label{nekdiff1}
	The above result differs a bit from \Nekovars result in \emph{\cite[(8.4.3.1) Lemma, p.\,161--162]{neksc}} since
	\Nekovar considers $\LamK$-modules with a $\LamK$-linear action from $G_K$ and therefore
	he considers $\LamK$ with a trivial $G_K$ action (cf.\ \emph{\hyperref[gkonlam]{Remark \ref*{gkonlam}}}).
	Furthermore, his Pontrjagin dual and ours for $\LamK$-modules differ in the action of $\GaK$ by
	an involution (cf.\ \emph{\hyperref[gakonpodu]{Remark \ref*{gakonpodu}}}).
	For a better comparison, if we consider $\LamK$ with the trivial action from
	$G_K$ the result of loc. cit in our notation is
	\[ (\LamK^{\vee})^{\iota} \cong \FGK(L/\OL)<1>. \]
	This is equivalent to
	\[ (\LamK^{\vee})^{\iota}<-1> \cong \FGK(L/\OL)\]
	and for the left hand side we obtain
	\begin{align*}
		(\LamK^{\vee})^{\iota}<-1> &= (\LamK^{\iota})^{\vee} <-1> \\
							&= (\LamK^{\iota}<-1>)^{\vee}\\
							&=((\LamK<1>)^{\iota})^{\vee}. 				
	\end{align*}
	In the second line we used \emph{\hyperref[gkpondushi]{Remark \ref*{gkpondushi}}}. But this means that
	\Nekovars result translate into ours since we considered $\LamK$ with the action from
	$G_K$ given by the canonical projection $\pr\colon G_K \twoheadrightarrow \GaK$.
	\end{rem}

	\begin{lem} \label{fdualf}
	Let $T\in \RepGKOLfg$. Then we have an isomorphism of $\LamK$-modules with a $\LamK$-semilinear
	action of $G_K$:
	\[ \FGK(T^{\vee}) \cong \ODK(\FCK(T)).\]
	\begin{proof}
	This proof follows similarly as in \cite[(8.4.5.1) Lemma, p.\,163]{neksc}, see \cite[5.2.42]{kupferer} for details.\\
	\end{proof}
	\end{lem}
	
	\begin{rem}
	Again, the above result differs slightly from the analogous result of \Nekovar
	(cf.\ \emph{\cite[(8.4.5.1) Lemma, p.\,163]{neksc}}). This is a consequence of the difference pointed out in the above
	\emph{\hyperref[nekdiff1]{Remark \ref*{nekdiff1}}}. Translated to our notation, \Nekovars result from
	(loc.\ cit.) then is that there is an isomorphism of $\LamK$-modules with a $\LamK$-semilinear action of $G_K$
	\[ \FGK((T^{\vee})^{\iota}) \cong \Hom_{\LamK}(\FCK(T)^{\iota},(\LamK^{\vee})^{\iota}).\]
	Note that \Nekovars original result is formulated for $\LamK$-modules with a linear action from
	$G_K$. But as pointed out in \emph{\hyperref[gkonlam]{Remark \ref*{gkonlam}}} both concepts are linked
	by the shifts $<1>$ and $<-1>$ respectively. So to be precise, \Nekovars result is the above shifted
	by $<-1>$. If we apply this shift, we would have to invert it below in order to compare \Nekovars result to our
	result.
	Since $\GaK$ acts trivially on $T$ and therefore also on $T^{\vee}$ we have
	$(T^{\vee})^{\iota}=T^{\vee}$ and we have a canonical isomorphism of
	$\LamK$-modules with a $\LamK$-semilinear action of $G_K$
	\[ \Hom_{\LamK}(\FCK(T)^{\iota},(\LamK^{\vee})^{\iota})=
		\Hom_{\LamK}(\FCK(T),\LamK^{\vee})=\ODK(\FCK(T)). \]
	Combining the above identifications then gives us an isomorphism of
	$\LamK$-modules with a $\LamK$-semilinear action of $G_K$
	\[ (\FGK(T^{\vee})) \cong \ODK(\FCK(T)),\]
	which is exactly our result.
	\end{rem}

	\begin{lem} \label{iwdcts}
	Let $T \in \RepGKOLfg$. We then have an isomorphism
	\[ \RGaiwb(\Kin|K,T) \cong \ODK\left(\RGactsb(G_K, \FGK(T^{\vee})(1))\right) [-2].\]
	For the cohomology groups we then have
	for all $i \geq 0$ an isomorphism of $\LamK$-modules
	\[ \ODK(\HIW^{i}(\Kin|K,T))\cong\Hcts^{2-i}(G_K,\FGK(T^{\vee}(1)))\cong \Hcts^{2-i}(H_K,T^{\vee}(1)).\]
	\begin{proof}
	This is \cite[(8.11.2.2); (8.11.2.3), p.\,201]{neksc}, but note that the shift of our complex is outside $\ODK(-)$
	and
	that we have
	$ \FGK(T^{\vee}) \cong \ODK(\FCK(T))$ (cf.\ \hyperref[fdualf]{Lemma \ref*{fdualf}}) since
	we have a slightly different convention for the involved action of $\GaK$.
	In particular, this is \hyperref[fdualf]{Lemma \ref*{fdualf}}
	together with \cite[(5.2.6) Lemma, p.\,92]{neksc}.
	The last isomorphism of the cohomology groups is
	\hyperref[shaind]{Proposition \ref*{shaind}}.
	\end{proof}
	\end{lem}

	\begin{prop} \label{secdual}
	Let $T \in \RepGKOLfg$. Then the
	sequence
	\[ \begin{xy} \xymatrix@C-0.6pc{
		0 \ar[r] & \HIW^1(\Kin|K,T) \ar[r] & \ODK\left(\mathcal{M}\right) \ar[rrr]^-{\ODK(\phKL)-\id} & & &
		\ODK\left(\mathcal{M}\right) \ar[r] & \HIW^2(\Kin|K,T) \ar[r] &0
	} \end{xy} \]
	is exact, where $\mathcal{M}=\MAK(T^{\vee}(1))$.
	\begin{proof}
	With $A\coloneqq T^{\vee}(1)$ we deduce from \hyperref[torcofseq]{Proposition \ref*{torcofseq}}
	and \hyperref[shaind]{Proposition \ref*{shaind}}
	that the sequence
	\[ \begin{xy} \xymatrix@C=0.9pc{
		0 \ar[r] & \Hcts^0(G_K,\FGK(A)) \ar[r] & \MAK(A) \ar[rrr]^{\phKL-\id}
		& & & \MAK(A) \ar[r] & \Hcts^1(G_K,\FGK(A)) \ar[r] &0
	} \end{xy} \]
	is exact and \hyperref[qind]{Proposition \ref*{qind}} says that it is a sequence of $\LamK$-modules.
	Applying $\ODK(-)$ then gives the exact sequence
	\[ \begin{xy} \xymatrix@R-2pc@C-0.7pc{
		0 \ar[r] & \ODK(\Hcts^1(G_K,\FGK(A))) \ar[r] & \ODK(\MAK(A)) \ar[rr]^-{\ODK(\phKL)-\id}
		 & & \cdots \\ \cdots \ar[r] & \ODK(\MAK(A)) \ar[r] & \ODK(\Hcts^0(G_K,\FGK(A))) \ar[rr] & & 0
	} \end{xy} \]
	(cf.\ \hyperref[matduex]{Remark \ref*{matduex}}). \hyperref[iwdcts]{Lemma \ref*{iwdcts}} translates this
	sequence into the desired one.
	\end{proof}
	\end{prop}
	
	This sequence looks similar to the sequence
	\[ \begin{xy} \xymatrix@C=0.8pc{
		0 \ar[r] & \HIW^1(\Kin|K,T) \ar[r] & \MAK(T(\tau^{-1})) \ar[rr]^{\psi-\id} & &
		\MAK(T(\tau^{-1})) \ar[r] & \HIW^2(\Kin|K,T) \ar[r] &0
	} \end{xy} \]
	from \hyperref[loctat13]{Theorem \ref*{loctat13}}
	where $\tau^{-1}=\chLT\chcyc^{-1}$ and
	$T \in \RepGKOLfg$. In order to compare
	these sequences, we prove the following.
	
	\begin{lem} \label{incdual1}
	Let $n \in \NN$. We have
	$\OmeAAK/\piL^n\OmeAAK
		= (\AAK/\piL^n\AAK)^{\vee}$ and a $\GaK$-linear inclusion
		\[ \begin{xy} \xymatrix{
		\OmeAAK/\piL^n\OmeAAK \ar@{^{(}->}[r] 				
			& \ODK(\AAK/\piL^n\AAK)	
		} \end{xy} \]
	\begin{proof}
	The isomorphism is a reformulation of an analogue of \cite[Lemma 3.5, p.\,11]{SV15}. For the inclusion
	using the tensor-hom adjunction   we obtain
	\begin{align*}
		\HomOL(\AAK/\piL^n\AAK,L/\OL)  & \cong
		\HomOL(\AAK/\piL^n\AAK \otimes_{\Lambda_K} \Lambda_K, L/\OL) \\
			& \cong \Hom_{\Lambda_K}(\AAK/\piL^n\AAK,\HomOL(\Lambda_K,L/\OL)).
	\end{align*}
	So we have to check that under this isomorphism
	$\HomcOL(\AAK/\piL^n\AAK,L/\OL)$ is sent to
	$\Hom_{\Lambda_K}(\AAK/\piL^n\AAK,(\Lambda_K)^{\vee})$.
	For this, recall the above isomorphism precisely: Let $f \in \HomcOL(\AAK/\piL^n\AAK,L/\OL)$,
	then $f$ is mapped to the element
	\[ [a \mapsto f_a \coloneqq [\lambda \mapsto f(\lambda a)]]\]
	in $ \Hom_{\Lambda_K}(\AAK/\piL^n\AAK,\HomOL(\Lambda_K,L/\OL))$.
	For $a \in \AAK/\piL^n\AAK$ the homomorphism $f_a$ then is the composition
	\[ \begin{xy} \xymatrix@R-2pc{
		\Lambda_K \ar[r] & \AAK/\piL^n\AAK \ar[r]^-{f} & L/\OL \\
		\lambda \ar@{|->}[r] & \lambda a &
	} \end{xy} \]
	of continuous maps, i.e. $f_a$ is continuous too and we get the desired
	inclusion
	\[ \begin{xy} \xymatrix{
		\HomcOL(\AAK/\piL^n\AAK, L/\OL) \ar@{^{(}->}[r] 		
	&  \Hom_{\Lambda_K}(\AAK/\piL^n\AAK,(\Lambda_K)^{\vee}).
	} \end{xy} \]
	It is easy to check this inclusion is $\GaK$-linear.
	\end{proof}
	\end{lem}
						
	\begin{mydef}
	Let $M$ be a topological $\AAK$-module with a continuous and semilinear action from $\GaK$. We define
	\[ \CD(M) \coloneqq \Hom_{\AAK}(M,\OmeAAK \otimes_{\AAK} \BBK/\AAK).\]
	And we define the $\GaK$-action on $\CD(M)$ to be
	\[ (\gamma \cdot f)(m) \coloneqq \gamma(f(\gamma^{-1}(m))),\]
	where $\GaK$ acts diagonal on the tensor product.
	\end{mydef}
	
	\begin{rem}
	Using the isomorphism $ \AAK(\chLT) \to \OmeAK, \ f \otimes t_0 \mapsto f \gLT \mrmd Z$ we can identify $\CD(M)$, for $M$ as above,
	with
	\[ \Hom_{\AAK}(M,\BBK/\AAK(\chLT)).\]
	\end{rem}
	
	\begin{lem} \label{incdual2}
	Let $M$ be a discrete $\AAK$-module with a continuous and semilinear action from $\GaK$
	such that $M=\varinjlim_m M_m$ where $M_m =\ker(\mu_{\piL^m})$.
	Then we have a $\GaK$-linear inclusion
		\[ \begin{xy} \xymatrix{
				\CD(M) \ar@{^{(}->}[r] & \ODK(M).					
		} \end{xy} \]
	\begin{proof}
	For $m \in \NN$ we obtain with the tensor-hom adjunction
	\begin{align*}
		\ODK(M_m)
			& \cong \Hom_{\Lambda_K}(M_m,(\Lambda_K)^{\vee}) \\
			&\cong \Hom_{\LamK}(M_m \otimes_{\AAK} \AAK/\piL^m\AAK, (\LamK)^{\vee})\\
			&\cong \Hom_{\AAK}(M_m, \Hom_{\Lambda_K}(\AAK/\piL^m\AAK/,(\LamK)^{\vee})).
	\end{align*}
	\hyperref[incdual1]{Lemma \ref*{incdual1}} then implies, that there is an inclusion
	\[ \begin{xy} \xymatrix{
	\Hom_{\AAK}(M_m,\OmeAAK/\piL^m\OmeAAK) \ar@{^{(}->}[r] & \ODK(M_m).			
	} \end{xy} \]
	But since $\piL^m M_m = 0$ it is
	\[ \Hom_{\AAK}(M_m,\OmeAAK/\piL^m\OmeAAK)=
		\Hom_{\AAK}(M_m,\OmeAAK \otimes_{\AAK}  \BBK/\AAK),\]
	i.e. we have an inclusion $\CD(M_m) \hookrightarrow \ODK(M_m)$. Since $\Hom_R(-,X)$
	commutes with limits for arbitrary rings $R$ and $R$-modules $X$, we get the desired inclusion
	$\CD(M) \hookrightarrow \ODK(M)$ by applying limits.
	\end{proof}
	\end{lem}

	\begin{lem} \label{incdual3}
	Let $A$ be a cofinitely generated $\OL$-module with a continuous action from $G_K$.
	Then we have
	\[ \CD(\MAK(A))\cong \MAK(A^{\vee}(\chLT)).\]
	This isomorphism respects the action from $\GaK$.
	\begin{proof}
	As usual we write $A=\varinjlim_{m} A_m$ with $A_m = \ker(\mu_{\piL^m})$. By an analogue of \cite[Lemma 3.6, p.\,11--12]{SV15} we have an isomorphism
	\[ \CD(\MAK(A_m)) \cong \MAK(A_m)^{\vee},\]
	which is $\GaK$-linear by similar arguments as   in (the proofs of)  \cite[Corollary 3.18, Proposition 3.19]{SV15}.
	\hyperref[remdual]{Remark \ref*{remdual}} says that we have a $\GaK$-linear isomorphism
	\[ \MAK(A_m)^{\vee} \cong \MAK((A_m)^{\vee}(\chLT)).\]
	Combining these results gives us the $\GaK$-linear isomorphism
	\[ \CD(\MAK(A_m))\cong \MAK((A_m)^{\vee}(\chLT)).\]
	Applying limits now gives the desired result.
	\end{proof}
	\end{lem}



	\begin{prop} \label{iwpsi}
	Let $T \in \RepGKOLfg$ and set
	\[ \Ccb_{\psi}(\MAK(T(\tau^{-1}))) \coloneqq \Ccb_{\CD(\varphi)}(\CD(\MAK(T^{\vee}(1)))[-1].\]
	Then the inclusion of complexes
	\[ \begin{xy} \xymatrix@R-1pc{
	\Ccb_{\psi}(\MAK(T(\tau^{-1}))) \ar@{^{(}->}[r] &
		 \Ccb_{\ODK(\varphi)}(\ODK(\MAK(T^{\vee}(1))))[-1] \ar@{=}[d]\\
		&\ODK(\Ccb_{\varphi}(\MAK(T^{\vee}(1))))[-2]
	}\end{xy} \]
	is a quasi isomorphism. So in particular we have an isomorphism in the derived
	category $\DDb(\LamK-{\mathbf{Mod}})$
	\[ \RGa(\Ccb_{\psi}(\MAK(T(\tau^{-1})))) \cong \RGaiwb(\Kin|K,T).\]
	\begin{proof}
	With $T^{\vee}(1)=T(-1)^{\vee}$, the above \hyperref[incdual2]{Lemma \ref*{incdual2}} and
	\hyperref[incdual3]{Lemma \ref*{incdual3}} imply
	\[ \begin{xy} \xymatrix{
	\MAK(T(\tau^{-1}) \ar[r]^-{\cong} & \CD(\MAK(T^{\vee}(1)) \ar@{^{(}->}[r] & \ODK(\MAK(T^{\vee}(1)).
	} \end{xy} \]
	The cited lemmata also show that both homomorphisms are $\GaK$-linear.
	Let \linebreak $\mathcal{M} \coloneqq \MAK(T^{\vee}(1))$ then \hyperref[secdual]{Proposition \ref*{secdual}}
	together with \hyperref[loctat13]{Theorem \ref*{loctat13}} implies the commutative diagram
	with exact rows and $\LamK$-linear vertical homomorphisms
	\[ \begin{xy} \xymatrix@C-1.1pc@R-0.3pc{
		0 \ar[r] & \HIW^1(\Kin|K,T) \ar[r] & \ODK\left(\mathcal{M}\right) \ar[rr]^{\ODK(\varphi)-\id} & &
		\ODK\left(\mathcal{M}\right) \ar[r] & \HIW^2(\Kin|K,T) \ar[r] &0 \\
		0 \ar[r] & \HIW^1(\Kin|K,T) \ar[r] \ar@{=}[u]& \MAK(T(\tau^{-1})) \ar@{^{(}->}[u]\ar[rr]^{\psi-\id} & &
		\MAK(T(\tau^{-1})) \ar[r] \ar@{^{(}->}[u] & \HIW^2(\Kin|K,T) \ar[r] \ar@{=}[u] &0.
	} \end{xy} \]
	This gives the desired quasi isomorphism. The second statement then follows from
	\hyperref[iwdcts]{Lemma \ref*{iwdcts}} by using \hyperref[qind]{Proposition \ref*{qind}}.
	\end{proof}
	\end{prop}
		
\begin{qes}
It follows that, for $A$ cofinitely generated over $\mathcal{O}_L$ and with continuous $G_K$-action, the complex
\[\xymatrix@C=0.5cm{
  0 \ar[r] & \ODK(\MAK(A))/\CD(\MAK(A)) \ar[rr]^{\ODK(\varphi)-\id } && \ODK(\MAK(A))/\CD(\MAK(A)) \ar[r] & 0 }\] is acyclic, in particular for   $ \CD(\MAK(A))= \AAK.$ Can one show this directly, without going the intricate way using Matlis duality and the Nekovar's results? Moreover, is it realistically conceivable that even  $\ODK(\MAK(A))=\CD(\MAK(A)) $  holds?
\end{qes}
		
	\begin{thm} \label{thmpsicts}
	Let $T \in \RepGKOLfg$ and let
	$K \subseteq K' \subseteq \Kin$ an intermediate field, finite over $K$, such that
	$\GaKp\coloneqq \Gal(\Kin|K')$ is isomorphic to some $\Zp^r$. Then we have an
	isomorphism in the derived category $\DDp(\OL\text{-}{\mathbf{Mod}})$
	\[ \RGaiwb(\Kin|K,T) \tender_{\Lambda_{K'}} \OL
		\cong \RGactsb(G_{K'},T).\]
	In particular, we have
	\[ \RGa(\Ccb_{\psi}(\mathcal{D}_{K|L}(T(\tau^{-1})))
		\tender_{\Lambda_{K'}} \OL
		\cong \RGactsb(G_{K'},T).\]
	\begin{proof}
	The first assertion is 
	\cite[(8.4.8.1) Proposition, p.\,168]{neksc}. Note that we have an isomorphism
	$\RGaiwb(\Kin|K',T) \cong \RGaiwb(\Kin|K,T)$ in
	$\DDp(\Lambda_{K'}\text{-}\mathbf{Mod})$ since the intermediate fields of $\Kin|K'$ are
	cofinal in the intermediate fields of $\Kin|K$.
	The second assertion then is an application of
	\hyperref[iwpsi]{Proposition \ref*{iwpsi}}.
	\end{proof}
	\end{thm}
	
	 Using \cite[Prop.\ 1.6.5 (3)]{fukaya-kato} we obtain the following variant.
\begin{thm}
Let $T \in \RepGKOLfg$ and let
	$K \subseteq K' \subseteq \Kin$ any intermediate field, finite over $K$. Then we have an
	isomorphism in the derived category $\DDp(\OL\text{-}{\mathbf{Mod}})$
\[ \RGaiw(\Kin|K,T) \tender_{\Lambda_K} \OL[\Gal(K'|K)]  \cong \RGacts(G_{K'},T)\] in particular  \[ \RGa(\Ccb_{\psi}(\mathcal{D}_{K|L}(T(\tau^{-1})))
		\tender_{\Lambda_K} \OL   \cong \RGacts(G_{K},T).\]
	\end{thm}
\begin{proof} We have the following isomorphisms
 \begin{align*}
 \RGacts(G_K,\FCK(T)) \tender_{\Lambda_K} \OL[\Gal(K'|K)]&\cong
	\RGacts(G_K,\Lambda_K^{\iota} \otimes_{\OL} T) \tender_{\Lambda_K} \OL[\Gal(K'|K)]\\
		& \cong  \RGacts(G_K, \OL[\Gal(K'|K)] \otimes_{\Lambda_K} \left(\Lambda_K^{\iota} \otimes_{\OL} T\right)) \\
		& \cong \RGacts(G_K, \OL[\Gal(K'|K)]^{\iota} \otimes_{\OL} T) \\
		& \cong \RGacts(G_{K'},T)
	\end{align*}
where the first isomorphisms comes from Remark \ref{FCKLam}, the second one from (loc.\ cit.), the third one is trivial while the last one is Shapiro's Lemma.
\end{proof}

	\begin{rem}
	We want to give a more concrete statement of the above \emph{\hyperref[thmpsicts]{Theorem \ref*{thmpsicts}}}.
	So let as there $T \in \RepGKOLfg$ and
	$K \subseteq K' \subseteq \Kin$ an intermediate field, finite over $K$, such that
	$\GaKp\coloneqq \Gal(\Kin|K')$ is isomorphic to some $\Zp^r$. Let furthermore $\gamma_1,\dots,\gamma_r$
	be a set of generators of $\GaKp$. The {\bfseries \emph{Koszul-complex}} $K_{\bullet}(\LamKp)$
	of $\LamKp$ then is the complex
	\[ \begin{xy} \xymatrix{
	0 \ar[r] & \bigwedge^r \LamKp  \ar[r]^-{\mrmd_r} & \bigwedge^{r-1} \LamKp \ar[r]^-{\mrmd_{r-1}}
	& \cdots \ar[r] & \LamKp \ar[r]^-{\mrmd_1} & \OL \ar[r] & 0 ,
	} \end{xy} \]
	where $\bigwedge^{i} \LamKp$ denotes the $i$-th exterior algebra of $\LamKp$ and
	\[ d_i(x_1 \wedge \cdots \wedge x_i) = \sum_{j=1}^{i} (-1)^{j+1} \pr(x_j) x_1\wedge \cdots \wedge \widehat{x_j}
	\wedge \cdots \wedge  x_i.\]
	Here $\widehat{(-)}$ denotes that this entry is omitted and $\pr$ denotes the projection \linebreak
	$\LamKp \twoheadrightarrow \LamKp/(\gamma_1-1,\dots,\gamma_r-1) \cong \OL$
	(cf.\ \emph{\cite[Section 15.28]{stacks-project}}). Under the
	(uncanonical) isomorphism $\LamKp \to \OL \llbracket X_1,\dots,X_r \rrbracket, \ \gamma_i-1 \mapsto X_i$
	the above projection becomes the projection to degree zero. Then by
	\emph{\cite[Theorem 16.5, p.\,128--129]{mats1}} the Koszul-complex $K_{\bullet}(\LamKp)$
	of $\LamKp$ is a free resolution of
	$\OL$ and therefore (cf.\ \emph{\cite[Section 15.57, Definition 15.57.15]{stacks-project}})
	$\RGa(\Ccb_{\psi}(\mathcal{D}_{K|L}(T(\tau^{-1}))) \tender_{\Lambda_{K'}} \OL$ is represented by the complex
	\[ \begin{xy} \xymatrix{
	(\Ccb_{\psi}(\mathcal{D}_{K|L}(T(\tau^{-1}))) \otimes_{\LamKp} K_{\bullet}(\LamKp)
	} \end{xy} \]
	which then is isomorphic to the complex
	\[ \begin{xy} \xymatrix@R-1.5pc{
		\Tot\Big(\mathcal{D}_{K|L}(T(\tau^{-1})) \otimes_{\LamKp} K_{\bullet}(\LamKp) \ar[rr]^-{(\psi-\id)\otimes \id}
		 & &\mathcal{D}_{K|L}(T(\tau^{-1})) \otimes_{\LamKp} K_{\bullet}(\LamKp)\Big) \cong\\
		\Tot\Big(K_{\bullet}(\mathcal{D}_{K|L}(T(\tau^{-1}))) \ar[rr]^{K_{\bullet}(\psi)-\id} & &
		K_{\bullet}(\mathcal{D}_{K|L}(T(\tau^{-1})))\Big).
	} \end{xy} \]
	Here $K_{\bullet}(\mathcal{D}_{K'|L}(T(\tau^{-1})))$ denotes the Koszul-complex of
	$\mathcal{D}_{K'|L}(T(\tau^{-1}))$ which is defined in
	an analogous way to the Koszul-complex of $\LamKp$. This last complex then is the generalization of
	the $\psi$-Herr complex from the classical theory.
	\end{rem}

Using the self-duality of the Koszul-complex and an inspection of the complex in the last remark compared to the Pontrjagin dual of the (by local Tate-duality) corresponding $\varphi$-Herr-complex, one can indeed derive now that the differentials in the original $\varphi$-Herr-complex are strict with closed image (at least for finitely generated torsion coefficients). Indeed, its dual complex  has the right cohomology groups, namely the duals of the cohomology groups of the original $\varphi$ by the above results. It would be desirable to show these topological properties directly in order to get a genuine theory within the world of $(\varphi,\Gamma)$-modules.

	
%
	
%



\bibliographystyle{amsplain}
\bibliography{book}{}
\end{document}